\documentclass[11pt,a4paper]{article}

\usepackage{graphicx}
\usepackage{soul}
\usepackage{amssymb}
\usepackage{tikz}
\usetikzlibrary{calc,trees,positioning,arrows,chains,shapes.geometric,%
    decorations.pathreplacing,decorations.pathmorphing,shapes,%
    matrix,shapes.symbols}
\tikzset{
>=stealth',
  punktchain/.style={
    rectangle,  
     fill=cyan!40,
    draw=black, very thick,
    text width=12em, 
    minimum height=2em, 
    text centered, 
    on chain},
  line/.style={draw, thick, <-},
  element/.style={
    tape,
    top color=white,
    bottom color=blue!50!black!60!,
    minimum width=8em,
    draw=blue!40!black!90, very thick,
    text width=10em, 
    minimum height=2.5em, 
    text centered, 
    on chain},
  every join/.style={->, thick,shorten >=1pt},
  decoration={brace},
  tuborg/.style={decorate},
  tubnode/.style={midway, right=2pt},
}
\usepackage{standalone}
\usepackage{epstopdf}
\usepackage[english]{babel}
\usepackage{latexsym}
\usepackage{subfigure}
\usepackage{color}
\usepackage[colorlinks=true,linkcolor=blue,citecolor=red]{hyperref}
\usepackage{amsmath,amsfonts,amsthm}
\usepackage[normalem]{ulem}

\newtheorem{rem}{Remark}

\newtheorem{thm}{Theorem}

\newtheorem{alg}{\small{Algorithm}}

\def\be{\begin{equation}}
\def\ee{\end{equation}}
\def\bea{\begin{eqnarray}}
\def\eea{\end{eqnarray}}

\def\RR{\mathbb R}

\def\IT{I_{\theta}}

\begin{document}

\title{Monte Carlo gPC methods for diffusive kinetic flocking models with uncertainties}

\author{Jos\'e Antonio Carrillo \\
{\small Department of Mathematics}\\
{\small  Imperial College London, SW7 2AZ, UK}\\
{\small\tt carrillo@imperial.ac.uk}\and
%\thanks{Department of Mathematics, Imperial College London, SW7 2AZ, UK ({\tt carrillo@imperial.ac.uk})} \and 
%Lorenzo Pareschi \thanks{Department of Mathematics and Computer Science, University of Ferrara, Via Machiavelli 35, 44121 Ferrara, Italy ({\tt lorenzo.pareschi@unife.it}).} \and
Mattia Zanella \\
{\small Department of Mathematical Sciences ``G. L. Lagrange''} \\
		{\small Politecnico di Torino, Torino, Italy} \\
		{\small\tt mattia.zanella@polito.it}}

%%%% maketitle %%%%%
\date{}
\maketitle
\begin{abstract}
In this paper we introduce and discuss numerical schemes for the approximation of kinetic equations for flocking behavior with phase transitions that incorporate  uncertain quantities. This class of schemes here considered make use of a Monte Carlo approach in the phase space coupled with a stochastic Galerkin expansion in the random space. The proposed methods naturally preserve the positivity of the statistical moments of the solution and are capable to achieve high accuracy in the random space. Several tests on a kinetic alignment model with self propulsion validate the proposed methods both in the homogeneous and inhomogeneous setting, shading light on the influence of uncertainties in phase transition phenomena driven by noise such as their smoothing and confidence band. 

\medskip


\end{abstract}
%\tableofcontents

\section{Introduction}
\label{sec1}

Uncertainty quantification (UQ) for partial differential equations describing real world phenomena gained in recent years lot of momentum in various communities. Without intending to review the very huge literature on this topic we mention \cite{APZa,CPZ,DPZ,HJ,PZ1,TZ,TZ2,X,ZL,ZJ} and the references therein. One of the main advantages of this approach relies in its capability to provide a sound mathematical framework to reproduce realistic experiments through the introduction of the stochastic quantities, reflecting our  incomplete information on some features on the systems' modelling.  This argument is especially true for the description of emergent social structures in interacting agents' systems in socio-economic and life sciences. Common examples are the emergence of consensus phenomena in opinion dynamics, flocking and milling patterns in swarming of animals or humans and the formation of stable wealth distributions in economic systems, see \cite{PT}. It is worth observing how for these models we can have at most statistical information on initial conditions and on the modeling parameters, which are in practice substituted by empirical social forces embedding a huge variability, see for example \cite{Ball_etal}.

From a mathematical viewpoint, the kinetic equations we are interested in are nonlinear Vlasov-Fokker-Planck equations depending on random inputs taking into account uncertainties in the interaction terms or in the boundary conditions. These equations arise in the description of collective phenomena modeling the evolution of a distribution function $f=f(\theta,x,v,t)$, $t\ge 0$, $x\in\RR^{d_x}$, $v\in\RR^{d_v}$, $d_x,d_v\ge 1$, and $\theta\in I_{\theta}\subseteq\RR$ a random field, according to
\be\begin{split}\label{eq:MF_general}
&\partial_t f+v\cdot\nabla_x f = \nabla_v \cdot \left[ \mathcal B[f]f+\nabla_v (Df) \right],
\end{split}\ee
where $\mathcal B[\cdot]$ is a non--local operator of the form 
\begin{equation}
\label{eq:Bf}
\mathcal B[f](\theta,x,w,t) =  S(\theta,v) + \int_{\RR^{d_x}}\int_{\RR^{d_v}}P(\theta,x,x_*)(v-v_*)f(\theta,x_*,v_*,t)dv_*dx_*,
\end{equation}
being $P\ge0$ for all $v,v_*\in\RR^{d_v}$ and $D=D(\theta)\ge 0$ is an uncertain constant diffusion that depends on the introduced uncertainty.

As a follow-up question to the progress of the analytical understanding of real phenomena, the existence of phase transitions driven by noise represents a deeply fascinating issue, see \cite{CGPS,CP,DFL,DCBC,GPY,GP} and the references therein. The notion of phase transition has been fruitfully borrowed from thermodynamics to highlight the phase change of the system under specific stimuli. In particular, it is of interest the emergence of patterns in the collective dynamics for critical strengths of noise as in the classical Kuramoto model \cite{BGP10,CCP,Kur81,SSK88} or in collective behavior \cite{BCCD,GPY}. In the present paper we concentrate on a flocking model for interacting agents with self-propulsion and diffusion where the parameters are assumed to be stochastic. This model has been recently investigated in \cite{ABCD,BCCD,BD,BC} in absence of uncertainties. The introduction of uncertain quantities points in the direction of a more realistic description of the underlying processes and helps us to compute possible deviations from the prescribed deterministic behavior. 

Suitable numerical methods that preserve the positivity of the distribution function are developed and are based on the so-called of Monte Carlo generalized polynomial chaos (MCgPC) methods. The introduced class of schemes is based on a Monte Carlo approach in the phase space that is coupled with stochastic Galerkin decomposition in the random space \cite{HJ,HJX,PT,X}. This method has been recently proposed in \cite{CPZ} in the case of zero diffusion and it will be extended for models incorporating noise in the present setting. Furthermore, beside the natural positivity preservation, the resulting methods are spectrally accurate in the random space for sufficiently regular uncertainties.  Furthermore, the adoption of a spectrally accurate methods in the random space is particularly efficient if compared with other non-intrusive approaches. 

The rest of this paper has the following structure: in Section \ref{sect:flocking} we introduce in detail the class of kinetic flocking models of interest, the issue of phase transition will be treated with a particular focus on the relation between self-propulsion strength and noise intensity in the case they both depend on uncertain quantities. Section \ref{sect:MCgPC} focuses on the construction of Monte Carlo generalized polynomial chaos methods, a reduction in the computational complexity is here achieved through a Monte Carlo mean-field algorithm discussed in previous works. Finally, Section \ref{sect:num} is devoted to numerical test for the validation of the proposed schemes. Here continuous and Monte Carlo schemes will be compared to show the effectiveness of the MCgPC approach.

%%%%%%%%%%%%%%%%%%%%%%%%%%%%%%%%%%%%%%
\section{Phase transitions in kinetic flocking models with uncertainties}\label{sect:flocking}

Kinetic models for aggregation-diffusion dynamics encountered extensive investigation in recent years \cite{BCCD,BD,BCM,CCH,CFRT,CFTV,DM,DFT}. This class of models describes the aggregate behaviour of large systems of self-propelled particles for which stable patterns emerge asymptotically. Nevertheless, at the present times, the study of the influence of ineradicable uncertainties in the modeling parameters is quite an unexplored area. We mention in this direction \cite{APZa,CPZ,HaJin,HJJ,TZ}. 

Let us consider first a kinetic flocking model for self-propelled particles with both state-independent interactions and uncertain diffusion given by \eqref{eq:MF_general} with the following choice
\[
P(\theta,x,x_*)  = \delta(x-x_*), \qquad S(\theta,v) = \alpha(\theta) (|v|^2-1)v, \qquad D= D(\theta),
\]
being $\delta(x-x_*)$ the Dirac delta distribution centered in $x_*\in\mathbb R^{d_x}$. The model of interest is now characterized by a localised Vlasov-Fokker-Planck equation with random inputs for the dynamics of  $f = f(\theta,x,v,t)$ and having the following form 
\be\label{eq:kinetic_CSM}
\partial_tf + v\cdot\nabla_x f = \nabla_v \cdot \left[ \alpha(\theta)(|v|^2-1)vf + \rho_f(v-u_f)f + D(\theta)\nabla_v f \right],
\ee
where $\alpha(\theta)\ge0$, $D(\theta)\ge0$ are respectively self--propulsion strength and intensity of the diffusion operator, and where $u_f$ is the momentum of the system which is not conserved in time due to the presence of the self-propulsion term
\[
\rho_f(\theta,x,t) = \int_{ \RR^{d_v}} f(\theta,x,v,t)dv,\qquad \rho_f (\theta,x,t)u_f(\theta,x,t) = \int_{\RR^{d_v} }v f(\theta,x,v,t)dv. 
\]
It seems worth stressing that, in the present setting, the role of the additive diffusion operator $\Delta_v f$ in \eqref{eq:MF_general}, representing the impact of unpredictable events, is strongly different from that of the uncertainties introduced in the definition of the model parameters. The velocity diffusion term is, in fact, representative of all the possible modifications of the dynamics not modeled in a structural way, whereas $\theta\in I_{\theta}$ summarizes all source of modifications in the \emph{prescribed dynamics}. We point the reader to \cite{TZ} for a more focussed discussion in a related setting.

In the following we will treat separately the case space-homogeneous case and then we will provide some insights on the non-localised inhomogeneous setting. 

\subsection{Space homogeneous problem}\label{sect:hom}
In \cite{BCCD} the authors investigated the space homogeneous version of \eqref{eq:kinetic_CSM} in the deterministic setting. It has been proven that a phase transition between the so-called polarized and unpolarized motion takes place as the noise intensity $D$ increases and for a specific range of the values of the self-propulsion strength $\alpha$. At a difference with the cited results, the present formulation of the model embeds from the very beginning the presence of uncertain quantities. 

Let us consider for simplicity the case of unitary mass, so that $f$ is a probability density for all times $t\ge 0$. First, we observe that \eqref{eq:kinetic_CSM} can be rewritten as a gradient flow. In fact, if we define 
\[
\xi (\theta,v,t)= \Phi(\theta,v) + (U*f)(\theta,v,t) + D(\theta)\log f(\theta,v,t),
\]
with $U(v)$ a interaction potential, given by $U(v)=\tfrac{|v|^2}{2}$, and $\Phi(v)$ a confining potential of the form
\[
\Phi(\theta,v) = \alpha(\theta) \left(\dfrac{|v|^4}{4}-\dfrac{|v|^2}{2} \right),
\]
the equation reads
\[
\partial_t f(\theta,v,t) = \nabla_v \cdot \left( f(\theta,v,t)\nabla_v \xi(\theta,v,t) \right). 
\]
A free energy functional which dissipates along solution is defined for all $\theta\in I_{\theta}$ as follows
\[
\begin{split}
\mathcal E(\theta,t) =& \int_{\RR^{d_v}} \left( \alpha(\theta)\dfrac{|v|^4}{4}+(1-\alpha(\theta))\dfrac{|v|^2}{2} \right) f(\theta,v,t)dv -\dfrac{1}{2}|u_f|^2\\
&+D(\theta)\int_{\RR^{d_v}}f(\theta,v,t)\log f(\theta,v,t)dv,
\end{split}
\]
with 
\[
u_f(\theta,t) = \int_{\RR^{d_v}}vf(\theta,v,t)dv.
\]
Stationary solutions of the space--homogeneous problem satisfy the identity $\nabla_v \xi(\theta,t)=0$ for all $\theta\in I_{\theta}$ and have the form 
\be\label{eq:ss_hom}
f^{\infty}(\theta,v) =C \exp\left\{ -\dfrac{1}{D(\theta)}\left[ \alpha(\theta)\dfrac{|v|^4}{4}+(1-\alpha(\theta))\dfrac{|v|^2}{2}-u_{f^{\infty}}\cdot v \right]\right\},
\ee
being $C>0$ a normalization factor. It is possible to prove the following result, see \cite[Theorem 2.1]{BCCD}.
\begin{thm}
The space--homogeneous nonlinear Fokker--Planck equation
\[
\partial_t f= \nabla_v \cdot \left[ \alpha(\theta)(|v|^2-1)vf + (v-u_f(\theta,t))f + D(\theta)\nabla_v f \right],
\]
exhibits a phase transition in the following sense:
\begin{itemize}
\item For small enough diffusion $D(\theta)\rightarrow 0$ for all $\theta\in I_{\theta}$  there is a function $u=u(D(\theta))$ with $\lim_{D\rightarrow 0}u(D(\theta)) = 1$, such that $f^{\infty}(\theta,v)$ with $u=(u(D(\theta)),0,\dots,0)$ is a stationary solution of the original problem. 
\item For large enough diffusion coefficient $D$ the only stationary solution is the symmetric distribution given by \eqref{eq:ss_hom} with $u_{f}\equiv 0$. 
\end{itemize}
\end{thm}

This behavior is reminiscent of the one observed for the Vicsek model \cite{BCCD,DFL}, where agents move with a constant speed and interact with their neighbors via a local alignment force and are subject to noise action. In particular, a critical noise intensity has been discovered whose value determines a phase transition between the ordered states and a chaotic state characterized by a null asymptotic average velocity of the system of agents. The critical noise value is theoretically demonstrated in \cite{DFL} for the Vicsek model while for the present model there is a strong numerical evidence for its existence \cite{BCCD}. 

\subsection{Space inhomogeneous case}

In the non-localized inhomogeneous setting sharp results on analogous phase transitions are sill non present in the literature due to the additional difficulties of the models. In the zero diffusion limit and without self-propulsion forces, small variations of the arguments in \cite{CFRT} prove that for a Cucker-Smale type interaction
\be\label{eq:P_CS}
P(\theta,x,x_*) = \dfrac{H(\theta)}{\left( 1+ |x-x_*|^2\right)^{\gamma}}, \qquad H\ge 0,
\ee
unconditional alignment, i.e. convergence to a profile travelling with the same mean velocity, emerges in the case $\gamma\le 1/2$ without conditions on the initial configuration of the system or on its dimensionality both at the kinetic and particle level \cite{CFRT,CS}. In this case the asymptotic distribution is a Dirac delta in the velocity space, meaning that for large times the agents share the same velocity and that they form a group with frozen mutual distances. At the particle level, the case $\gamma>1/2$ leads to a conditional flocking, i.e. alignment emerges for provable initial configurations of the system \cite{CS}.  Initial density conditions to achieve asymptotic alignment for $\gamma>1/2$ are still not proven,  in general we may expect flocking for high mean-density distributions \cite{CFTV}. 

In the following we will present a numerical insight for the case of a VFP model with Cucker-Smale type interaction forces \eqref{eq:P_CS} in dimension $d_x=1$ and $d_v=1$, self-propulsion forces and uncertain diffusion coefficient. In particular, numerical results obtained with accurate numerical schemes highlight that a phase transition can occur also in this regime, shedding light on the deep interplay between alignment forces and noise strength in more general settings with not localized alignment \cite{CS,MT}.

\section{Stochastic Galerkin methods for  kinetic equations}\label{sect:MCgPC}
We introduce Stochastic Galerkin (SG) numerical methods with applications to the  nonlinear Vlasov-Fokker-Planck (VFP) equation \eqref{eq:MF_general}. We discuss the class of stochastic Galerkin (SG) methods and, in particular, we concentrate on the  generalized Polynomial Chaos (gPC) decomposition  \cite{DPZ,X,XK,ZJ}. These methods gained increased popularity in recent years since they guarantees spectral accuracy in the random space under suitable regularity conditions. 

In our schemes, we exploit Monte Carlo methods for the approximation of the numerical solution of the (VFP) in the phase space taking advantage of particle based reformulation of the problem that converges in distribution to the solution for an increasing number of agents. The core idea then is to apply SG-gPC techniques for the efficient approximation of the random field in the resulting MC approximation. We highlight that the combination of the two approaches leads to a positive approximation of statistical moments of the solution even in the nonlinear case. 

In the following we will derive a SG-gPC scheme for the continuous problem and we will consider specific formulation for the so-called Monte Carlo gPC methods (MCgPC), see  \cite{CPZ}. 

\subsection{Preliminaries on SG-gPC expansion}

In this section we derive a SG-gPC approximation for the uncertain nonlinear VFP equation \eqref{eq:MF_general} with nonlocal drift $\mathcal B[\cdot]$ having the structure \eqref{eq:Bf}. 
Let $(\Omega,F,P)$ be a probability space where as usual $\Omega$ is the sample space, $F$ is a $\sigma-$algebra and $P$ a probability measure, and let us define the random variable 
\[
\theta: (\Omega,F) \rightarrow (I_{\theta},\mathcal{B}_{I_\theta}),
\]
with $I_\theta\subseteq \RR$ and $\mathcal{B}_{I_\theta}$ the Borel set, and with known probability density function $\Psi(\theta):I_\theta\rightarrow \mathbb R^+$. Hence, we consider the linear space $\mathbb P_M$ of polynomials of degree up to $M$, which is generated by a family of orthonormal polynomials $\{\Phi_h(\theta)\}_{h=0}^M$ such that
\[
\mathbb{E}\left[ \Phi_h(\theta)\Phi_k(\theta) \right] = \int_{I_{\theta}} \Phi_h(\theta)\Phi_k(\theta)\Psi(\theta)d\theta = \delta_{hk},
 \]
being $\delta_{hk}$ the Kronecker delta function. Assuming that $\Psi(\theta)$ has finite second order moment, we can approximate the distribution $f \in L^2(\Omega,\mathcal F,P)$ in terms of the following chaos expansion
\[
f(\theta,x,v,t) \approx f^M(\theta,x,v,t) = \sum_{h=0}^M \hat f_h(x,v,t)\Phi_h(\theta), 
\]
being $\hat f_h(x,v,t) = \mathbb E\left[ f(\theta,x,v,t)\Phi_h(\theta) \right]$ for all $h=0,\dots,M$. The SG-gPC formulation for the approximated distribution $ f^M = f^M(\theta,x,v,t)$ solution of \eqref{eq:MF_general} is then 
\[
\partial_t f^M +v\cdot \nabla_v f^M = \nabla_v \cdot \left[ \mathcal B[f^M]f^M + \nabla_v \left(D f^M \right) \right]. 
\]
Thanks to the orthogonality of the polynomial basis of $\mathbb P_M$ we obtain a coupled system of $M+1$ purely deterministic PDEs for the evolution of each projection $\hat f_h(x,v,t)$, $h = 0,\dots,M$ that reads
\be\label{eq:gPC_general}
\begin{split}
& \partial_t \hat f_h(x,v,t)  + v\cdot \nabla_x \hat f_h(x,v,t) = \\
& \qquad \nabla_v \cdot \left[ \sum_{k=0}^M S_{hk} \hat f_k(x,v,t) + \sum_{k=0}^M P_{hk}[f^M]\hat f_k(x,v,t) +  \sum_{k=0}^M D_{hk} \nabla_v \hat f_k(x,v,t) \right],
\end{split}\ee
having defined the following matrices for $h,k = 0,\dots,M$
\begin{align}
& S_{hk}(v) = \int_{I_{\theta}} S(\theta,v) \Phi_h(\theta)\Phi_k(\theta)\Psi(\theta)d\theta,  \label{eq:Shk}\\
& P_{hk}[ \hat f] =  \int_{I_{\theta}} \int_{\RR^{d_x\times d_v}}P(\theta,x,x_*)(v-v_*)f^M(\theta,x_*,v_*,t)  \Phi_h(\theta)\Phi_k(\theta)\Psi(\theta) dv_*\, dx_* d\theta \label{eq:Phk} \\
& D_{hk} =  \int_{I_{\theta}} D(\theta)\Phi_h(\theta)\Phi_k(\theta)\Psi(\theta)d\theta. \label{eq:Dhk}
\end{align}
The statistical quantities of interest may be defined in terms of the coefficients of the expansion. In particular we have
\[
\mathbb E[f(\theta,x,v,t)] \approx \hat f_0(x,v,t), \qquad Var[f(\theta,x,v,t)] \approx \sum_{h=0}^M \hat f_h(x,v,t) \mathbb E[\Phi_h^2(\theta)] - \hat f_0^2(x,v,t). 
\]

We introduce the vector $\hat{\textbf{f}} = \left(\hat f_0,\dots,\hat f_M \right)$ and the $(M+1)\times (M+1)$ matrices $\textbf{S} = \{S_{hk}\}_{h,k=0}^M$, $\textbf{P}[\hat{\textbf{f}}]=\{P_{hk}\}_{h,k=0}^M$, $\textbf{D} = \{D_{hk}\}_{h,k=0}^M$, whose components are given by \eqref{eq:Shk}-\eqref{eq:Dhk}, we can reformulate \eqref{eq:gPC_general} in more compact form as follows
\[
\partial_t \hat{\textbf{f}} + v\cdot \nabla_x \hat{\textbf{f}} = \nabla_v \cdot \left[ ( \textbf{S} + \textbf{P}[\hat{\textbf{f}}]) \hat{\textbf{f}} + \textbf{D} \nabla_v  \hat{\textbf{f}} \right]. 
\]

In the following we indicate with $\| \hat{\textbf{f}} \|_{L^2}$ the standard $L^2$ norm of the vector $\hat{\textbf{f}}(x,v,t)$
\[
\| \hat{\textbf{f}} \|_{L^2} := \left[ \int_{\RR^{d_x}\times \RR^{d_v}} \left(\sum_{h=0}^M \hat f_h^2(x,v,t) \right)  dv\,dx \right]^{1/2} . 
\]
We can easily observe that since the norm of $f^M(\theta,x,v,t)$ in $L^2(\Omega)$ is
\[
\begin{split}
\| f^M \|_{L^2(\Omega)} = \int_{I_{\theta}} \int_{\RR^{d_x}\times \RR^{d_v}} \left(\sum_{h = 0}^M \hat f_h \Phi_h(\theta) \right)^2 dv\,dx\,\Psi(\theta)d\theta. 
\end{split}
\]
thanks to the orthonormality of the basis $\{\Phi_k\}_{k=0}^M$ in $L^2(\Omega)$ it follows that  
\[
\| f^M\|_{L^2(\Omega)} = \|\hat{\textbf{f}} \|_{L^2}.
\]

We can prove the following stability result. For notation simplicity we consider the case $d_x = d_v = 1$. 

\begin{thm}
Assume that there exists $C_{\mathcal B},C_D>0$ such that $\|\partial_v( S_ {hk}+P_{hk}) \|_{L^{\infty}}\le C_{\mathcal B}$ and $D_{hk}\le C_D$ for all $h,k=0,\dots,M$, then
\[
\|\hat{\normalfont{\textbf{f}}} \|_{L^2}^2 \le e^{t(C_{\mathcal B}+2C_D)}\| \hat{\normalfont{\textbf{f}}}(0)\|^2_{L^2}. 
\]
\end{thm}

\begin{proof}
We multiply \eqref{eq:gPC_general} by $\hat f_h$ and integrate over $\RR^2$ to obtain
\[
\begin{split}
&\int_{\RR^2} \left[ \partial_t \left( \dfrac{1}{2}\hat f_h^2 \right) + v\partial_x \left(\dfrac{1}{2}\hat{f}_h^2 \right) \right]dv\,dx \\
&\qquad\qquad = \int_{\RR^2}\partial_v  \left[\sum_{k=0}^M (S_{hk}+P_{hk})\hat f_k + \partial_v \sum_{k=0}^M D_{hk}\hat f_k \right] \hat f_h dv\,dx.
\end{split}
\]
After integration by parts the transport term vanishes and the right hand term reads
\[
\begin{split}
&\sum_{k=0}^M \int_{\RR^2} \hat f_h \partial_v \left((S_{hk}+P_{hk} )\hat f_k\right)dv\,dx \\
&\qquad= \sum_{k=0}^M \int_{\RR^2}\left( \hat f_k \hat f_h \partial_v (S_{hk}+P_{hk}) +(S_{hk}+P_{hk}) \hat f_h \partial_v \hat f_k \right)dv\,dx \\
& \qquad = - \sum_{k=0}^M \int_{\RR^2} (S_{hk}+P_{hk}) \partial_v \left( \hat f_h \hat f_k \right)dv\,dx - \sum_{k=0}^M \int_{\RR^2} \hat f_k \partial_v \left((S_{hk}+P_{hk}) \hat f_h \right) dv\,dx,
\end{split}
\]
and thanks to the symmetry of $\textbf{S}$ and $\textbf{P}$ we have
\[\begin{split}
2 \sum_{h,k=0}^M \int_{\RR^2} \hat f_h \partial_v \left( (S_{hk}+P_{hk}) \hat f_k \right)dv\,dx & = - \sum_{h,k=0}^M \int_{\RR^2} (S_{hk}+P_{hk}) \partial_v (\hat f_k \hat f_h)dv\,dx \\
& = \sum_{h,k=0}^M \int_{\RR^2} \hat f_h \hat f_k \partial_v (S_{hk}+P_{hk}) dv\,dx.
\end{split}\]
Since $\| \partial_v (S_{hk}+P_{hk})\|_{L^\infty}\le C_{\mathcal B}$ we have
\[
\sum_{h,k=0}^M \int_{\RR^2} \hat f_h \partial_v \left( (S_{hk}+P_{hk}) \hat f_k\right)dv\,dx \le \dfrac{C_{\mathcal B}}{2} \sum_{h,k=0}^M \int_{\RR^2} |\hat f_h \hat f_k| dv\,dx,
\]
and from Cauchy-Schwarz we have
\[
\sum_{h,k=0}^M \int_{\RR^2} \hat f_h \partial_v \left( (S_{hk}+P_{hk}) \hat f_k\right)dv\,dx \le \dfrac{C_{\mathcal B}}{2} \| \hat{\textbf{f}} \|_{L^2}^2
\]
Furthermore, since $D_{hk}\le C_D$ we have
\[
\sum_{h,k=0}^M \int_{\RR^2} \partial_v \left( \partial_v D_{hk}\hat f_k \right)\hat f_hdv\,dx \le - C_D\sum_{h,k=0}^M \int_{\RR^2} |\partial_v \hat f_h \partial_v \hat f_k| dv\,dx,
\]
and we obtain 
\[
\begin{split}
\dfrac{1}{2}\partial_t \| \hat{\textbf{f}} \|^2_{L^2} & \le \dfrac{C_{\mathcal B}}{2}\| \hat{\textbf{f}}\|_{L^2}^2 - \|\partial_v \hat{\textbf{f}} \|_{L^2}^2 \le \left(\dfrac{C_{\mathcal B}}{2}+C_D \right) \| \hat{\textbf{f}} \|_{L^2}^2
\end{split}\]
and thanks to Gronwall's Lemma, we conclude. 
\end{proof}

We highlight how, as for classical spectral methods, the introduced SG decomposition leads to the loss of important physical properties of the solution like the positivity. A positivity preserving scheme is instead obtained if the matrix $\textbf{S} + \textbf{P} = \{S_{hk}+ P_{hk}\}_{h,k=0}^M$ and the matrix $\textbf{D} = \{D_{hk}\}_{h,k=0}^M$ are diagonal. This case corresponds to the situation where no uncertainties are present in self-propelling and diffusion terms and in case of interaction with a background distribution that does not encode any uncertainty. Under the introduced assumptions and with uncertain initial distribution of agents $f(\theta,x,v,0)= f_0(\theta,x,v)$ the SG-gPC decomposition of the original problem reads for all $h=0,\dots,M$
\[
\partial_t \hat f_h(x,v,t) + v\cdot \nabla_x \hat f_h(x,v,t) = \nabla_v \cdot \left[ (S_{hh}+P_{hh}) \hat f_h(x,v,t) + D_{hh} \nabla_v \hat f_h(x,v,t) \right]. 
\]
Therefore, we have to solve a set of $M+1$ decoupled VFP equations. In this case we can apply structure preserving type schemes to ensure the positivity of the statistical moments, accurate description of the large time behavior of each projection and entropy dissipation, see  \cite{BCH,CCH,PZ1,PZ2,Z}. 

To tackle efficiently the general fully nonlinear case in the following we will introduce a novel scheme that exploits the spectral convergence in the random space of SG-gCP methods and that naturally preserves the positivity of the numerical distribution. 

\subsection{Particle based SG-gPC formulation of kinetic equations}

As for classical spectral methods, the solution of the coupled SG system \eqref{eq:gPC_general} for $f^M$ looses its positivity and then a clear physical meaning. In order to overcome the difficulty recently the so-called Monte Carlo gPC (MCgPC) scheme has been proposed. These methods combine the advantages of a Monte Carlo approach for the approximation of $f$ in the phase space and they conserve spectral accuracy in the random space. We refer to \cite{CPZ} for the study a wide range of mean-field equations describing the emergence of patterns and ordered behavior in large interacting systems in the zero diffusion limit. 

The Monte Carlo (MC) method is a probabilistic particle method that describe the evolution of density functions by resorting to the computation of interactions in a finite set of randomly chosen particles of which it is known a priori the dynamics, see \cite{DP15,PT}. In MC methods the updated distribution function in the phase space is typically reconstructed from particles as post-processing. Several approaches are possible for the reconstruction step, which can estimated through a histogram, weighted integration rule methods or by a convolution of the empirical particle distribution with a suitable mollifier, see \cite{HE}. All the mentioned approximations preserve the positivity of the obtained numerical distribution function. Concerning the order of accuracy for MC methods, we have a convergence rate of the order $\mathcal O(\frac{1}{\sqrt{N}})$ where $N$ is the number of considered samples, see \cite{Caflisch}.

Therefore, in order to set-up a MC scheme for the general equation \eqref{eq:MF_general} we need to find a system of stochastic differential equations which converge in distribution to the correct solution of the problem. This problem has been tackled in a more general setting in \cite{BCC} and equation \eqref{eq:MF_general} can be derived from the following second order system of SDEs for $(x_i(\theta,t),v_i(\theta,t))\in \RR^{d_x}\times \RR^{d_v}$, $i=1,\dots,N$, with random inputs
\be\label{eq:systemSDE}
\begin{cases}\vspace{1em}
d x_i(\theta,t) = v_i(\theta,t) dt, \qquad \\
d v_i(\theta,t) = \Big(\alpha(\theta)(1-|v_i|^2)v_i + \dfrac{1}{N} \displaystyle\sum_{j=1}^N P(\theta,x_i,x_j) (v_j-v_i) \Big)dt + \sqrt{2D(\theta)}dW_i
\end{cases}
\ee
with initial positions and velocities $(x_i(\theta,0),v_i(\theta,0))$. In \eqref{eq:systemSDE} we denoted by $\{W_i\}_{i=1}^N$ a vector of $N$ independent Wiener processes and independent with the random varlable $\theta$. Notice that the random variable $\theta$ plays the role of a parameter, thus for each sampling from the distribution $\Psi(\theta)$, we can apply the mean-field results developed in \cite{BCC} to conclude that \eqref{eq:kinetic_CSM} is the mean field limit of \eqref{eq:systemSDE}.

It is worth to remark that in the case of vanishing self-propulsion, i.e. $\alpha(\theta)\equiv 0$, and for symmetric interactions, i.e. $P(\theta,x_i,x_j) = P(\theta,x_j,x_i)$, the above system of SDEs preserves the initial mean velocity of the system in the limit $N\rightarrow +\infty$. Indeed we have
 \[
 \begin{split}
 d \sum_{i =1}^N v_i(\theta,t) & = \dfrac{1}{N} \sum_{i,j=1}^N P(\theta,x_i,x_j)(v_j-v_i) + \sqrt{2D(\theta)} \sum_{i=1}^N dW_i 
 = \sqrt{2D(\theta)} d\sum_{i=1}^N W_i,
 \end{split}\]
to conclude due to a law of large numbers argument. In general, the mean velocity is not conserved due to the presence of  the self-propulsion forces and appropriate methods must be developed to catch the correct transient regimes. 

In the following we summarize the assumptions to ensure the mean-field convergence. Let us introduce the empirical measure associated to the particle system \eqref{eq:systemSDE} weighted by the distribution of the uncertainty $\Psi(\theta)$, i.e.
\[
f^{(N)}(\theta,x,v,t) = \left[\dfrac{1}{N} \sum_{i=1}^N \delta(x-x_i)\otimes \delta(v-v_i) \right]\Psi(\theta). 
\]
The transition to chaos for the particle system \eqref{eq:systemSDE} follows from a small variation of \cite[Theorem 1.1]{BCC} under non restrictive assumptions on the a kernel $K(x_i-x_j,v_i-v_j) = -P(\theta,|x_i-x_j|)(v_i-v_j)$ and the drift $S(\theta,v_i)$. The proof is based on the following argument: the $N$ interacting process $(x_i(\theta,t),v_i(\theta,t))_{i=1}^N$ behaves in the limit $N\rightarrow +\infty$ like the process defined by the set $(\bar{x}_i(\theta,t),\bar{v}_i(\theta,t))_{i=1}^N$ solution of the McKean-Vlasov equations
\be\label{eq:systemSDE_2}
\begin{cases}
&d \bar{x}_i(\theta,t)= \bar{v}_i(\theta,t) dt, \qquad \\
&d \bar{v}_i(\theta,t) = -S(\theta,\bar{v}_i)dt - K*f(\theta,\bar{x}_i,\bar{v}_i)dt + \sqrt{2D(\theta)}dW_i, \\
&(\bar{x}_i(\theta,0),\bar{v}_i(\theta,0)) = ({x}_i(\theta,0),{v}_i(\theta,0)),
\end{cases}
\ee
where $f$ is the probability law of $(\theta,\bar{x}_i(\theta,t),\bar{v}_i(\theta,t))$and $\{W_i\}_{i=1}^N$ are the Wiener processes characterizing \eqref{eq:systemSDE}. The processes are identically distributed and independent from the random variable $\theta$, and by the It{\^o} formula their law $f$ evolves according a VFP equation of type \eqref{eq:MF_general}. We summarize this in the following result representing a direct extension of Theorem 1.1 in  \cite{BCC} that holds for all $\theta \in \IT$. 
 
\begin{thm}[{\cite[Theorem 1.1]{BCC}}]\label{thm:meanfield}
Let $f_0(\theta,x,v)$ be a Borel probability measure and let $(x_i(\theta,0),v_i(\theta,0))$ for all $i=1,\dots,N$ be  $N$ independent variables for all $\theta\in\IT$. We assume that the drift and the antisymmetric kernel satisfy that there exist constants $A,L,p>0$ such that
\[\begin{split}
-(v-w)\cdot \left(S(v)-S(w) \right)&\le A|v-w|^2, \\
|K(x,v) - K(y,v)| &\le L \min\{|x-y|,1\} \,(1+|v|^p)
\end{split}\]
for all $x,y\in\RR^{d_x}$, $v,w\in\mathbb R^{d_v}$. If the particle system \eqref{eq:systemSDE} and the processes \eqref{eq:systemSDE_2} have global solutions on the finite time interval $[0,T]$ with initial data $({x}_i(\theta,0),{v}_i(\theta,0))$ such that
\[
\sup_{0\le t\le T} \sup_{\theta\in I_{\theta}}\left\{ \int_{\RR^{2d_x}\times \RR^{2d_v}}\!\!\!\!\!\!\!\!\!\!\!\!\!\!\!\!\!\!\!\!\!|K(x-y,v-w)|^2 df(\theta,x,v)df(\theta,y,w) + \int_{\RR^{d_x}\times \RR^{d_v}}\!\!\!\!\!\!\!\!\!\!\!\!\!\!\!\!\!\! (|x|^2+e^{a|v|^p})df(\theta,x,v) \right\}<\infty\,,
\]
with $f(\theta,\cdot,\cdot) = \textrm{law}(\bar x_i,\bar v_i)$ for $\theta \in I_{\theta}$, then there exists a constant $C >0$ such that
\be\label{eq:thm1}
\mathbb E_{\theta,W}\left[ |x_i-\bar x_i|^2 + |v_i-\bar v_i|^2  \right] \le \dfrac{C}{N^{e^{-Ct}}}.
\ee
If there exists $p^\prime>p$ such that
\be\label{eq:thm2}
\sup_{0\le t\le T} \sup_{\theta\in I_{\theta}} \int_{\RR^{d_x}\times \RR^{d_v}} e^{a|v|^{p^\prime}}df(\theta,x,v)<\infty
\ee
then for all $0<\epsilon<1$ there exists a constant $C>0$ such that 
\[
\mathbb E_{\theta,W}\left[ |x_i-\bar x_i|^2 + |v_i-\bar v_i|^2  \right] \le \dfrac{C}{N^{1-\epsilon}},
\]
for all $0\le t\le T$ and $N\ge 1$. 
\end{thm}

We have denoted with $\mathbb E_{\theta,W}[\cdot]$ the expectation taken with respect to the distribution of the Wiener processes and the random variable $\theta$. We highlight how the provided bounds hold for all $\theta\in \IT$ and in particular we take expectation with respect to the introduced stochastic parameter.
It is easily verified that for the introduced self-propulsion force and the kernels of interest the assumptions are met. Therefore the mentioned theorem gives the following quantitative result on the convergence of the empirical measure $f^{(N)}$ to the distribution $f$: let $\varphi$ a Lipschitz map on $\RR^{d_x}\times \RR^{d_v}$ then 
\begin{align}
& \mathbb E_{\theta,W}\left[\left| \dfrac{1}{N} \sum_{i=1}^N \varphi(x_i,v_i) - \int_{\RR^{d_x}\times \RR^{d_v}} \varphi(x,v) df(\theta,x,v) \right|^2 \right]  \nonumber\\
&\qquad \le 2 \mathbb E_{\theta,W}\left[| \varphi(x_i,v_i)- \varphi(\bar{x}_i,\bar{v}_i)|^2 + \left| \dfrac{1}{N} \sum_{i=1}^N \varphi(\bar x_i,\bar v_i) - \int_{\RR^{d_x}\times \RR^{d_v}} \varphi(x,v)df(\theta,x,v) \right|^2 \right] \nonumber \\
&\qquad \le \varepsilon(N) + C/N, \label{kkk}
\end{align}
where $\varepsilon(N) \rightarrow 0$ for $N\rightarrow + \infty$ with a rate $\varepsilon(N)$ defined in \eqref{eq:thm1}-\eqref{eq:thm2} depending on the assumptions on the solutions in Theorem \ref{thm:meanfield} and the law of large numbers to deal with the last term. These arguments are standard in interacting particle systems and analogous to \cite{BCC} and the references therein.

In the next section we derive a SG-gPC decomposition of \eqref{eq:systemSDE} so that we will preserve the exponential convergence in the random space with respect to all the uncertain quantities. 

\subsubsection{Stochastic Galerkin scheme for the particle system}\label{sect:SG_particle}

In the following we consider the gPC decomposition of the microscopic dynamics \eqref{eq:systemSDE}. Space and velocity variables of the $i$th agent, for all $i=1,\dots,N$, are approximated by $(x_i^M,v_i^M)$ where
\[
 x_i^M(\theta,t) = \sum_{k=0}^M \hat x_{i,k} \Phi_k(\theta),\qquad
 v_i^M(\theta,t) = \sum_{k=0}^M \hat v_{i,k} \Phi_k(\theta) ,
\]
being as before $\{\Phi_k(\theta)\}_{k=0}^M$ an orthonormal basis of $L^2(\Omega)$. 

Next we give explicit representation of the SG-gPC expansion at the particle level. To obtain the SG-gPC decomposition for the particle system we rewrite \eqref{eq:systemSDE} in terms of $x_i^M$, $v_i^M$ for all $i=1,\dots,N$, whose projection in the linear spaces of degree $h=0,\dots,M$ reads

\begin{equation*}
\begin{cases} \vspace{1em}
& \displaystyle d \int_{I_{\theta}}x_i^M \Phi_h(\theta)\Psi(\theta)d\theta = \int_{I_{\theta}} v_i^M \Phi_h(\theta)\Psi(\theta)d\theta dt \\
&\displaystyle d\int_{I_{\theta}}v_i^M \Phi_h(\theta)\Psi(\theta)d\theta = \\
&\qquad \displaystyle \int_{I_{\theta}} \left( \alpha(\theta)(1-|v_i^M|^2)v_i^M +\dfrac{1}{N} \sum_{j=1}^N P(\theta,x_i^M,x_j^M)(v_j^M-v_i^M) \right)\Phi_h(\theta)\Psi(\theta)d\theta dt  \\
&\qquad \displaystyle+ \int_{I_{\theta}}\sqrt{2D(\theta)}\Phi_h(\theta)\Psi(\theta)d\theta dW_i,
\end{cases}
\end{equation*}
and again thanks to the orthogonality of the polynomial basis we obtain for all $i=1,\dots,N$ the following coupled system of SDEs describing the evolution of each component of the original variables in the phase space
 \begin{equation}
 \begin{cases}\vspace{1em}
 &d \hat{x}_{i,h} = \hat{v}_{i,h} dt \\
 &d \hat{v}_{i,h} = \left(\displaystyle\sum_{k=0}^M s_{hk}(v_i^M) \hat{v}_{i,k} + \dfrac{1}{N} \displaystyle\sum_{j=1}^N \sum_{k=0}^M p_{hk}^{ij} (\hat{v}_{j,k}- \hat{v}_{i,k}) \right)dt + d_{h} dW_i,
 \end{cases}
 \label{eq:system_SDEh}
 \end{equation}
 where 
 \[\begin{split}
 s_{hk}(v_i^M) &= \int_{I_{\theta}} \alpha(\theta)(1-|v_i^M|^2)\Phi_k(\theta)\Phi_h(\theta)\Psi(\theta)d\theta \\
 p_{hk}^{ij}  &= \int_{I_{\theta}} P(\theta,x_i^M,x_j^M) \Phi_k(\theta)\Phi_h(\theta)\Psi(\theta)d\theta,
 \end{split}\]
 and 
 \[
 d_h = \int_{I_\theta} \sqrt{2D(\theta)}\Phi_h(\theta)\Psi(\theta)d\theta.
 \]
 
 It is worth to remark that in the case of vanishing self-propulsion $s_{hk}(v_i^M) = 0$ for all $h,k=0,\dots,M$ and for symmetric interactions $p_{hk}^{ij} = p_{hk}^{ji}$ we recover the conservation of the mean velocity of the system also in the SG decomposition.
 
The convergence of the SG-gPC expansion \eqref{eq:gPC_general} for sufficiently regular function follows from standard results in polynomial approximation theory, we recall for example \cite{Fun}. In general, thanks to the property of the introduced polynomial basis we have
\begin{thm}\label{thm:proj}
We consider $\theta\in I_\theta$ with distribution $\Psi(\theta)$ and the basis $\{\Phi_h\}_{h=0}^M$ of orthonormal polynomial basis in $L^2(\Omega)$ . Then, for any $g(\theta,t) \in L^2(\Omega)$
\[
\| g - g^M\|_{L^2(\Omega)} \rightarrow 0 \qquad \textrm{as}\quad M\rightarrow \infty,
\]
where $g^M = \sum_{k=0}^M \left( \int_{I_\theta} g(\theta,t) \Phi_k(\theta)d\Psi(\theta)\right)\Phi_k(\theta)$.
\end{thm}

In the present setting, thanks to Theorem \ref{thm:proj} it is reasonable to expect the convergence of the empirical measure
\[
f^{(N,M)}(\theta,x,v,t) = \left[\dfrac{1}{N} \sum_{i=1}^N \delta(x-x_i^M) \otimes \delta(v-v_i^M)\right] \Psi(\theta)
\]
to the empirical measure $f^{(N)}$ as $M\rightarrow \infty$ for all $t\ge0$ and for $(x_i^M,v_i^M)$ solution to \eqref{eq:system_SDEh}. Hence, from the mean-field convergence showed in Theorem \ref{thm:meanfield} we would guarantee that $f^{(N,M)}$ converges to $f(\theta,x,v,t)$ solution of the initial VFP problem by taking first the limit of the Galerkin modes $M\to\infty$ and then the limit of large number of particles $N\to\infty$. We can prove the following result following the footprint of \eqref{kkk} and Theorem \ref{thm:proj} applied to $x_i(\theta,t)$ and $v_i(\theta,t)$.

\begin{thm}
Let us define the following empirical measure
\[
f^{(N,M)}(\theta,x,v,t) = \left[\dfrac{1}{N} \sum_{i=1}^N \delta(x-x_i^M) \otimes \delta(v-v_i^M)\right] \Psi(\theta).
\]
Hence, provided that $x_i^M(\theta,t) \in L^2(\Omega)$, $v_i^M(\theta,t) \in L^2(\Omega)$ we have
\[
f^{(N,M)}(\theta,x,v,t) \rightarrow f^{(N)}(\theta,x,v,t),
\]
in $P(P(\Omega\times\RR^{d_x}\times \RR^{d_v}))$.
\end{thm}

\begin{proof}
To prove the convergence in $P(P(\Omega\times\RR^{d_x}\times \RR^{d_v}))$ we consider a sufficiently regular test function in all variables $(\theta,x,v)$, that we can assume of the form $\varphi_1(\theta)\varphi_2(x,v)$ without loss of generality, and we compute
\begin{equation*}\begin{split}
\int_{I_\theta}&\left| \dfrac{1}{N} \sum_{i=1}^N \varphi_2(x_i,v_i) -  \dfrac{1}{N} \sum_{i=1}^N \varphi_2(x_i^M,v_i^M) \right | \Psi(\theta) |\varphi_1(\theta)| d\theta  \le \\ 
& \qquad\qquad\qquad\qquad\qquad \dfrac{1}{N }\sum_{i=1}^N \int_{I_\theta}\left|\varphi(x_i,v_i)-\varphi(x_i^M,v_i^M)\right| \Psi(\theta) |\varphi_1(\theta)|d\theta
\end{split}\end{equation*}
providing the convergence result for $M\rightarrow +\infty$.
\end{proof}
In the last section we will give numerical evidence of this result.

\subsubsection{Monte Carlo gPC scheme}

We now approximate the limiting stochastic kinetic equation taking advantage of the particle reformulation of the problem. In fact, since the solution of the system of SDEs \eqref{eq:systemSDE} converges in distribution to the solution of the original problem  \eqref{eq:MF_general} for $N\rightarrow +\infty$, we can approximate the original dynamics by means of a Monte Carlo (MC) method in the phase space. The main drawback of this approach lies in the computational cost $\mathcal O(M^2N^2)$, since at each time step and for each gPC projection each agent modifies its velocity in a genuine nonlinear way. 

A significant reduction in terms of computational cost can be achieved through a mean field MC evaluation of the interaction dynamics as originally proposed in \cite{AlbiPareschi13}, see \cite{CPZ} for the UQ framework. Thanks to this approach we have an efficient algorithm for transport and interaction in the phase space and we can reconstruct the expected solution from the particle system from positions and velocities at the microscopic level, which is considered in the SG-gPC setting as in Section \ref{sect:SG_particle}. This approach has been recently analysed in connection to other problems in \cite{JLL}.
 \newpage
\begin{alg}[\textbf{MCgPC for nonlocal nonlinear VFP equations}]
\qquad\qquad\qquad\qquad
\begin{itemize}
\item[1.] Consider $N$ samples $(x_{i},v_{i})$ with $i=1,\dots,N$ from the initial $f_0(x,v)$, and fix $S\le N$ a positive integer;
\item[2.] Perform gPC expansion up to order $M\ge 0$ over the set of $S\le N$ particles to obtain the  projections $(\hat x_{i,h},\hat v_{i,h})$, $h =0,\dots,M$.
\item[] \textit{for $n=0$ to $T-1$}
\item[3.] Generate $N$ Brownian paths $\{\eta_i\}_{i=1}^N = \{W_i^{n+1}-W_i^n\}_{i=1}^N \sim \mathcal N(0,1)$, $\{\eta_i(0)\}_{i=1}^N = 0$
\begin{itemize}
\item [] \hspace{-0.7cm}\textit{for $i = 1$ to $N$}
\item [a)] sample $S$ particles $j_1,\dots,j_S$ uniformly without repetition among all particles;
\item[b)] compute the position and velocity change 
\[\begin{split}
\hat x_{i,h}^{n+1} &= \hat x_{i,h}^{n} + \hat v_{i,h}^{n} \Delta t \\
\hat v_{i,h}^{n+1} &= \hat v_{i,h} +\Delta t\sum_{k=0}^M s_{hk}^i v_{i,k}^n +  \dfrac{\Delta t}{S}\sum_{s=1}^S\sum_{k=0}^M p_{kh}^{i{j_s}}(\hat v_{j_s,k}-\hat v_{i,k}) + \sqrt{2d_h \Delta t}\eta_i
\end{split}\]
\item[]  \hspace{-0.7cm}\textit{end for}
\end{itemize}
\item[5.] Reconstruction $\mathbb E[f(\theta,x,v,n\Delta t)] = \int_{I_\theta} f(\theta,x,v,n\Delta t)\Psi(\theta)d\theta$ .% and $Var[f(x,v,\theta,n\Delta t)]$.
\item[]\textit{end for}
\end{itemize}
\label{alg:1}
\end{alg}

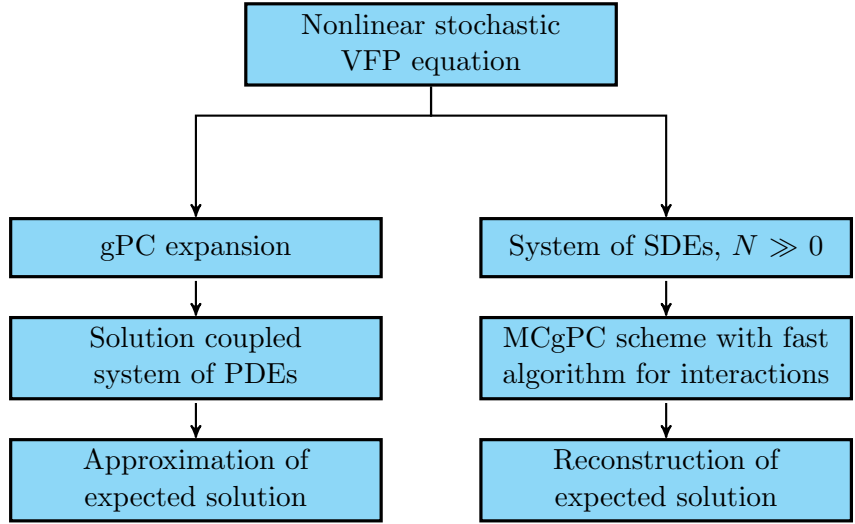
\begin{figure}[t]
\begin{center}
\begin{tikzpicture}
[node distance=0.5cm, start chain=going below,]
\node [punktchain] (Kinetic)  {\textnormal{Nonlinear stochastic \\ VFP equation}};
\node[below=2cm of Kinetic](dummy){};
\begin{scope}[start branch=venstre,every join/.style={->, thick, shorten <=1pt}, ]
\node [punktchain,left=of dummy] (SG-gPC)  {\textnormal{gPC expansion}};
\node [punktchain,join,on chain=going below] (sol SG)  {\textnormal{Solution coupled system of PDEs}};
\node [punktchain,join, on chain=going below] (app)  {\textnormal{Approximation of expected solution}};
\end{scope}
\begin{scope}[start branch=hoejre,every join/.style={->, thick, shorten <=1pt}, ]
\node[punktchain,right=of dummy]  (SG-gPC-SDE)  {\textnormal{System of SDEs, $N\gg 0$}};
\node [punktchain,join, on chain=going below] (MCgPC)  {\textnormal{MCgPC scheme with fast algorithm for interactions}};
\node [punktchain,join, on chain=going below] (app2)  {\textnormal{Reconstruction of expected solution}};
\end{scope}
%\node[below=2cm of gPC_1](dummy_2){};
%\node [punktchain,right=of dummy_2] (constrained control)  {\textnormal{Constrained gPC System}};
\draw[|-,-|,->, thick,](Kinetic.south) |-+(0,-1em)-| (SG-gPC.north);
\draw[|-,-|,->, thick,] (Kinetic.south) |-+(0,-1em)-| (SG-gPC-SDE.north);
%\draw[thick,->](gPC_1.south) |-+(0,-3em)-| (constrained control.north);
%\draw[thick,->](MPC_2.south) |-+(0,-3em)-| (constrained control.north);
\end{tikzpicture}
\caption{Possible numerical approaches to nonlinear stochastic VFP equations, the right branch describes the MCgPC scheme.}
\label{fig:scheme}
\end{center}
\end{figure}

We sketch in Figure \ref{fig:scheme} the two approaches for the approximation of statistical quantities of the nonlinear nonlocal stochastic VFP problems of interest. In the left branch we find the standard SG approach where first we consider the gPC approximation of the original problem, which generates a coupled system of PDEs that can be solved through deterministic methods to obtain the evolution of expected quantities. On the right branch we describe the introduced MCgPC procedure, based on a particle reformulation of the problem which converges in distribution in the phase space to the solution of the problem. The advantage of considering a gPC scheme for the microscopic system lies in the preservation of the typical spectral convergence in the random space of the method. We highlight how thanks to the adoption of the computational strategy in Algorithm \ref{alg:1} the overall cost becomes $\mathcal O(M^2SN)$, $S\ll N$. 

For the reconstruction of expected quantities, in the present manuscript we consider the histogram of position and velocity of the set of particles in the phase space, we point the reader to \cite{HE} for possible alternatives. Thanks to the MC approach the resulting method preserves the positivity of the expected distribution function. 

\begin{rem}
In the case $S=N$ we obtain the typical convergence rate of order $\mathcal O(1/\sqrt{N})$ of Monte Carlo in the phase space, where $N$ is the number of particles and spectral convergence in $M$ in the random space. Due to the presence of the fast evaluation of interactions, the case $S<N$ induces an additional error of order $\mathcal O(\sqrt{1/S-1/N})$ in the microscopic dynamics. Let $f^{(S,M)}(\theta,x,v,t)$ the empirical distribution density estimated at time $t> 0$ from Algorithm \ref{alg:1} and let $J = \{j_1,\dots,j_S\}$, $|J| = S$, be the vector of indexes sampled uniformly without repetition in $\{1,\dots,N\}$. Then for any test function $\varphi$ we have
\[
\begin{split}
\left| \left\langle f^{(S,M)},\varphi \right\rangle - \left\langle f^{(N,M)},\varphi \right\rangle \right|   &= \left| \dfrac{1}{S} \sum_{j \in J} \varphi(x_j^M,v_j^M) -   \dfrac{1}{N} \sum_{j =1}^N \varphi(x_j^M,v_j^M)\right| \\
 & = \left|\left(\dfrac{1}{S}-\dfrac{1}{N} \right) \sum_{j\in J}\varphi(x_j^M,v_j^M) - \dfrac{1}{N}\sum_{j\not\in J}\varphi(x_j^M,v_j^M) \right| \\
 &\le \left|  \left\langle f,\varphi \right\rangle - \left(\dfrac{1}{S}-\dfrac{1}{N} \right)\sum_{j\in J}\varphi(x_j^M,v_j^M)  \right|   \\
 & \qquad+ \left|  \left\langle f,\varphi \right\rangle - \dfrac{1}{N}\sum_{j\not\in J}\varphi(x_j^M,v_j^M)  \right| = I_1 + I_2,
 \end{split}\]
whose leading error is given by $I_1$ since $S\ll N$ and, for $N\gg 0$, we may observe that the accuracy with which we approximate $f^{(N,M)}$ with $f^{(S,M)}$ is $\mathcal O\left(\sqrt{\dfrac{1}{S}-\dfrac{1}{N} } \right)$ thanks to a central limit theorem-type argument. 
\end{rem}

 \section{Numerical tests}\label{sect:num}
 
 In this section we present several numerical examples based on \eqref{eq:kinetic_CSM} both in the homogeneous and inhomogeneous case. We test the effectiveness of the MC-gPC scheme through several tests  based on VFP equations. In all test the integration of the system of stochastic differential equations \eqref{eq:systemSDE} is performed through a standard Euler-Maruyama method whereas the solution of the system of PDEs derived from the SG procedure is solved through a standard central scheme coupled with a fourth order Runge-Kutta integration. In the whole section we will consider a uniform noise, therefore Gauss-Legendre polynomial basis are chosen in the gPC setting. Numerical investigations on the influence of uncertainties in phase transition phenomena are presented through the section. Finally, we explore the non localized inhomogeneous model with Cucker-Smale type interactions. 

\subsection{Test 1: Space homogeneous case}

In this first test we consider the space homogeneous problem in 1D with uncertain diffusion parameter of the form
\be\label{eq:diff_def}
D(\theta) = \bar D + \lambda\bar D\theta, \qquad \bar D\ge 0,
\ee
where $\theta\sim\mathcal U([-1,1])$ and $\lambda\le 1$ is such that $D(\theta) \ge 0$ for all $\bar D\ge 0$. In this first test we consider a constant self propelling strength $\alpha(\theta) = \alpha\ge 0$. 
\begin{figure}
\centering
\subfigure[$\alpha= 1$, $\bar D = 0.2$]{
\includegraphics[scale=0.45]{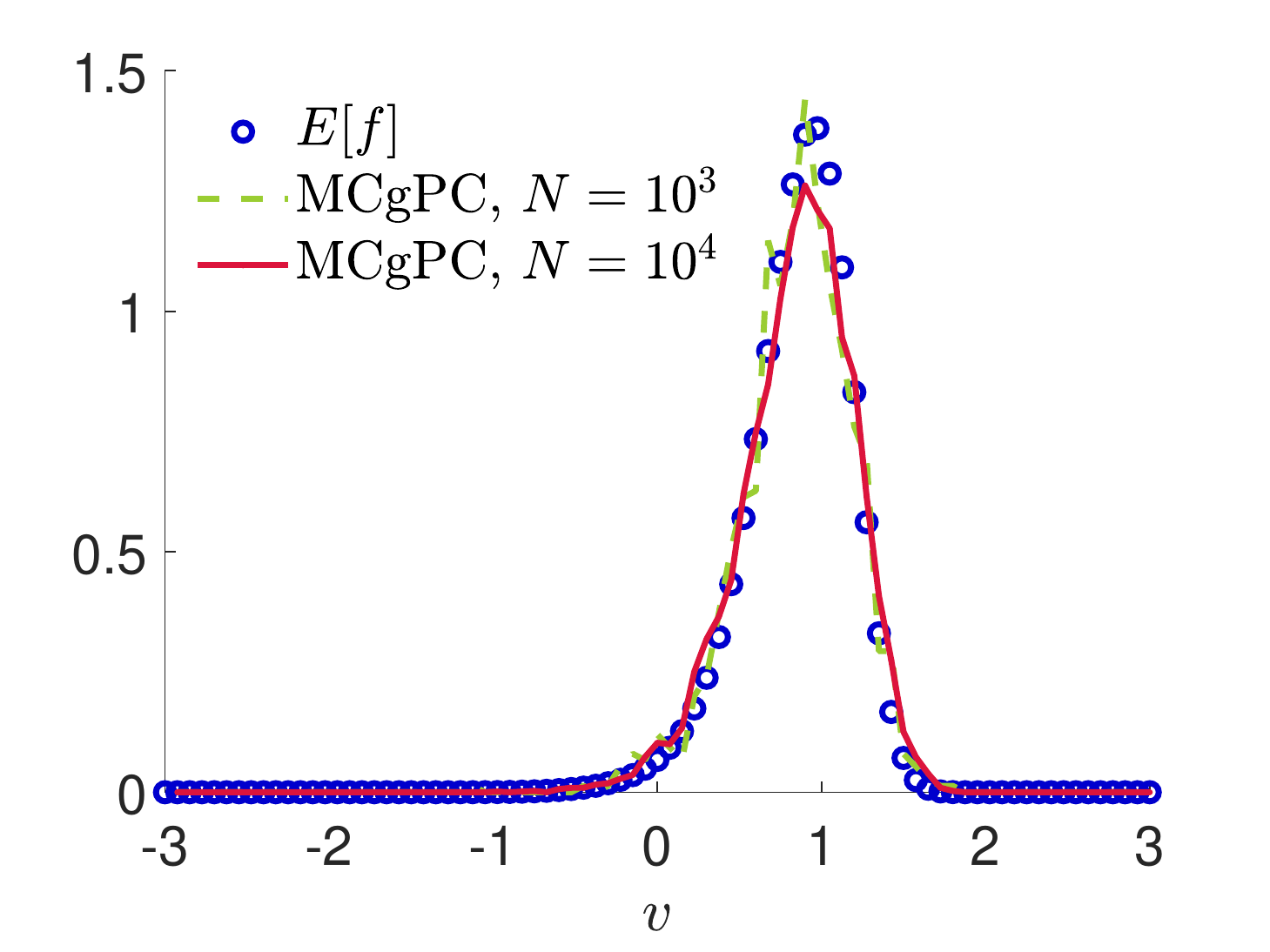}}
\subfigure[$\alpha = 1$, $\bar D = 0.8$]{
\includegraphics[scale=0.45]{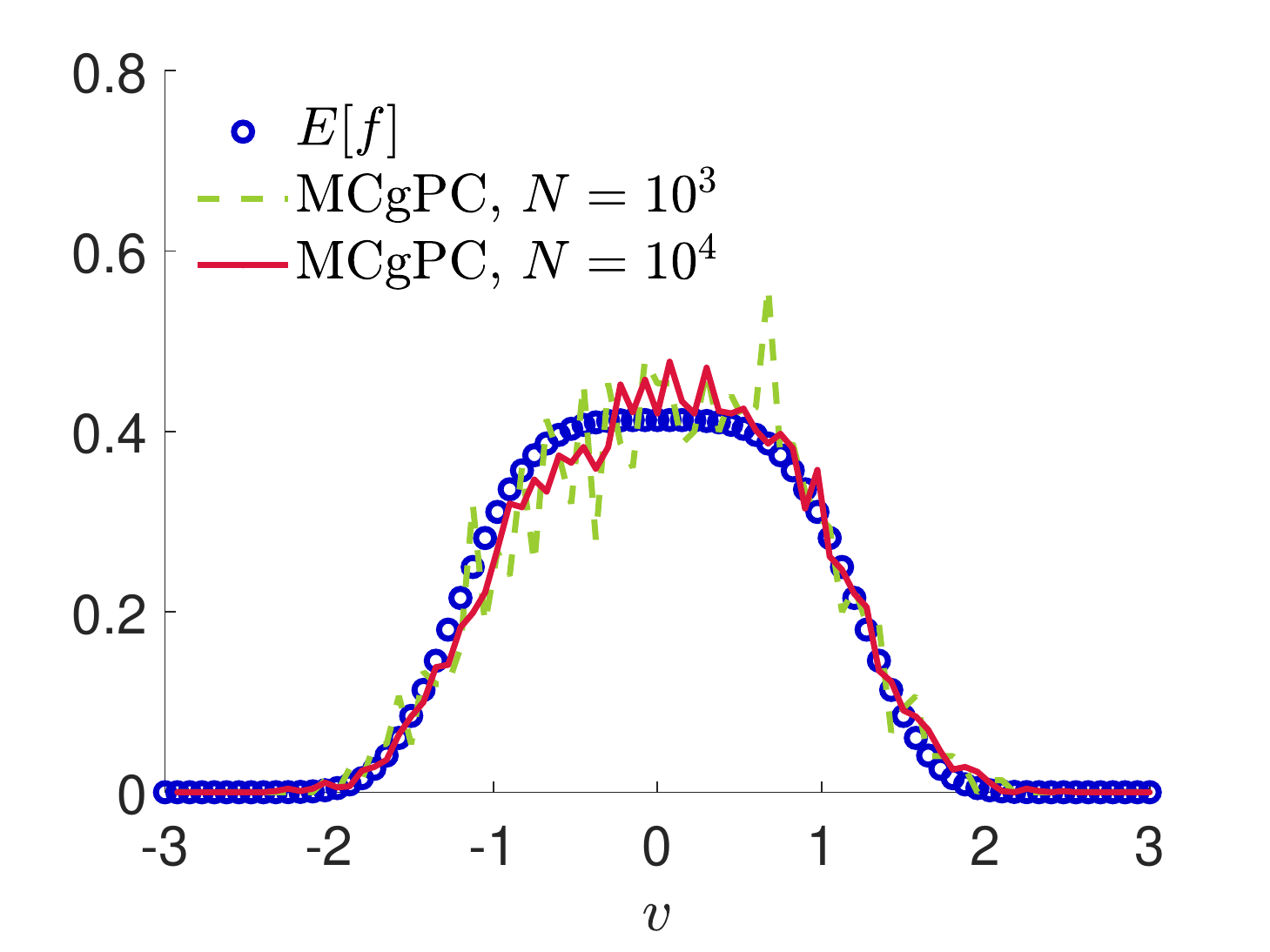}} \\
\subfigure[$\alpha= 2$, $\bar D = 0.2$]{
\includegraphics[scale=0.45]{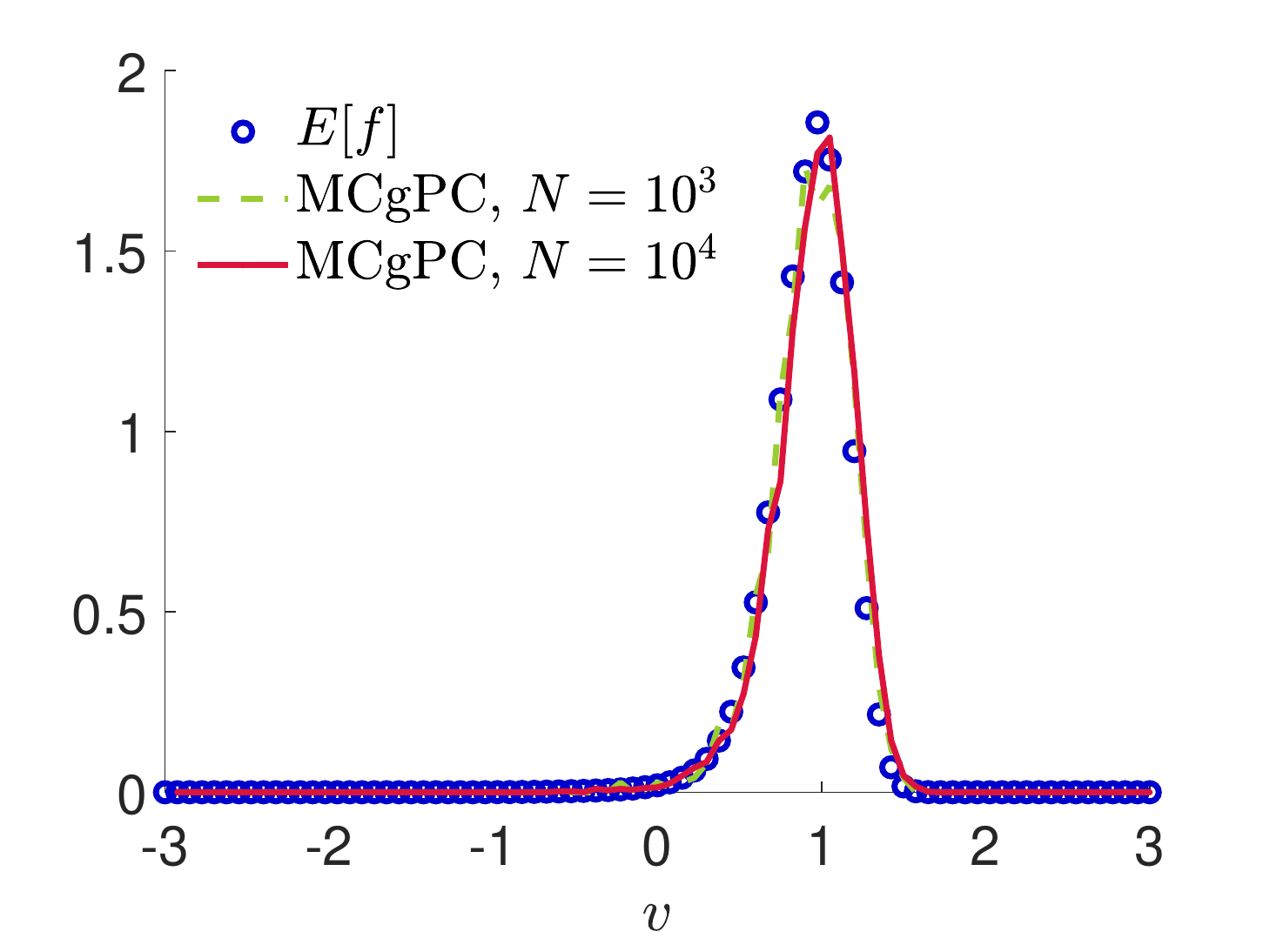}}
\subfigure[$\alpha= 2$, $\bar D = 0.8$]{
\includegraphics[scale=0.45]{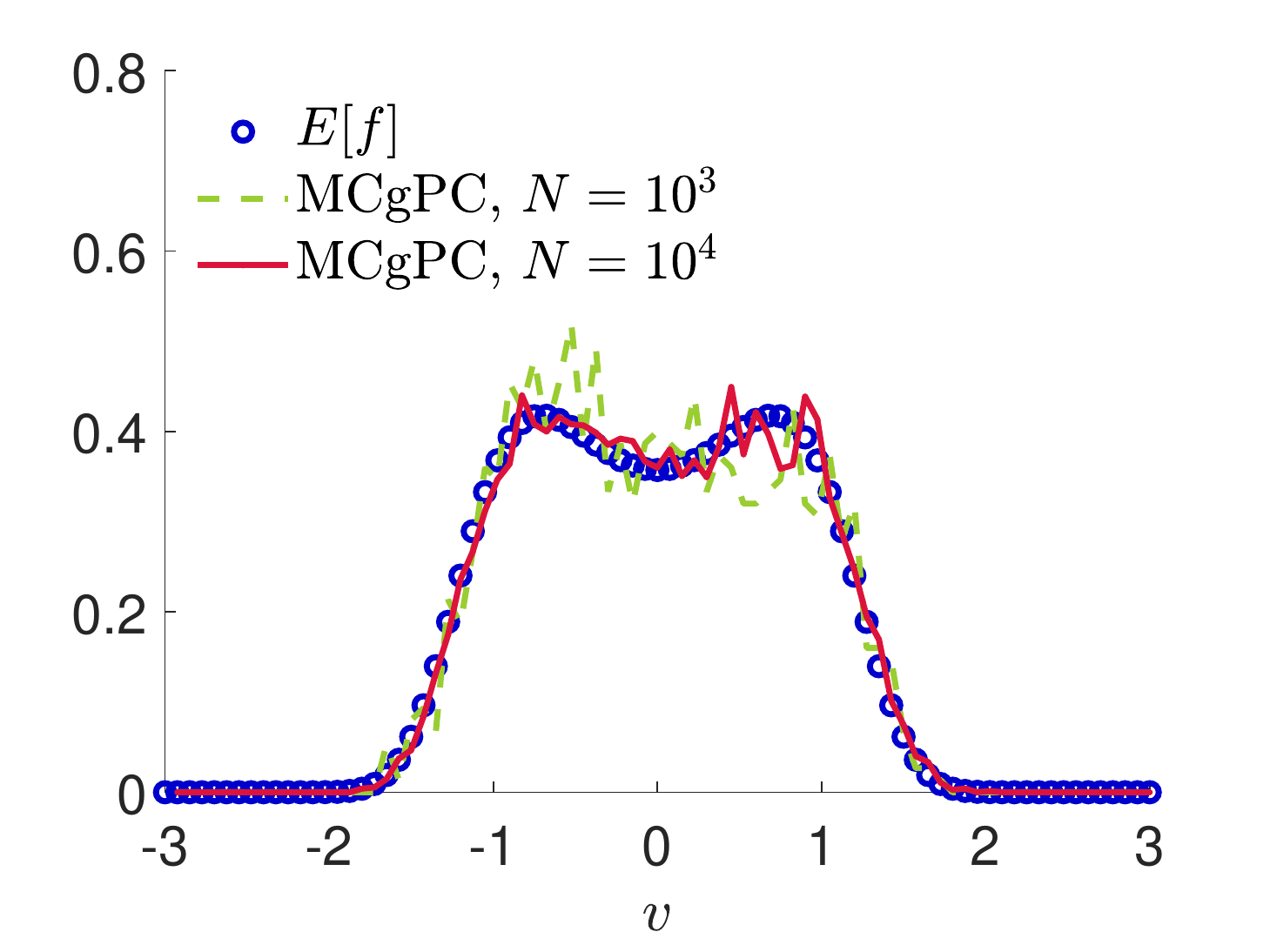}}
\caption{\textbf{Test 1}. We compare the expected density at time $T= 50$ obtained from \eqref{eq:test1D_gPC} through the gPC system and its MCgPC approximation with an increasing number of samples $N=10^3$, $N = 10^4$ and $M = 4$ gPC modes. We considered case $\bar D = 0.2$ (left column) and $\bar D=0.8$ (right column) and $\lambda = 0.1$, $\alpha = 1$ (top row) and $\alpha = 2$ (bottom row), the velocity space is discretized with $N_v = 81$ gridpoints and $\Delta t= 10^{-2}$.  }
\label{fig:1D_alpha1}
\end{figure}

The evolution of the density function $f(\theta,v,t)$, $v\in\RR$ is ruled by the following PDE 
\[\label{eq:test1D}
\partial f(\theta,v,t) = \partial_v \left[ \alpha(|v|^2-1)f(\theta,v,t) + (v-u_f(\theta,t))f(\theta,v,t) + D(\theta)\partial_v f(\theta,v,t) \right],
\]
whose SG-gPC approximation is given for all $h = 0,\dots,M$ by 
\be\begin{split}\label{eq:test1D_gPC}
\partial_t \hat f_h(v,t) &= \partial_v \left[ \alpha(|v|^2-1)v \hat f_h(v,t)+\sum_{k=0}^M P_{hk} \hat f_k(v,t)+ \sum_{k=0}^M D_{hk} \partial_v \hat f_k(v,t)\right],
\end{split}\ee
being 
\[
\begin{split}
P_{hk} &= \dfrac{1}{\|\Phi_h^2 \|_{L^2(\Omega)}} \int_{\IT}(v-u_{f^M})\Phi_h(\theta)\Phi_k(\theta)d\Psi(\theta),\\
D_{hm} &= \dfrac{1}{\|\Phi_h^2 \|_{L^2(\Omega)}}  \int_{\IT}D(\theta)\Phi_h(\theta)\Phi_m(\theta)d\Psi(\theta). 
\end{split}
\]
At the particle level we obtain from \eqref{eq:system_SDEh} the following coupled system of SDEs for the evolution of the particles' velocities
\[
d \hat v_{i,h} =\left( \sum_{k=0}^M s_{hk}(v_i^M)\hat{v}_{i,k} + \hat u_{h}- \hat v_h\right)dt + d_h dW_i,
\]
being 
\[
\begin{split}
s_{hk}(v_i^M) &= \int_{I_{\theta}} \alpha(1-|v_i^M|^2)\Phi_k(\theta)\Phi_h(\theta)\Psi(\theta)d\theta,\\
d_h &= \int_{I_{\theta}} (2\bar D + 2\lambda \bar D\theta)^{1/2} \Psi(\theta)d\theta.
\end{split}
\]
and $\left\{ W_i \right\}_{i=1}^N$ defines a set of $N$ independent Wiener processes. Furthermore, we indicated with $\hat u_h$ the  $h$th projection of the average velocity of the system. 
\begin{figure}
\centering
\includegraphics[scale=0.5]{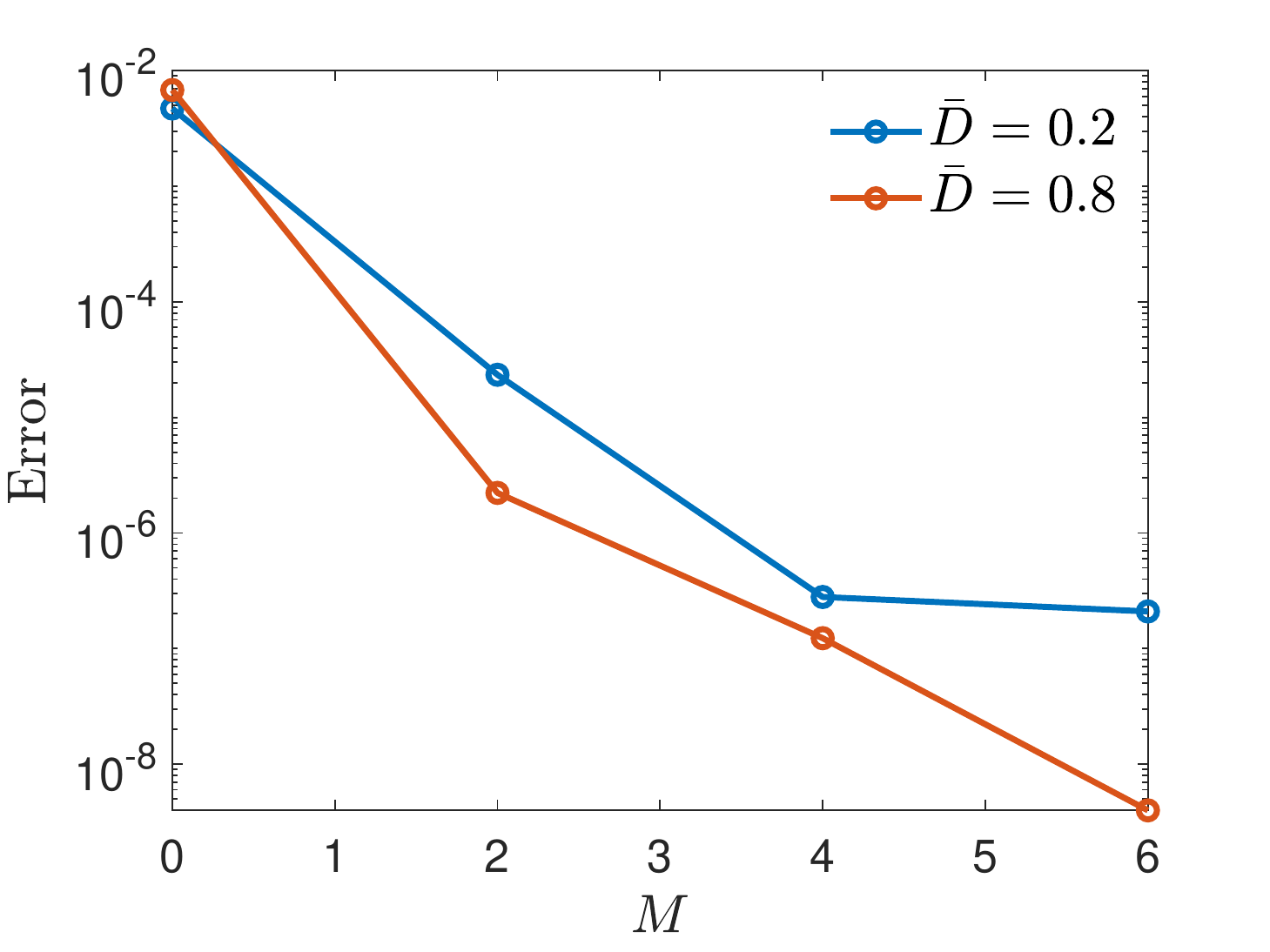}
\caption{\textbf{Test 1}. Convergence in $L^2(\Omega)$ based on a reference temperature $\mathcal T^{ref}$ at time $T=50$ computed at the particle level with a gPC expansion of degree $M=20$. }
\label{fig:L2_error}
\end{figure}

The long time solution is given by \eqref{eq:ss_hom} as discussed in Section \ref{sect:hom}. We consider as deterministic initial condition the Gaussian distribution
\[
f_0(v) = \dfrac{1}{\sqrt{2\pi \sigma^2}} \exp \left\{ -\dfrac{(v-\mu)^2}{2\sigma^2} \right\},\qquad \mu = 1, \sigma^2 = \dfrac{1}{4}. 
\]
At the particle level the initial velocities are a sample of $N\gg 0$ velocities from the distribution $f_0(v)$. In Figure \ref{fig:1D_alpha1} we compare the numerical expected distribution of the SG-gPC system \eqref{eq:test1D_gPC} and the reconstructed expected density of the MCgPC scheme at time $T=50$ for two uncertain diffusion coefficients \eqref{eq:diff_def} with $\bar D=0.2$, $\bar D = 0.8$ and $\lambda=0.1$ in both cases. It is easily observed how for an increasing number of particles the statistical quantities are well approximated by the scheme.

\begin{figure}
\centering
\subfigure[$\lambda = 10^{-1}$]{
\includegraphics[scale=0.45]{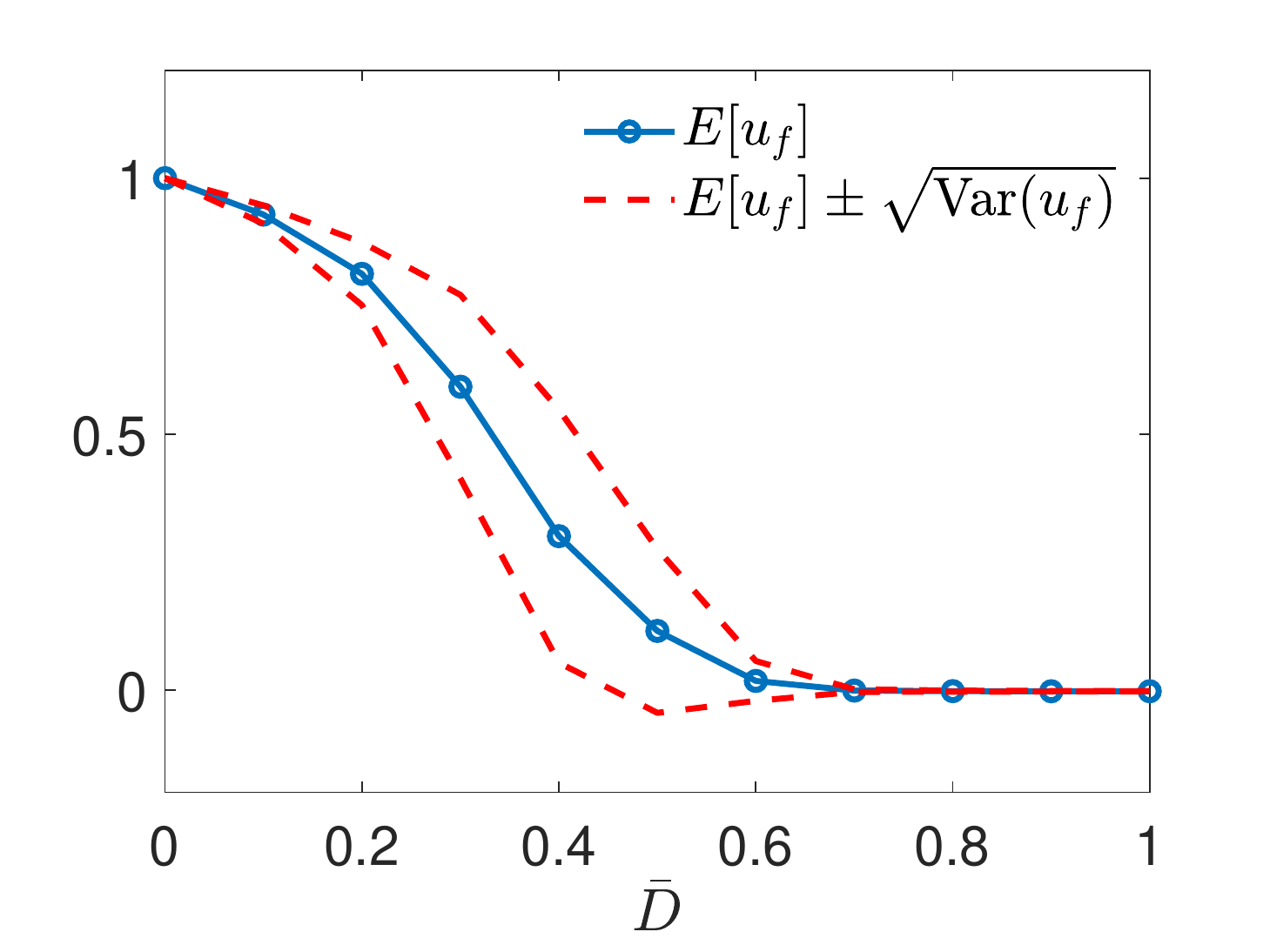}}
\subfigure[$\lambda = 10^{-1}$]{
\includegraphics[scale=0.45]{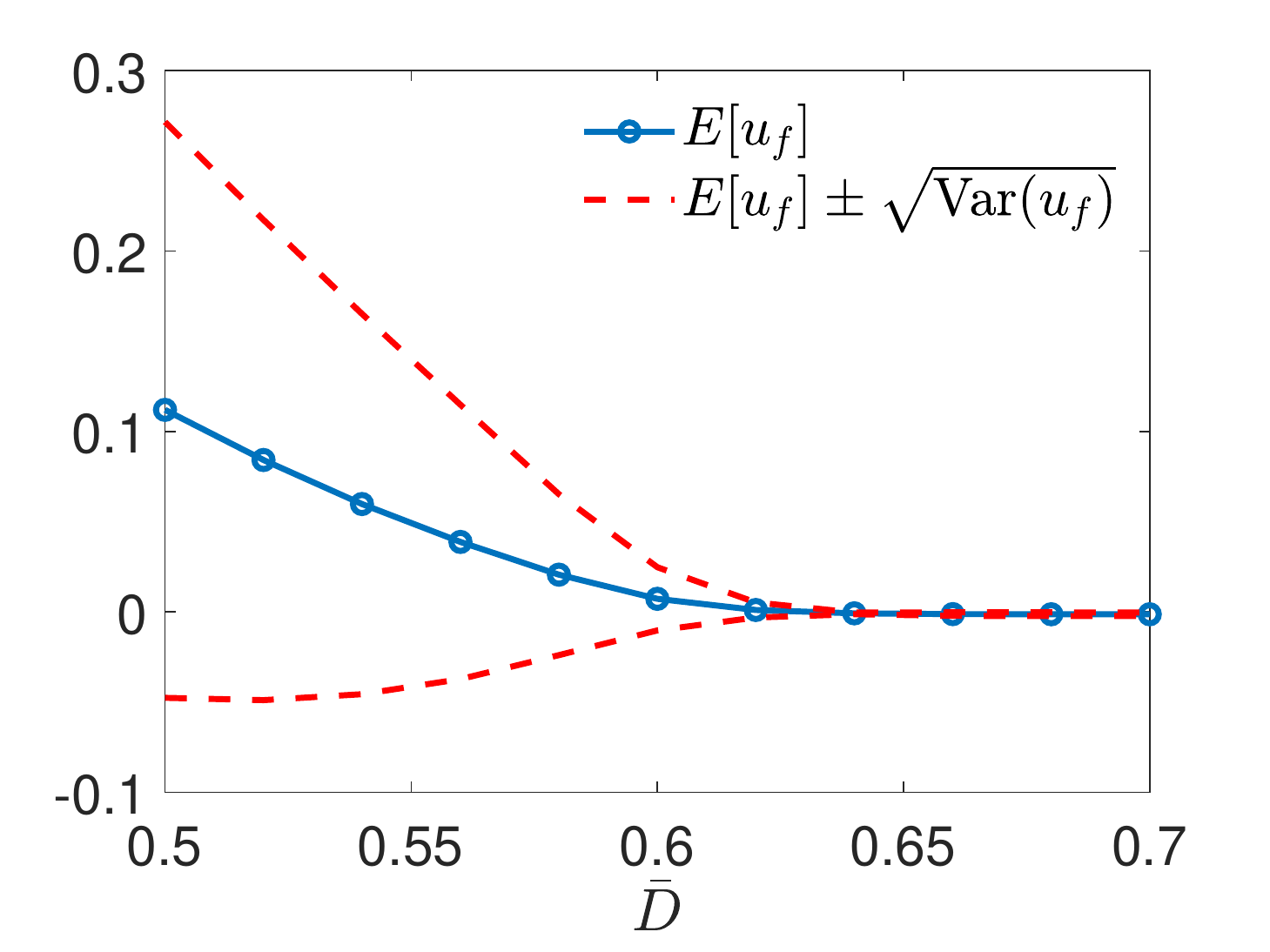}} \\
\subfigure[$\lambda = 10^{-3}$]{
\includegraphics[scale=0.45]{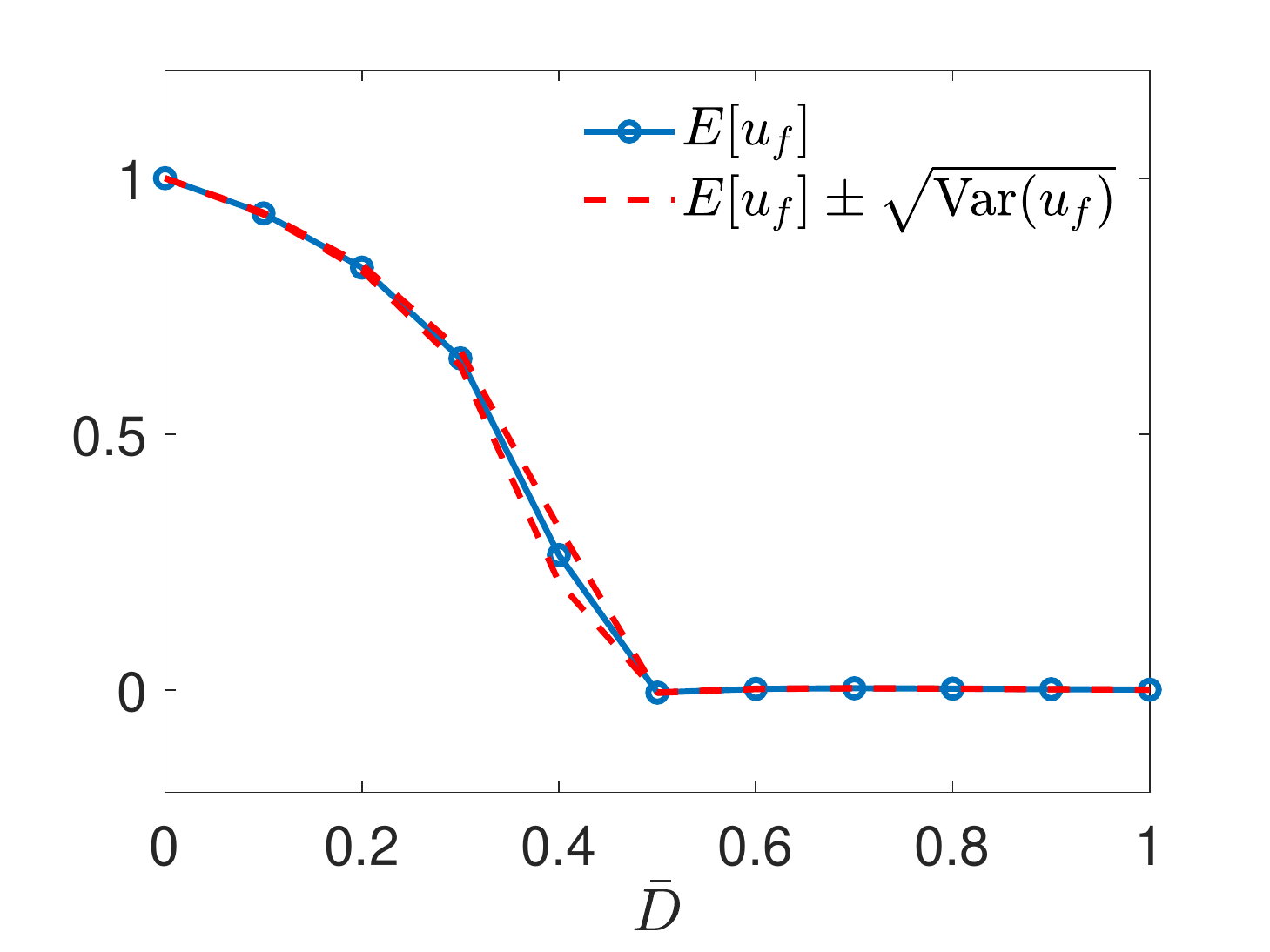}}
\subfigure[$\lambda = 10^{-3}$]{
\includegraphics[scale=0.45]{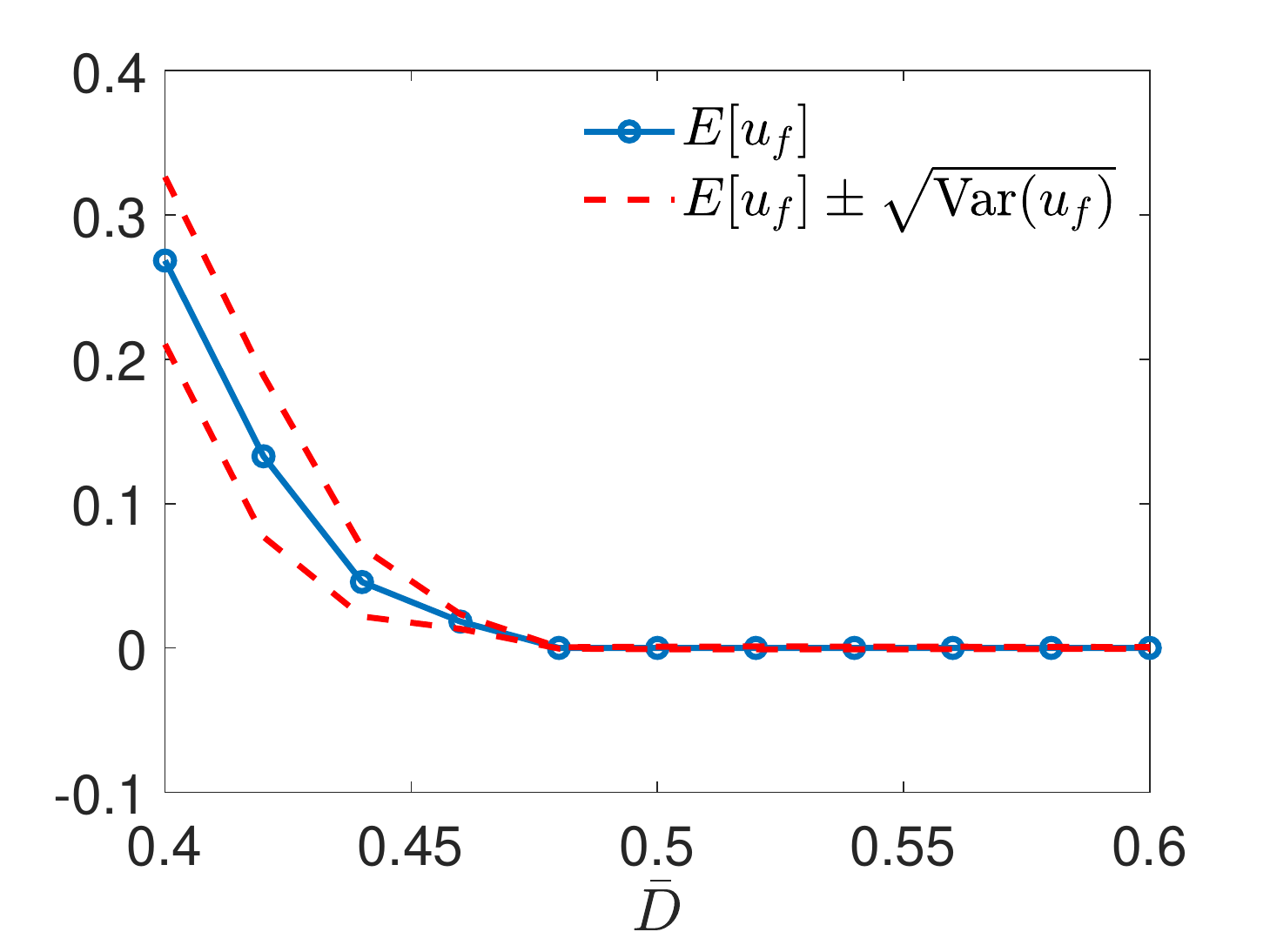}}
\caption{Evolution of the expected average velocity of the system $\mathbb E[u_f(\theta,T)]$ for $T=50$ and several $\bar D = \frac{\kappa}{10}$, $\kappa = 0,\dots,10$. In red we computed the confidence band. The result is obtained through the MCgPC scheme with $N = 10^4$ particles and $M= 4$ gPC modes. We considered an uncertain diffusion of the form \eqref{eq:diff_def} with $\lambda = 10^{-1}$ in the top-left figure and $\lambda = 10^{-3}$ in the bottom-left figure. The right plots are a further refinement of $\bar D$ in a subinterval near the phase transition.  Different sizes of the variability region for $\mathbb E[u_f]$ indicates different sensibilities to the action of uncertainties. }
\label{fig:Dbands}
\end{figure}

In Figure \ref{fig:L2_error} we study the convergence of the temperature of the system 
\[
\mathcal T(\theta,t) =  \int_\RR (v-u(\theta,t))^2 f(\theta,v,t)dv,
\] 
in $L^2(\Omega)$ obtained through the MCgPC algorithm for an increasing order of the gPC expansion. We considered  $N=10^4$ particles and two uncertain diffusion parameters of the form \eqref{eq:diff_def} with $\bar D= 0.2$ and $\bar D= 0.8$, $\lambda =0.1$ and $\theta\sim \mathcal U([-1,1])$. Time integration has been performed up to time $T=50$ with time step $\Delta t = 10^{-2}$. 

To complete the overview of the one dimensional setting we computed in Figure \ref{fig:Dbands} the large time behaviour of the expected average velocity $\mathbb E[u_f(\theta,T)]$ for several $\bar D = \frac{\kappa}{10}$, $\kappa = 0,\dots,10$ with constant self-propulsion term $\alpha =1$. We considered an uncertain diffusion of the form \eqref{eq:diff_def} with $\lambda = 10^{-1}$ (top row) and $\lambda = 10^{-3}$ (bottom row). In each figure we marked in dashed red the confidence band computed in terms of the approximated standard deviation $\sqrt{\textrm{Var}(u_f)}$. 

We can easily observe the regions of maximal sensitivity with respect to the presence of uncertainties. In particular, for high diffusion values the variability of the expected average velocity vanishes and we may argue that the phase transition predicted in \cite{BCCD} is actually a quite stable pattern in the space homogeneous regime. Nevertheless, the averaging of uncertain quantities acts as a smoothing factor of the phase transition as we can clearly observe in Figure \ref{fig:Dbands}-(b) and \ref{fig:Dbands}-(d). For a vanishing influence of $\theta\in I_{\theta}$ given by $\lambda\rightarrow 0$ in \eqref{eq:diff_def} the transition becomes sharper coherently with the deterministic case, see figure \ref{fig:Dbands}-(d). Summarizing the main effect of the uncertainties is the smoothing of the transition point making it less sharper and abrupt than in the deterministic cases.

\subsection{Test 2. Space inhomogeneous case}

In this section we focus on inhomogeneous models. First we consider the localised case for which a phase transition is expected to happen similarly to the homogeneous case, see \cite{BCCD}, even if this has not yet been proved. To explore other possible alternatives we will also consider the case of Cucker-Smale type interactions for which no analogous theoretical results regarding phase transitions currently exist. In all the tests of the present section we consider a sample of $N = 10^5$ particles for the MCgPC scheme. 

\subsubsection{ Test 2A. Localized interaction case}
We consider a space dependent interaction function of Dirac delta form, i.e. $P(x,x_*) = \delta(x-x_*)$. With this choice agents interact only is they share the same position. In this case, the MCgPC scheme translates in a second order system of SDEs of the form 
\[
\displaystyle 
\begin{cases}
d\hat x_{i,h} = \hat v_{i,h}dt \\
d \hat v_{i,h} = \displaystyle  \left( \sum_{k=0}^M s_{hk}(v_i^M)\hat{v}_{i,k} + \dfrac{1}{N} \sum_{k=0}^M \sum_{i=1}^N P_{hk}^{ij} (\hat{v}_{j,k}-\hat{v}_{i,k})\right)dt + d_h dW_i, 
\end{cases}
\]
where $s_{hk}$ has the same definition of the space homogeneous case. As before we consider an uncertain diffusion parameter $D(\theta)$ and a constant self-propulsion $\alpha=1$.

\begin{figure}
\centering
\subfigure[gPC,  $t=0.5$]{
\includegraphics[scale=0.3]{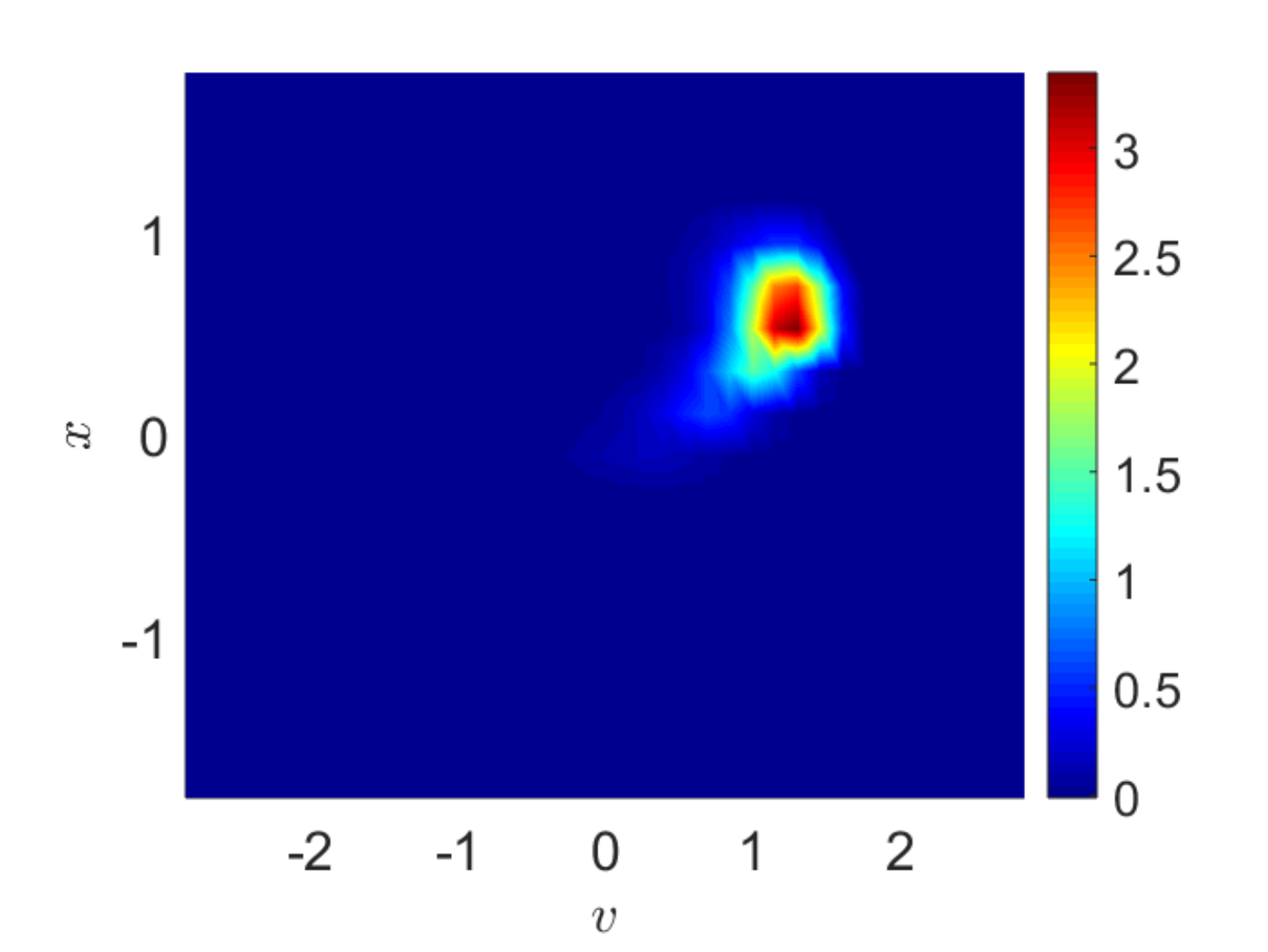}}
\subfigure[gPC,  $t=1.0$]{
\includegraphics[scale=0.3]{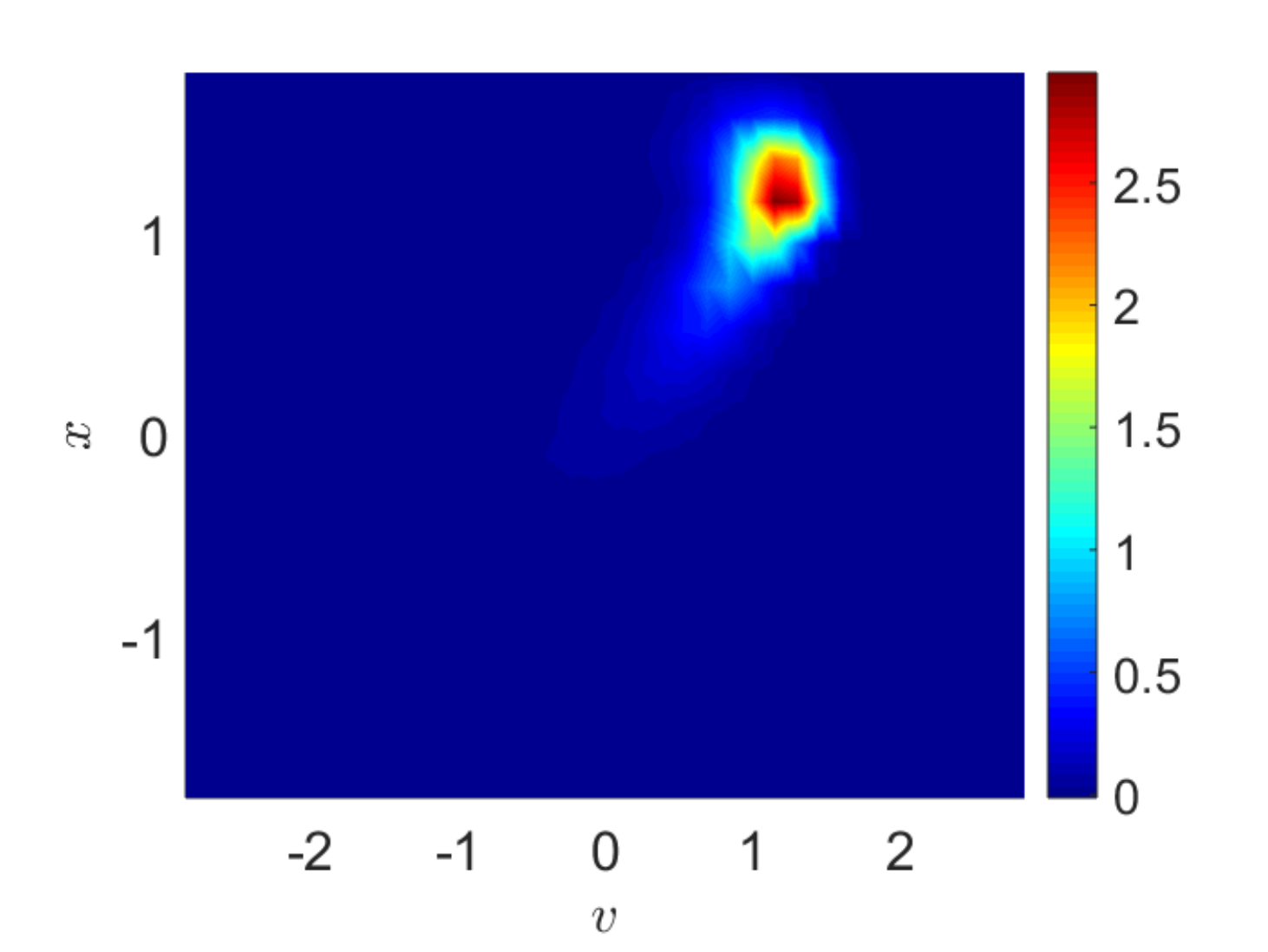}}
\subfigure[gPC, $t=5.0$]{
\includegraphics[scale=0.3]{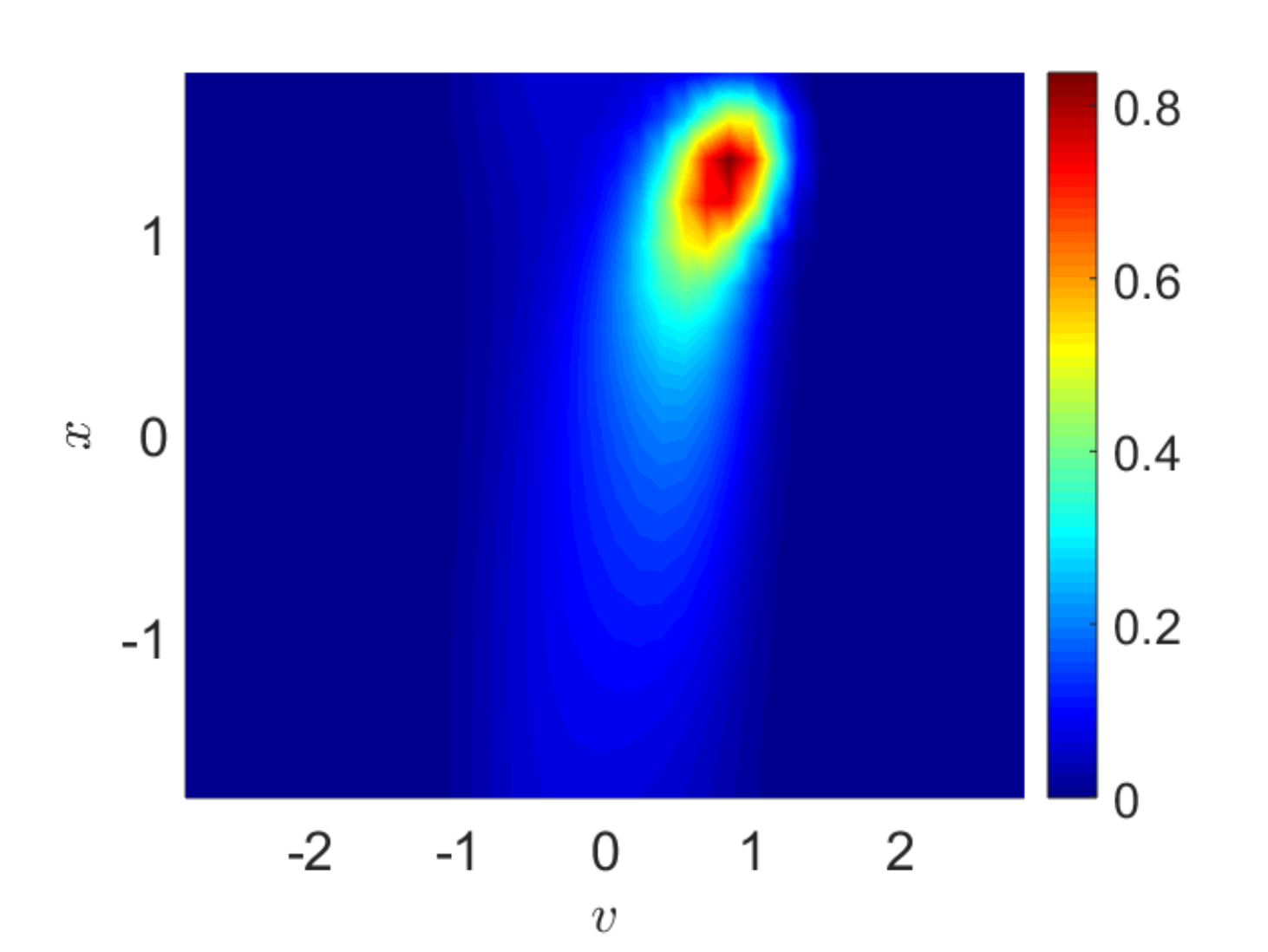}} \\
\subfigure[MCgPC, $t=0.5$]{
\includegraphics[scale=0.3]{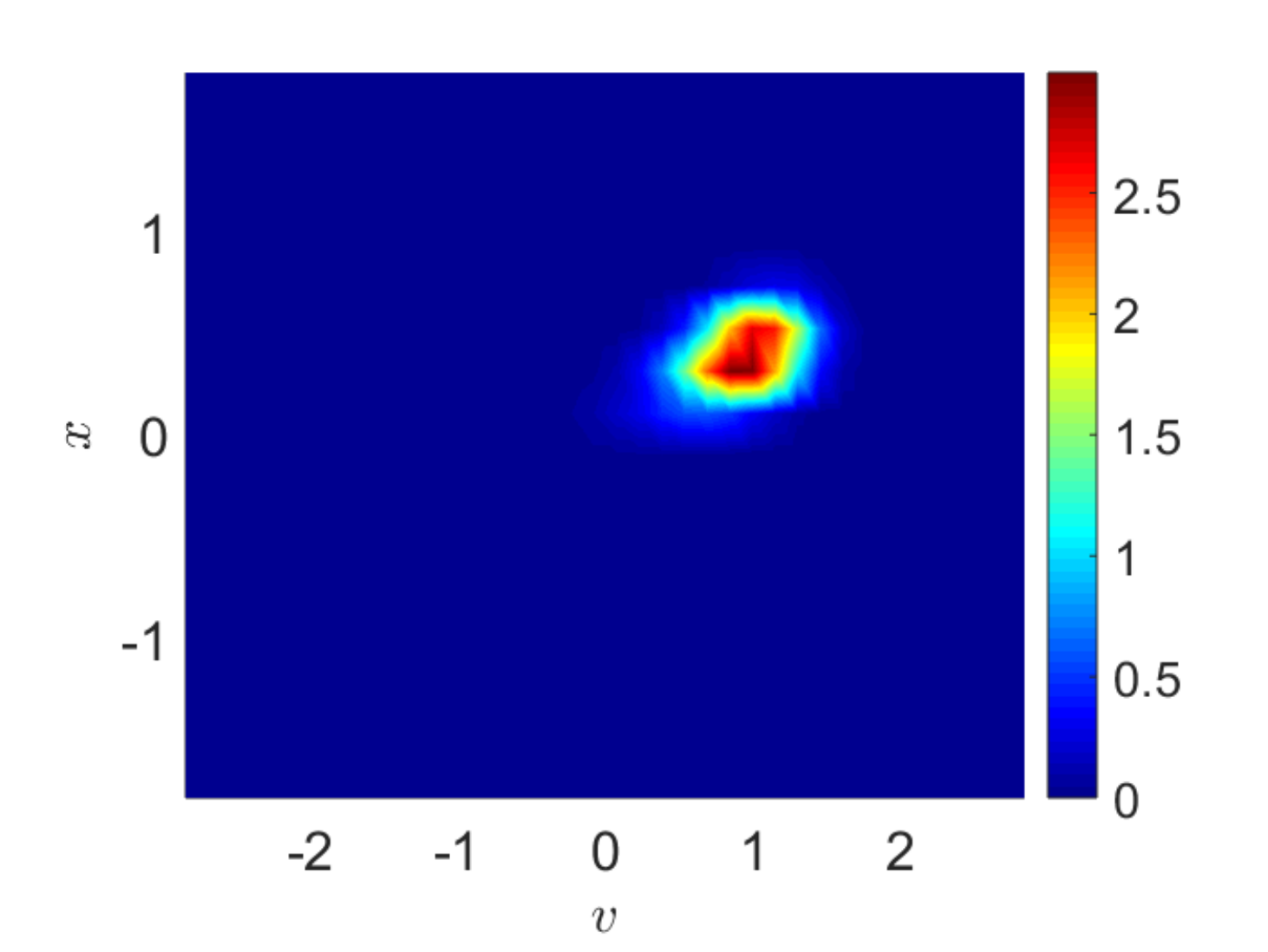}}
\subfigure[MCgPC, $t=1.0$]{
\includegraphics[scale=0.3]{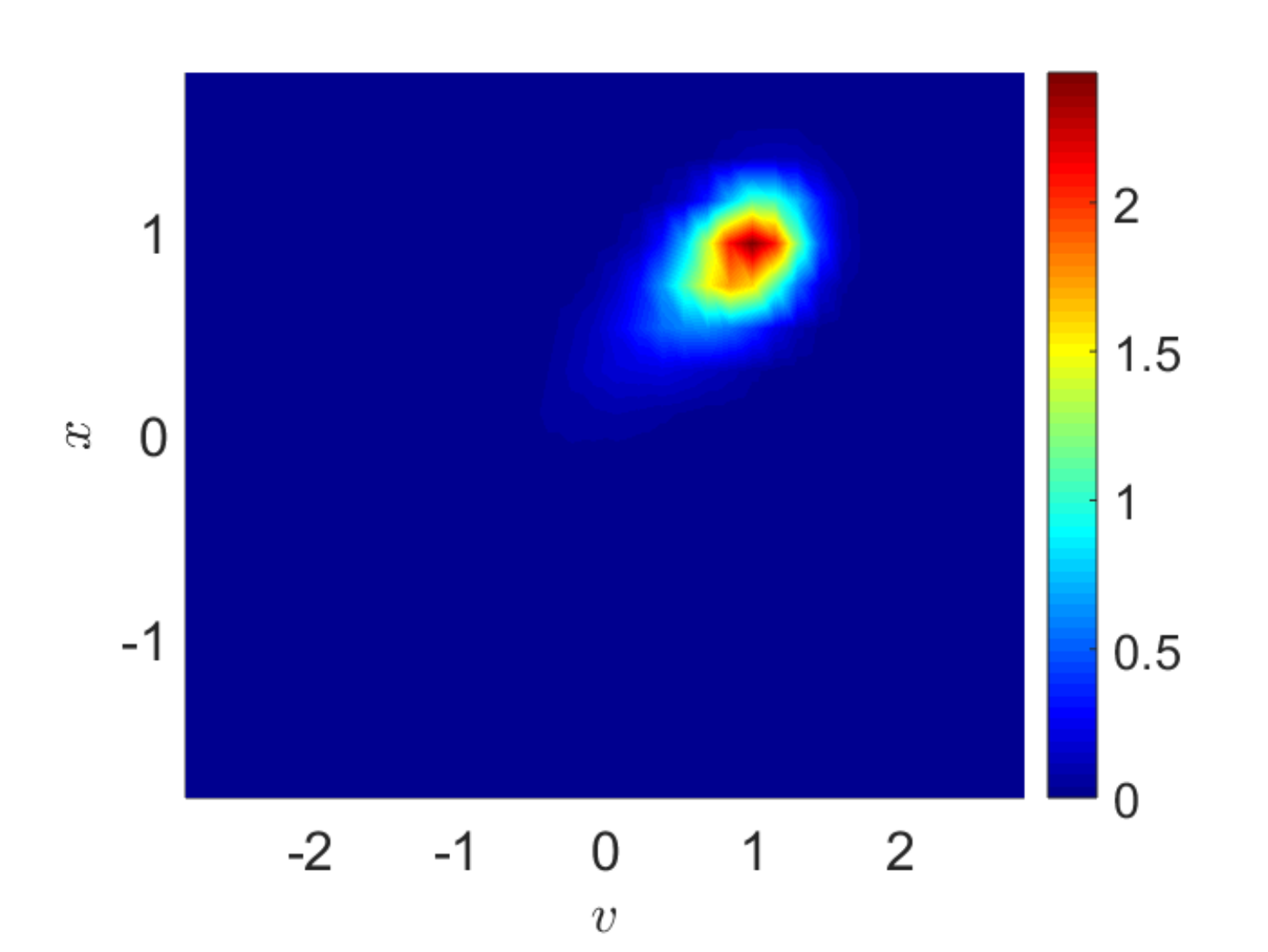}}
\subfigure[MCgPC, $t=5.0$]{
\includegraphics[scale=0.3]{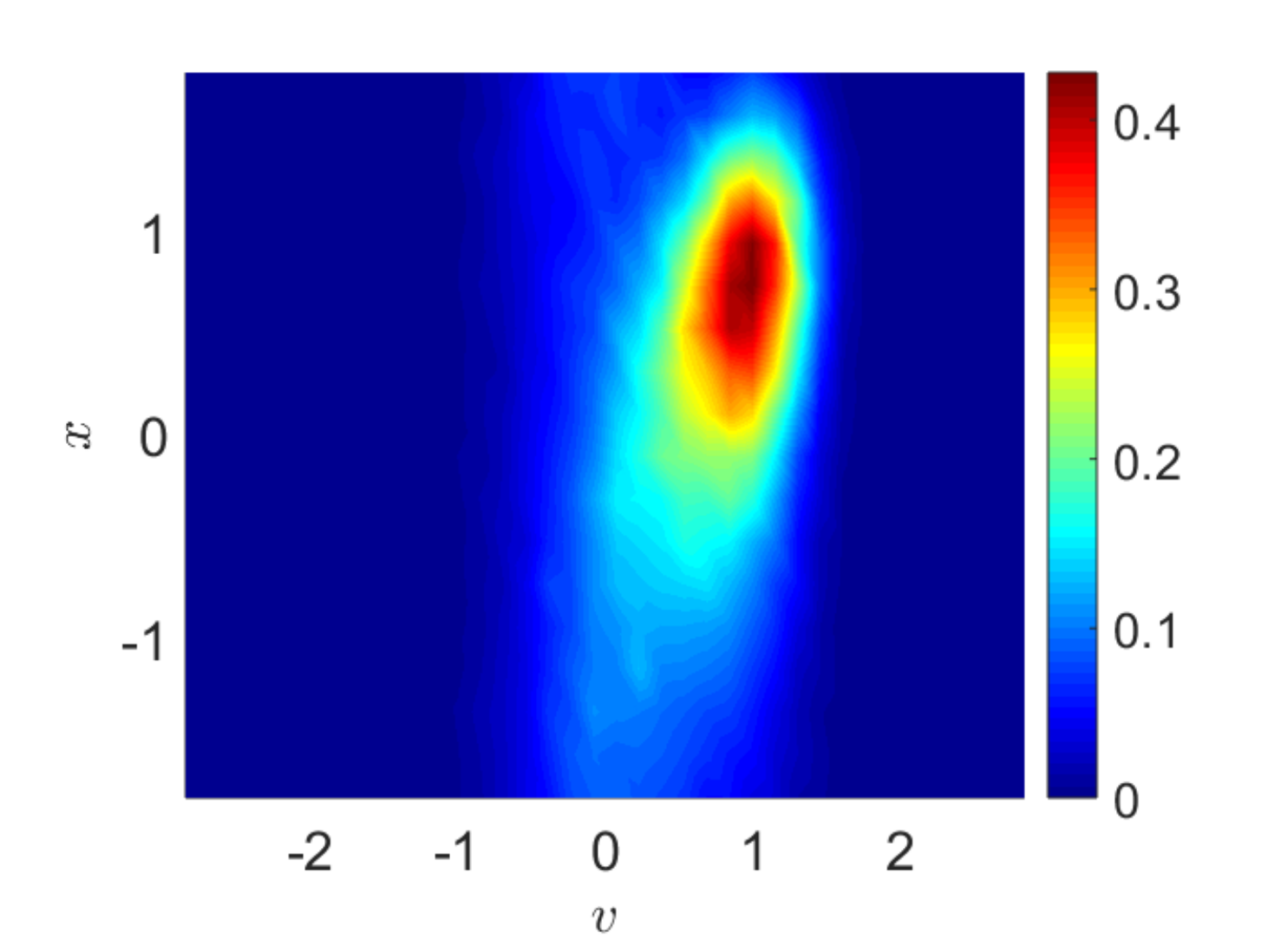}}
\caption{\textbf{Test 2A}. Evolution of the expected density in the inhomogeneous kinetic equation \eqref{eq:MF_general} with localized interactions at times $t = 0.5, 1.0, 5.0$. We considered deterministic self-propulsion strength $\alpha = 1$ and uncertain diffusion of the form \eqref{eq:diff_def} with  $\bar D=0.2$ and $\lambda = 0.1$. \textbf{Top row}: numerical solution of the gPC system.  \textbf{Bottom row}: reconstruction of the expected density obtained through MCgPC scheme for $N = 10^5$ particles with mean-field MC algorithm with interaction subset of $S=10$ agents. We considered $N_x=20$ gridpoints in the interval $[-2,2]$ for the space discretization and $N_v = 40$ gridpoints in the interval $[-3,3]$ for the velocity discretization. }
\label{fig:inhom_loc_D02}
\end{figure}

 We introduce a discretization of the phase space $\{x_i\}_{i=1}^{N_x}$, $\{v_i\}_{i=1}^{N_v}$ such that $x_{i+1}-x_i=\Delta x$, $v_{i+1}-v_i=\Delta v$ made by $N_x = 20$ gridpoints for the space variable and by $N_v = 40$ points for the velocity variable. At time $t=0$ we consider the following multivariate Gaussian distribution 
\[
f_0(x,v) = \dfrac{1}{2\pi \sigma_x \sigma_v} \exp\left\{-\dfrac{1}{2}\left[ \dfrac{(x-\mu_x)^2}{2\sigma_v^2}+ \dfrac{(v-\mu_v)^2}{2\sigma_v^2}  \right] \right\}
\]
we considered $\mu_x = 0$, $\mu_v = 1$, $\sigma_v = \frac{1}{10}$, $\sigma_x = \frac{1}{100}$. Let us consider the evolution of the gPC system of PDEs derived in \eqref{eq:gPC_general} with localized interactions in \eqref{eq:Phk}. At the numerical level we perform a second order Strang splitting approach. The transport step is integrated through a third order Runge-Kutta scheme with a fifth order WENO reconstruction \cite{CWS}. In the following we compare the behavior of the continuous approximation with the one obtained through the MCgPC scheme where, to localize the interaction, we considered as a smoothing factor the indicator function of the numerical cell $[x_i, x_{i+1}]$.

In Figure \ref{fig:inhom_loc_D02} and Figure \ref{fig:inhom_loc_D08} we compare the evolution of the expected density over the time interval $[0,5]$ for the inhomogeneous kinetic equation \eqref{eq:MF_general} with local interactions. In particular, we present the expected density obtained through numerical integration of the SG-gPC system of PDEs with the one approximated through the MCgPC scheme with fast evaluation of the interaction term through a mean-field MC approach described in Algorithm \ref{alg:1} and interaction subset of $S= 10$ particles. The uncertain diffusion has the form \eqref{eq:diff_def} with $\bar D = 0.2$, $\bar D = 0.8$ and $\lambda = 0.1$. Periodic boundary conditions has been considered. The results are in agreement with the space homogeneous case since the interaction are localised, it is easily observed how a phase transition occurs switching from an ordered state to a chaotic isotropic state for increasing values of the diffusion. In Figure \ref{fig:kkk} the average in position distributions in velocity are compared between the MCgPC and gPC. These results corroborate the appearance of a smoothed phase transition: a shifted distribution with positive average velocity for small values of the average diffusion is observed while for larger diffusion the distribution becomes almost symmetric.

\begin{figure}
\centering
\subfigure[gPC,  $t=0.5$]{
\includegraphics[scale=0.3]{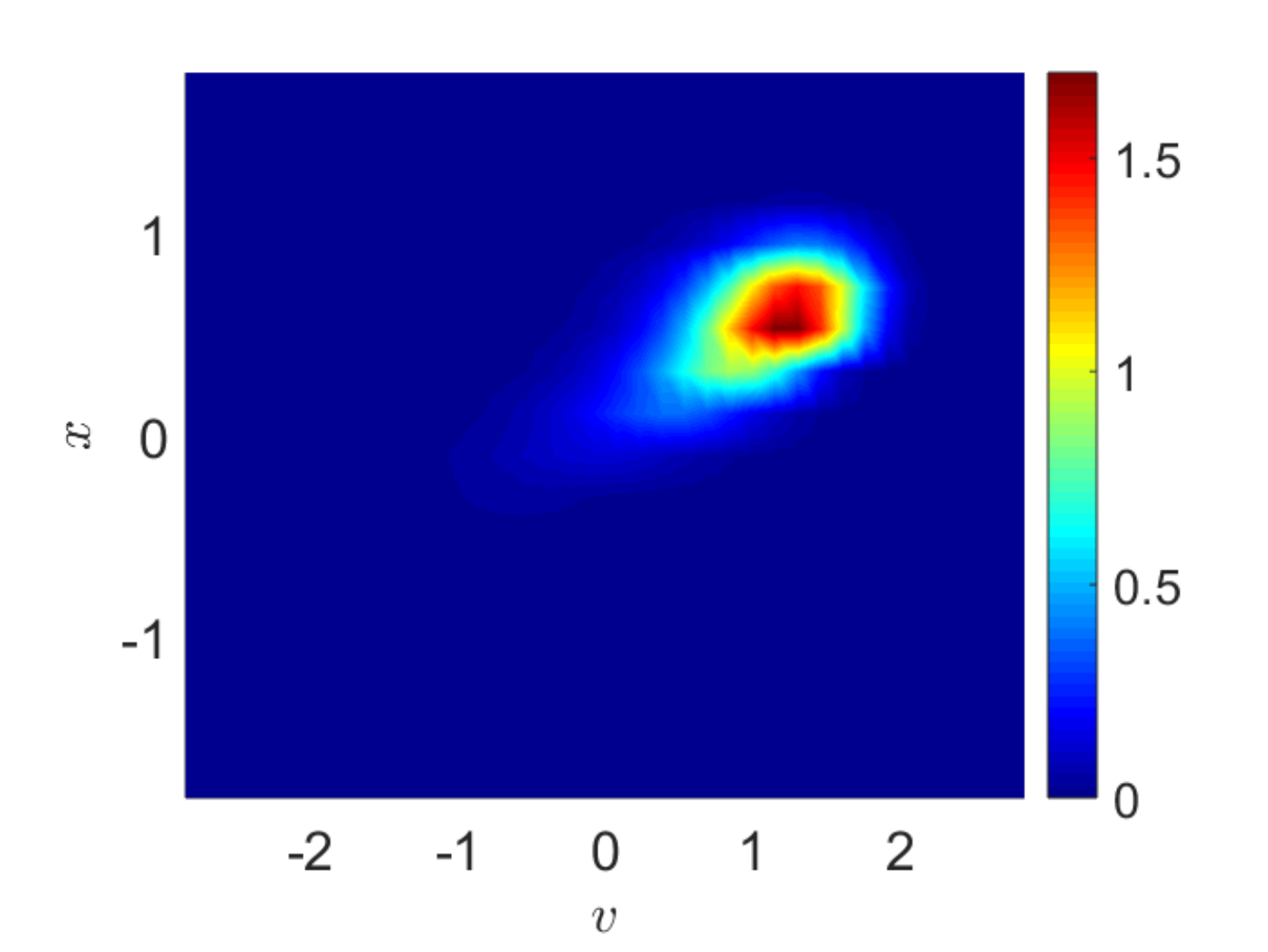}}
\subfigure[gPC,  $t=1.0$]{
\includegraphics[scale=0.3]{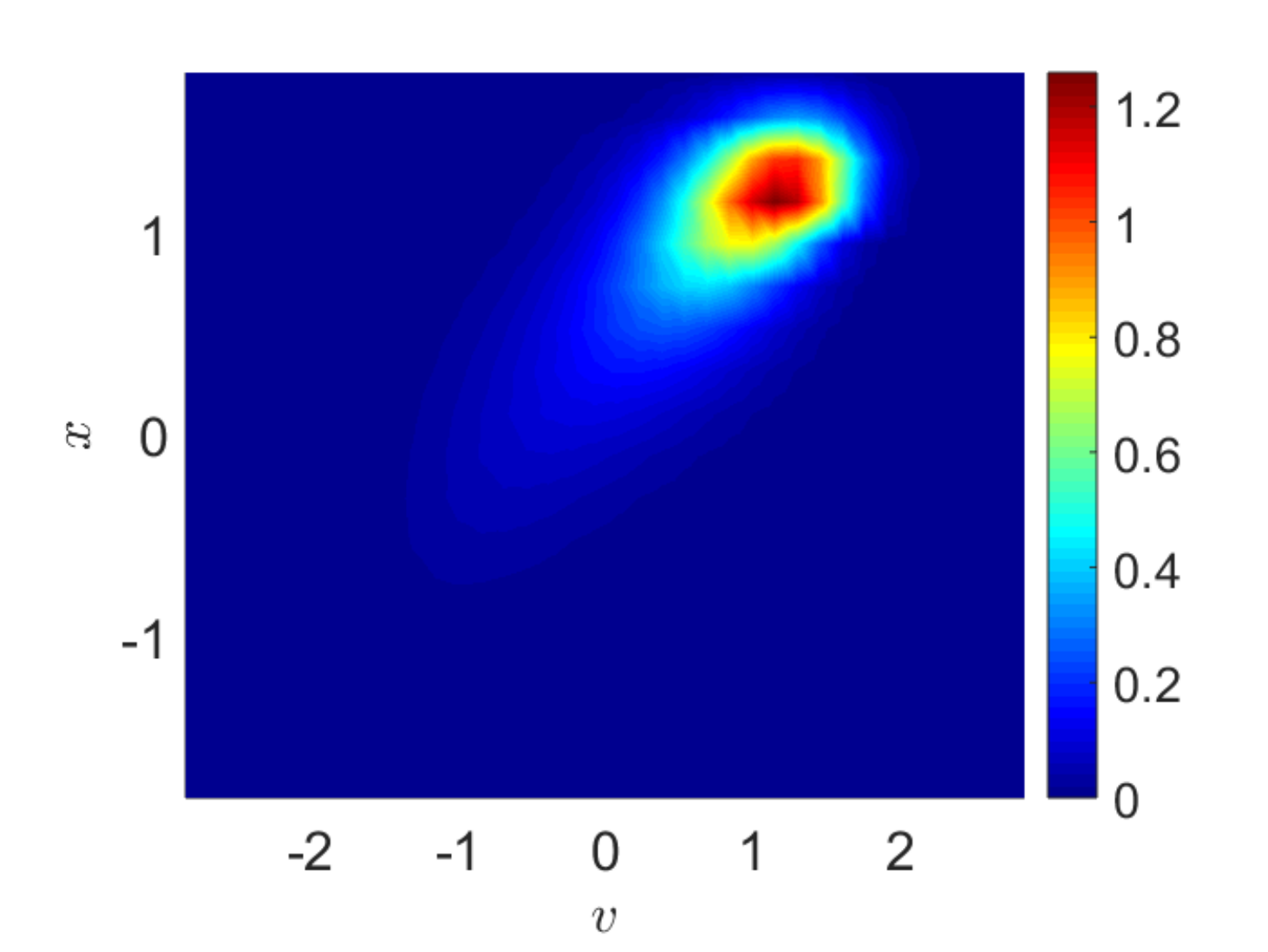}}
\subfigure[gPC, $t=5.0$]{
\includegraphics[scale=0.3]{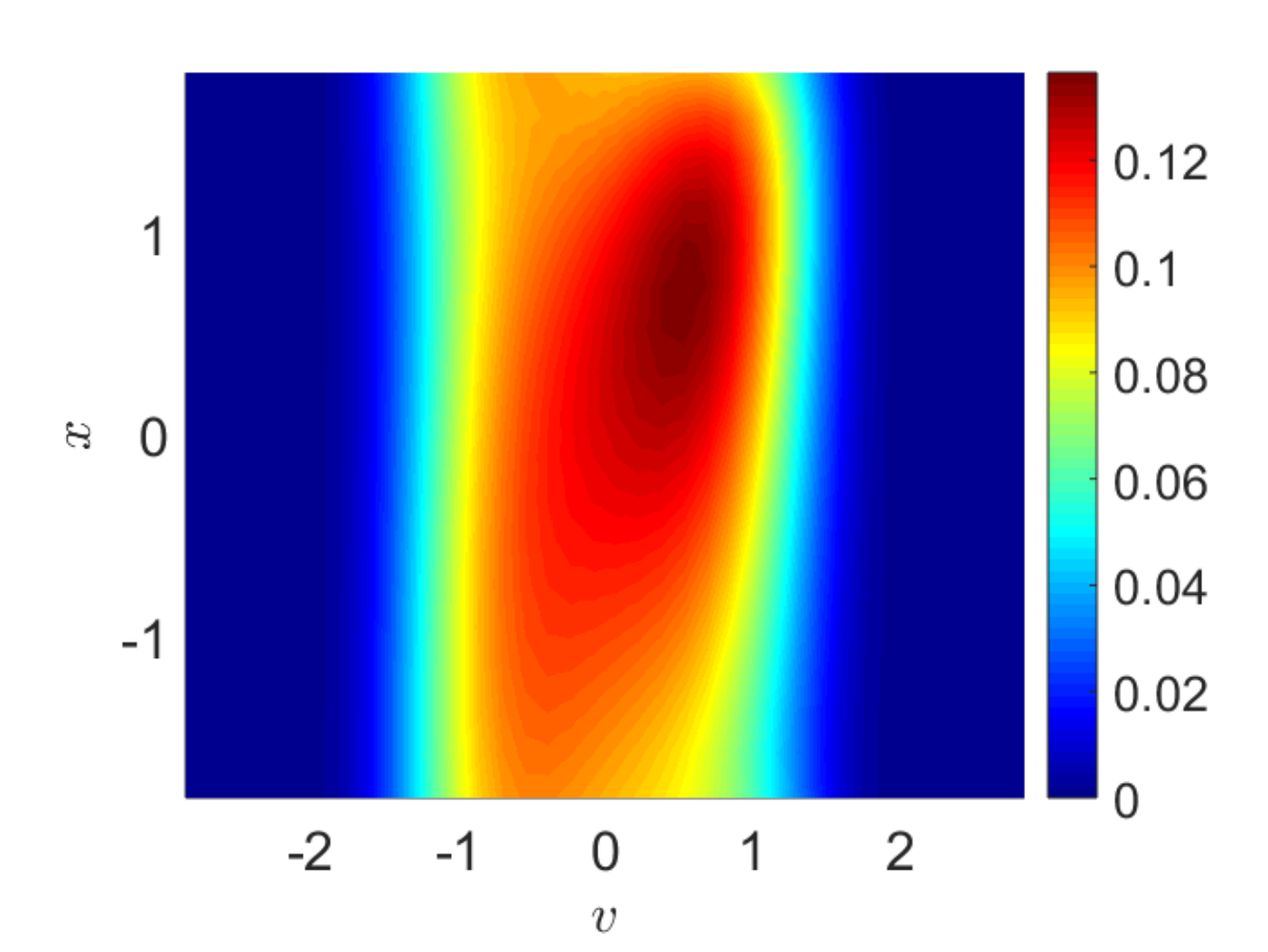}} \\
\subfigure[MCgPC, $t=0.5$]{
\includegraphics[scale=0.3]{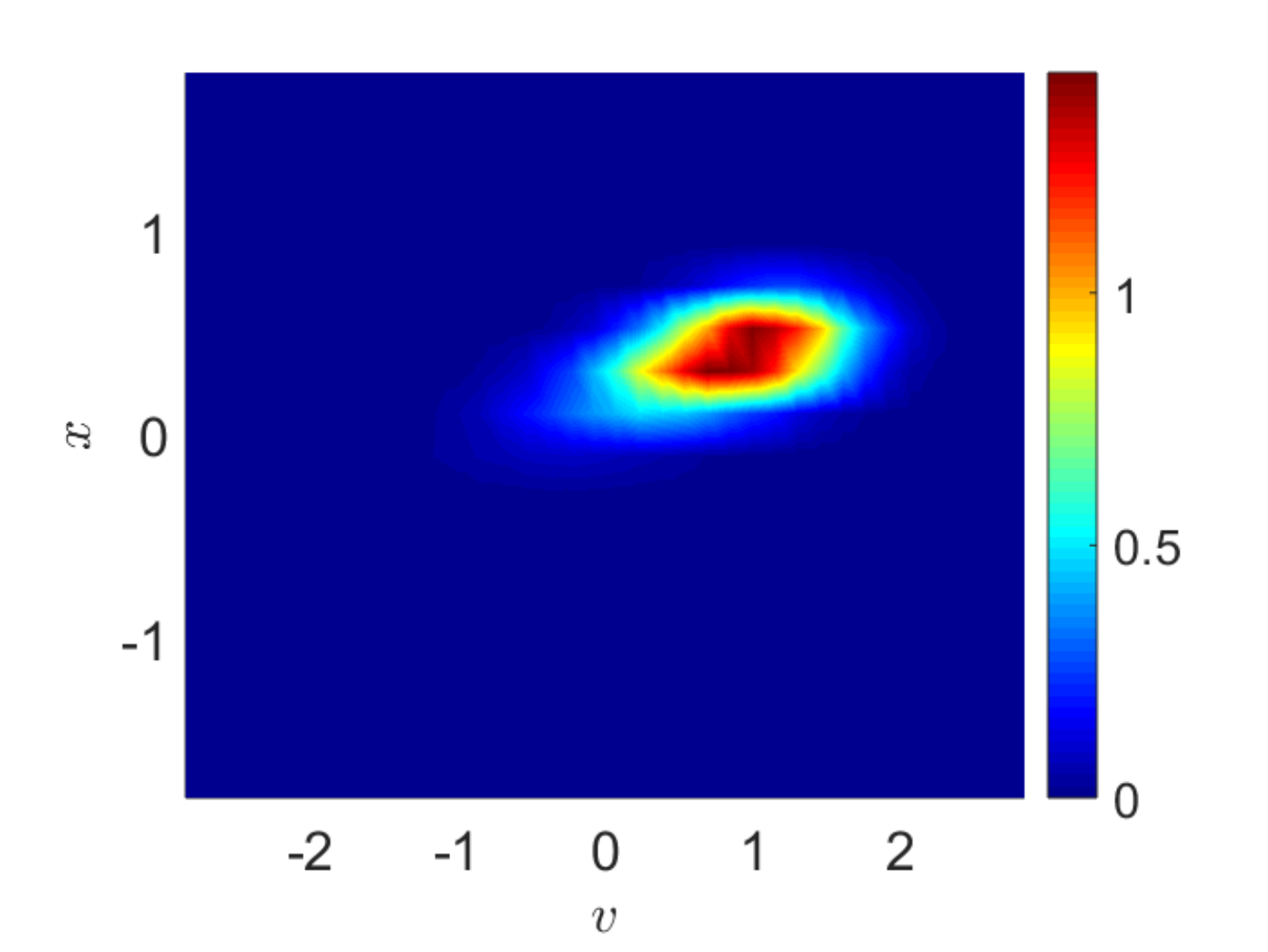}}
\subfigure[MCgPC, $t=1.0$]{
\includegraphics[scale=0.3]{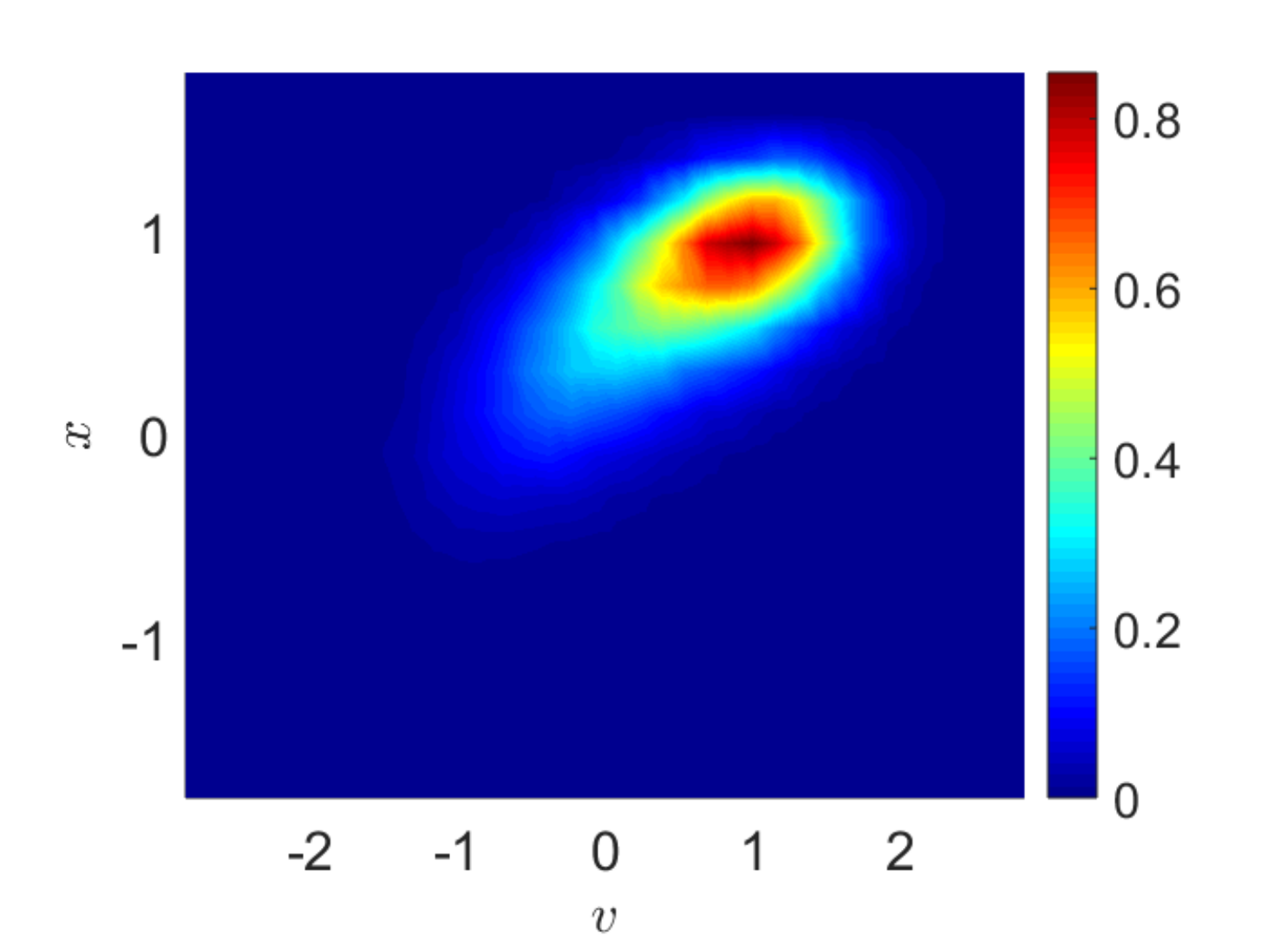}}
\subfigure[MCgPC, $t=5.0$]{
\includegraphics[scale=0.3]{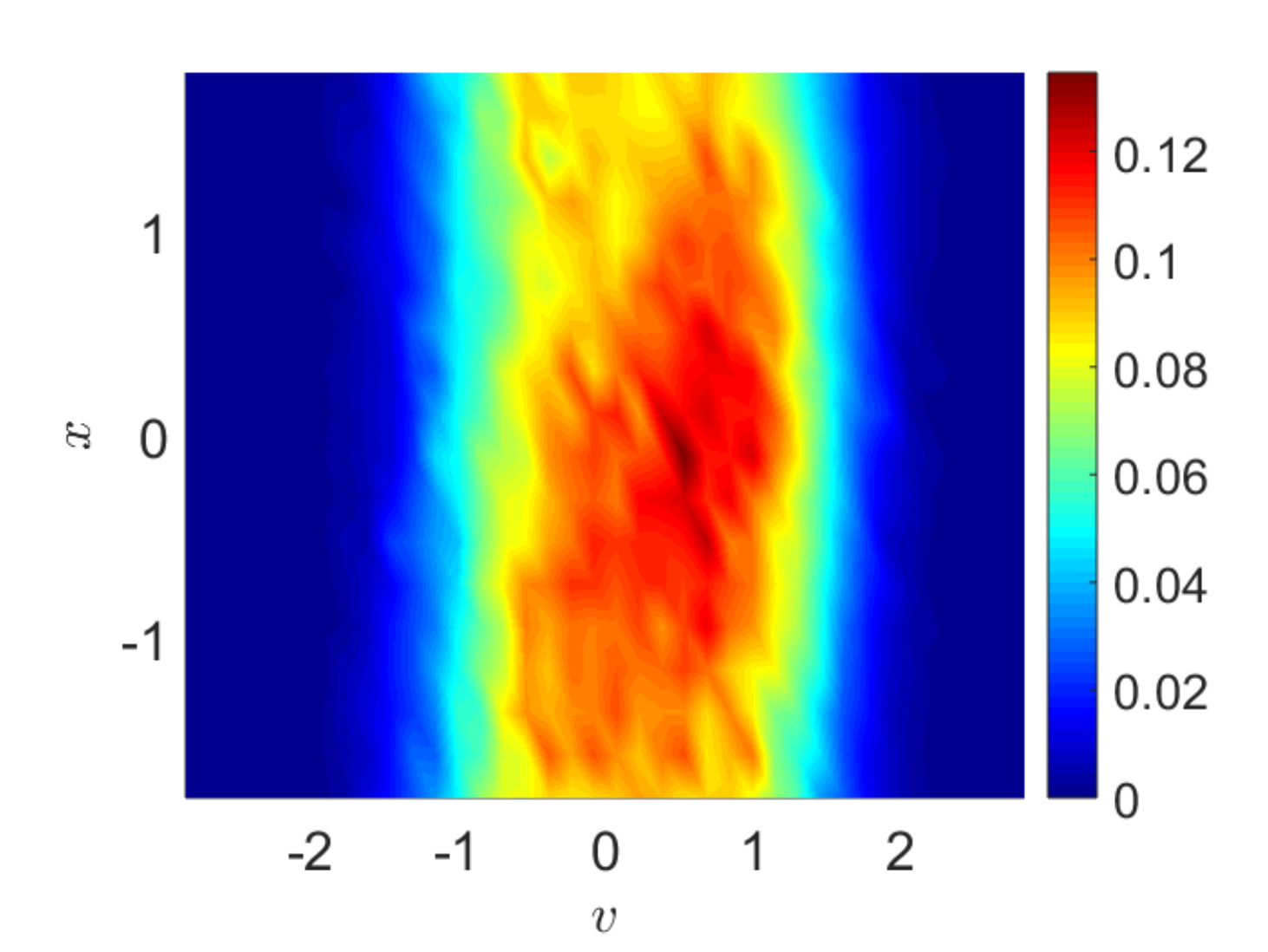}}
\caption{\textbf{Test 2A}. The same as Figure \ref{fig:inhom_loc_D02} but with $\bar D = 0.8$. }
\label{fig:inhom_loc_D08}
\end{figure}

\begin{figure}
\centering
\subfigure[$\bar D = 0.2$]{
\includegraphics[scale=0.45]{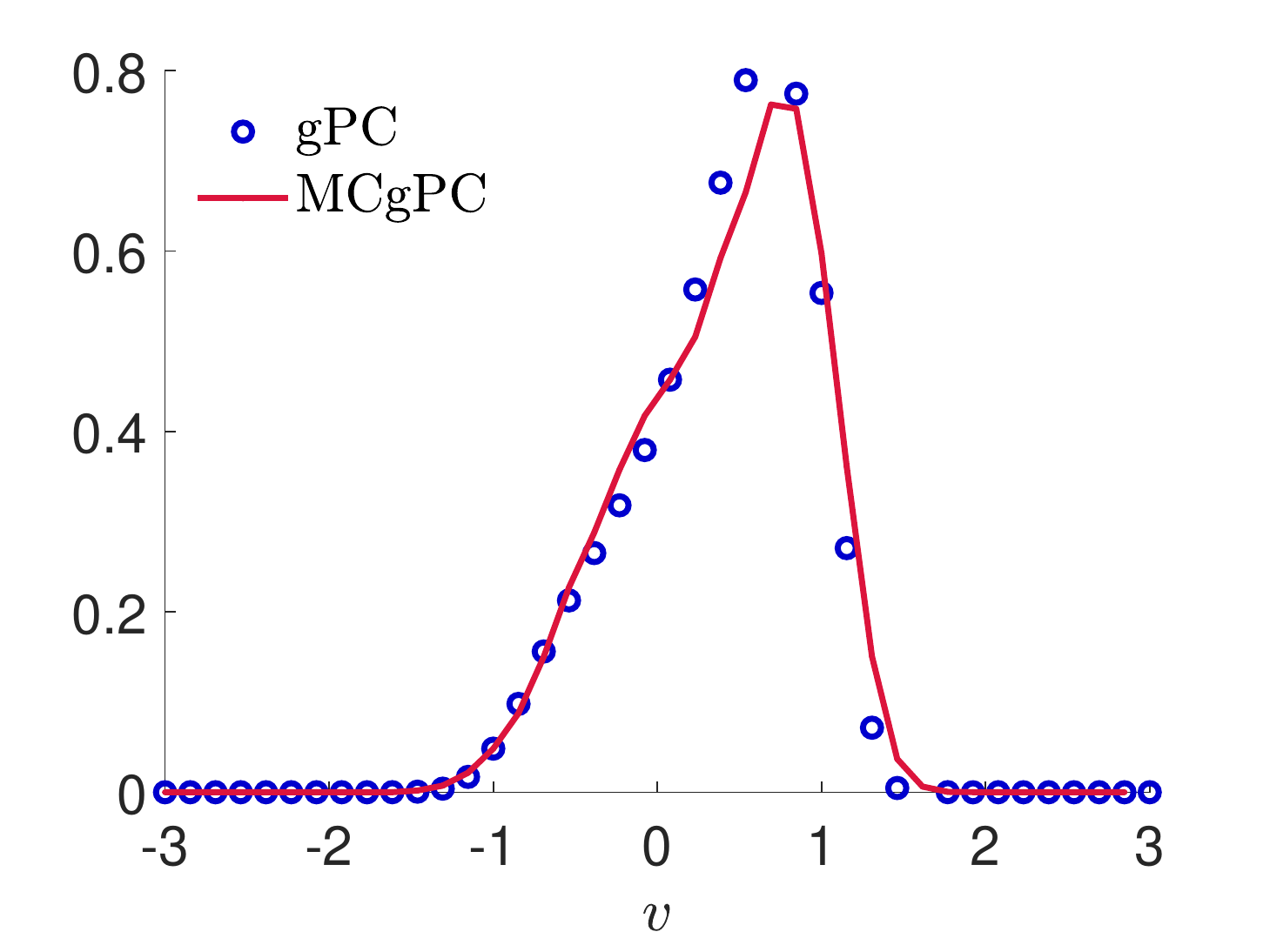}}
\subfigure[$\bar D = 0.8$]{
\includegraphics[scale=0.45]{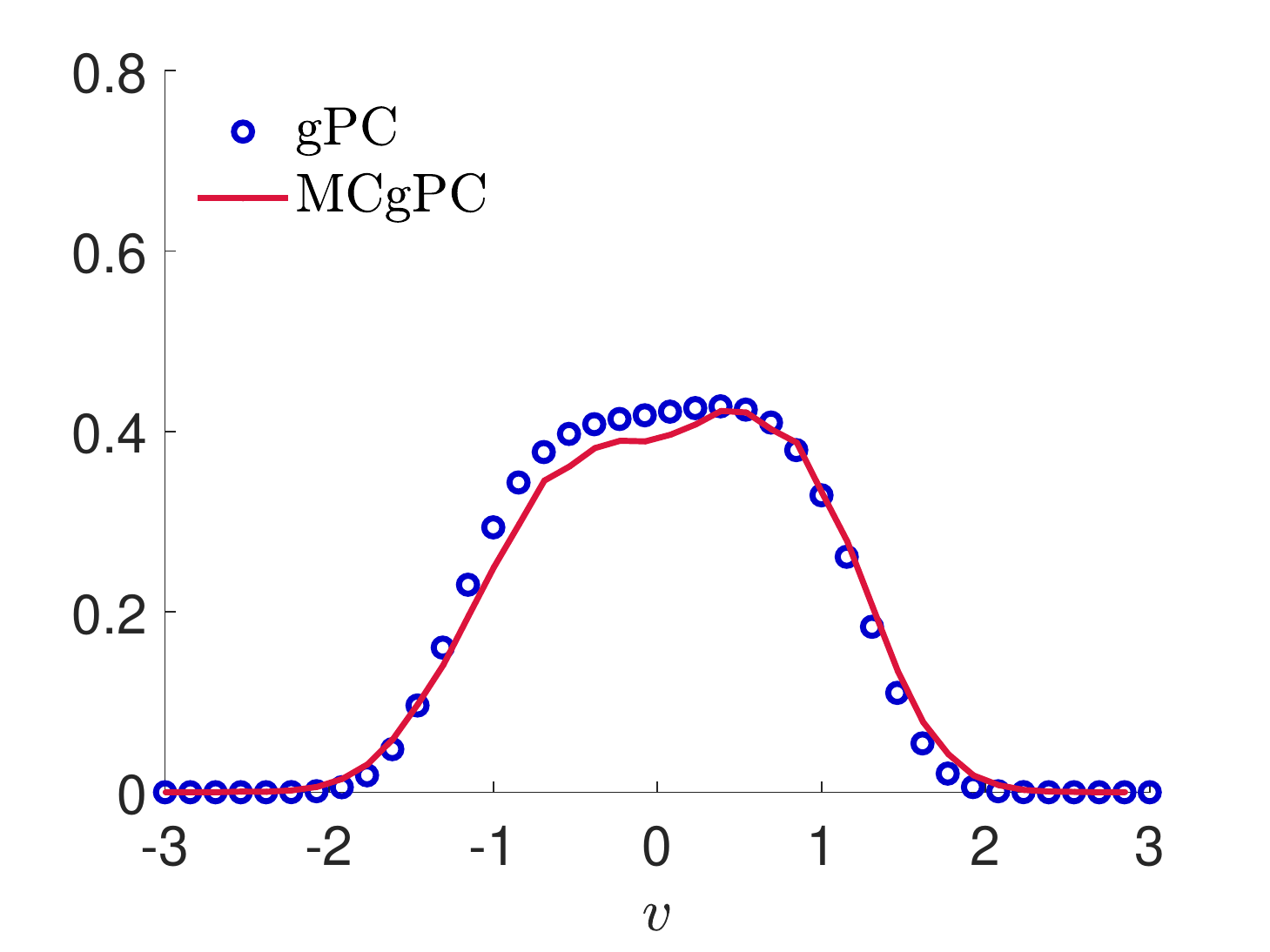}}
\caption{\textbf{Test 2A}. Marginal distributions $\int_{\RR^{d_x}}\mathbb E[f(\theta,x,v,t)]dx$, at time $t=5$, obtained from the numerical solutions of the inhomogeneous VFP with localized interactions and uncertain diffusion parameter reported in Figures \ref{fig:inhom_loc_D02}-\ref{fig:inhom_loc_D08}. }
\label{fig:kkk}
\end{figure}

\subsubsection{Test 2B. Cucker-Smale type interactions}

In this section we consider the evolution of statistical quantities obtained from \eqref{eq:MF_general} with space dependent interaction kernel of the Cucker-Smale (CS) type
\be\label{eq:CS_kernel}
P(x,x_*) = \dfrac{1}{(1 + |x-x_*|^2)^\gamma}, \qquad \gamma = 0.1. 
\ee
This choice of interaction has been introduced in \cite{CS} and received a lot of attention in the recent literature on kinetic models for emergent behavior. Without intending to review the huge related literature we mention \cite{CFRT,CFTV,HT} for an introduction on the topic. Recent efforts on CS dynamics in the UQ setting are \cite{APZa,HJ}. As we already mentioned, at the present times there are no analytical results regarding possible phase transitions for this highly nonlinear kernel in the VFP setting. The determination of sharp estimates on the possible phase transitions is an open problem which goes beyond the purpose of the present manuscript. Hence, to complete the overview on the introduced methods we give a numerical insight on the features of the resulting dynamics. 

 In order to observe the action of uncertainty in the system we consider a stochastic diffusion parameter $D(\theta) = \bar D + \lambda \bar D\theta$, with $\lambda = 10^{-1}$ and $\theta\sim \mathcal U([-1,1])$ in the cases $\bar D = 0.2$ and $\bar D= 0.8$. In Figure \ref{fig:CS_02}-\ref{fig:CS_08} we report the expected dynamics of the VFP model \eqref{eq:MF_general} with CS interactions and stochastic diffusion obtained through the MCgPC scheme. Periodic boundary conditions have been taken into account. In particular, a sample of $N = 10^5$ particles is considered for the  reconstruction of the expected density in the interval $[-2,2]\times[-3,3]$ with $N_x = 20$, $N_v = 40$ gridpoints. In the random space we considered $M=3$ modes for the gPC expansion. We can observe how the expected density moves from an ordered state for small values of diffusion to a chaotic isotropic for sufficiently big values of diffusion. We postpone to later works a more detailed insight on these phenomena for general interaction kernels. 

\begin{figure}
\subfigure[$t=0$]{
\includegraphics[scale=0.3]{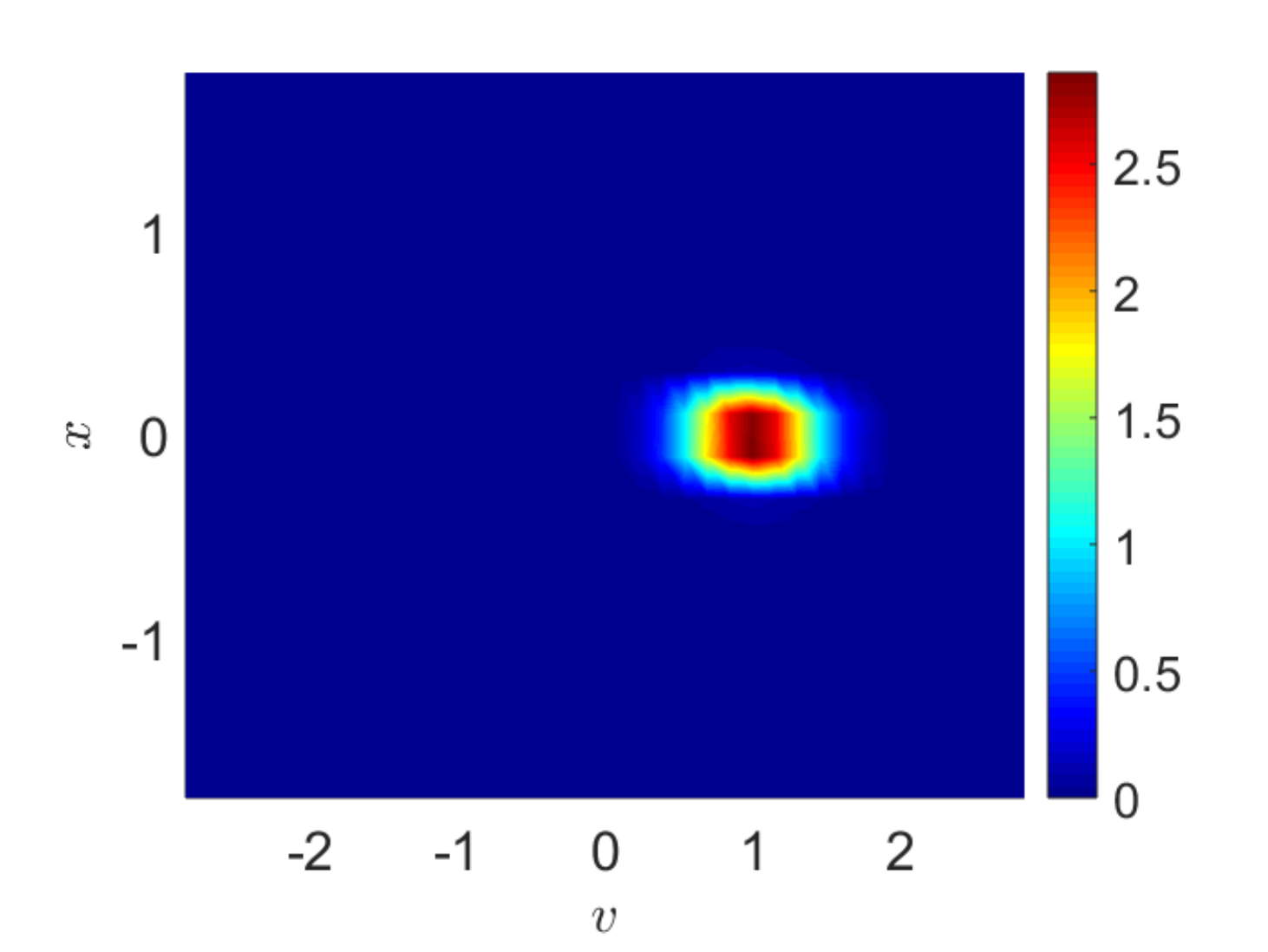}}
\subfigure[$t=1$]{
\includegraphics[scale=0.3]{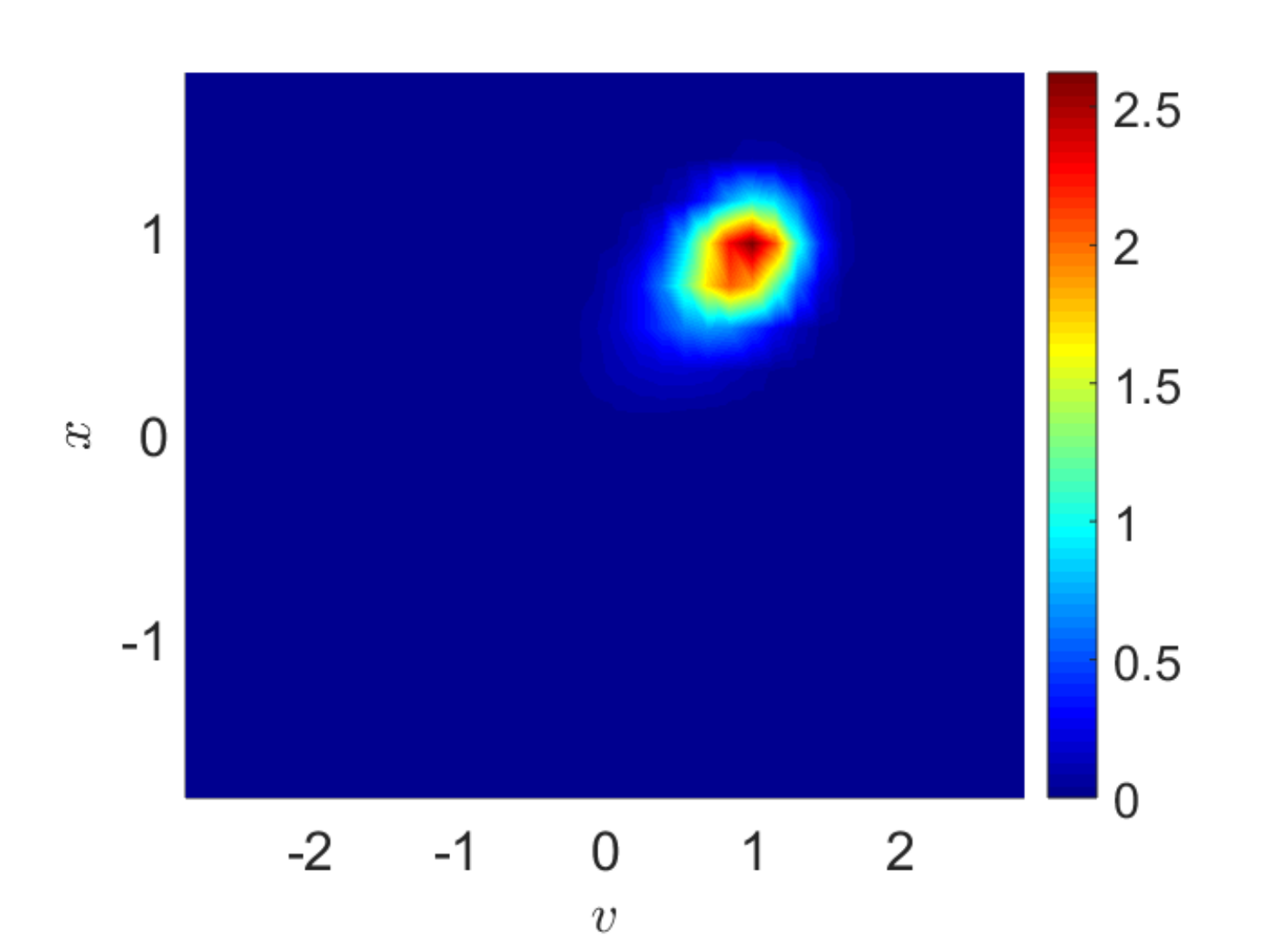}}
\subfigure[$t=2$]{
\includegraphics[scale=0.3]{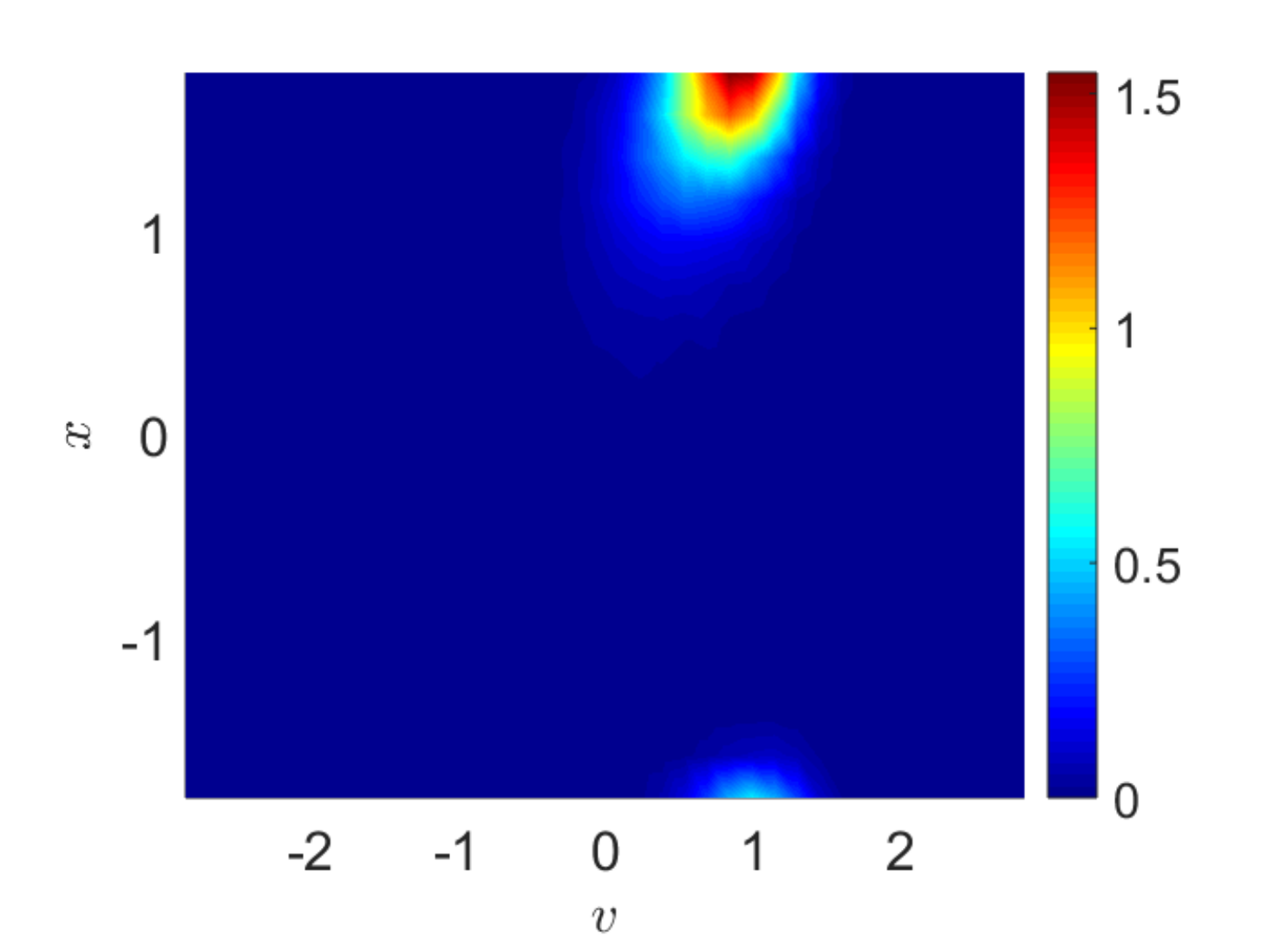}}\\
\subfigure[$t=3$]{
\includegraphics[scale=0.3]{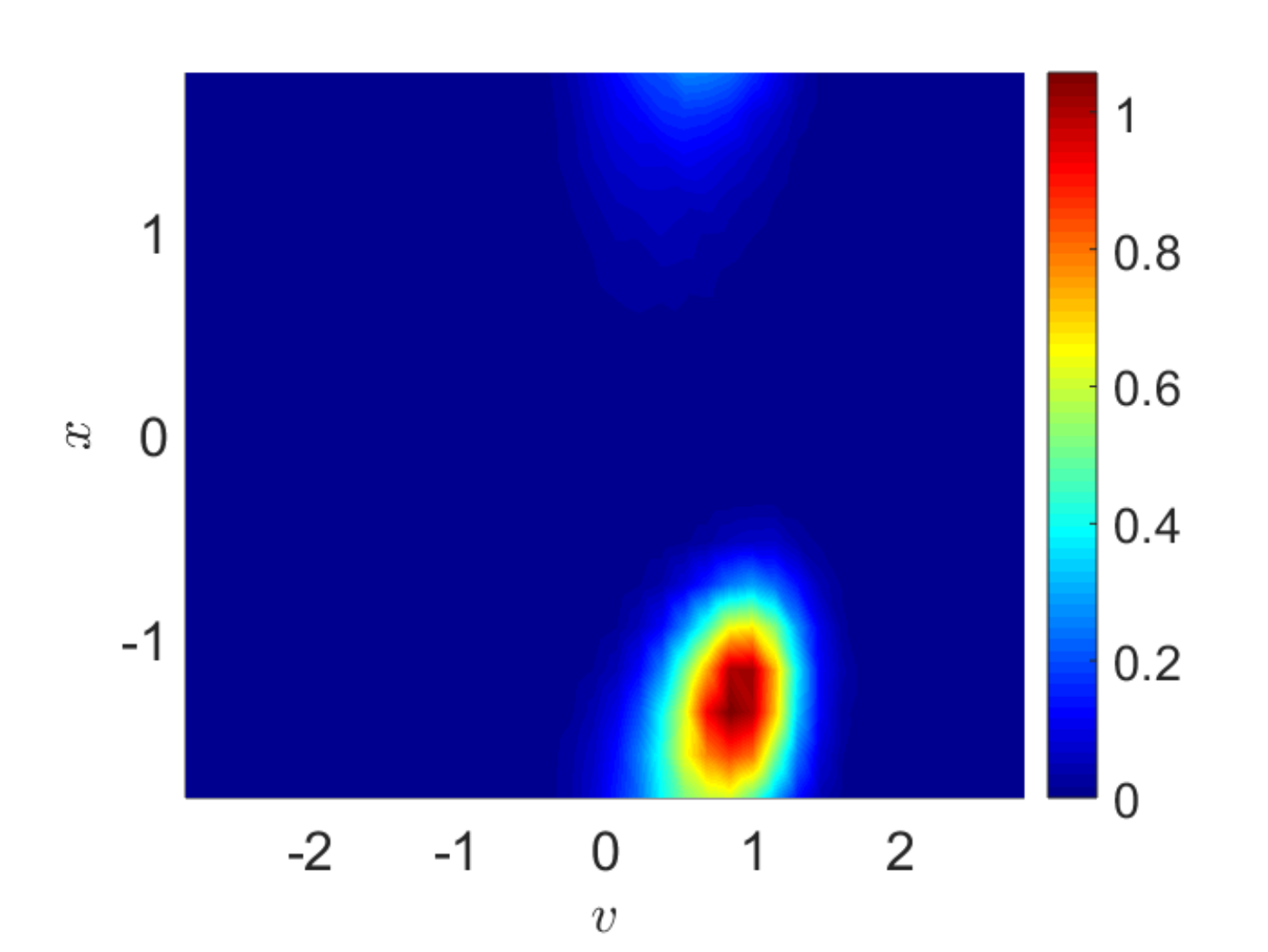}}
\subfigure[$t=4$]{
\includegraphics[scale=0.3]{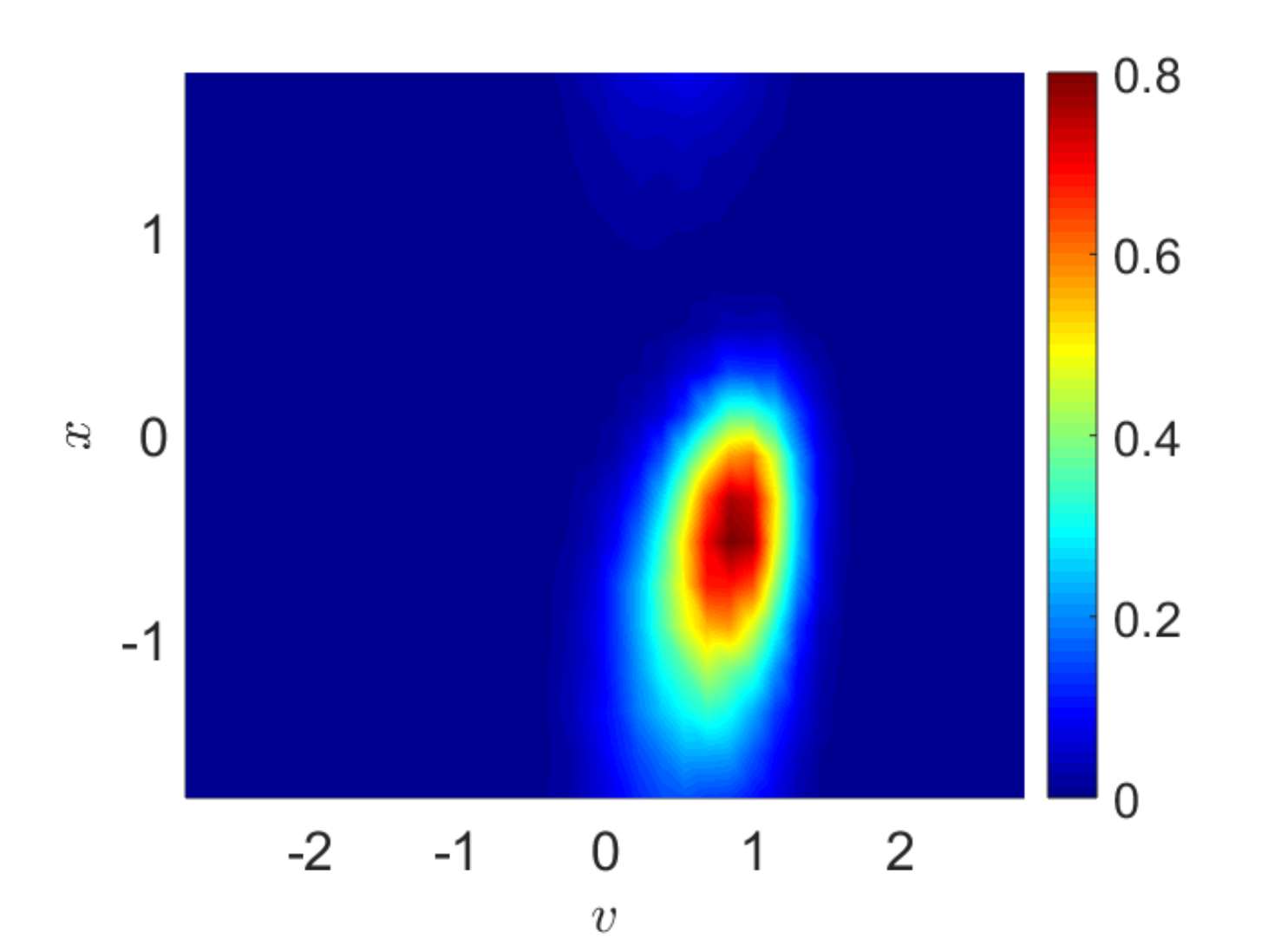}}
\subfigure[$t=5$]{
\includegraphics[scale=0.3]{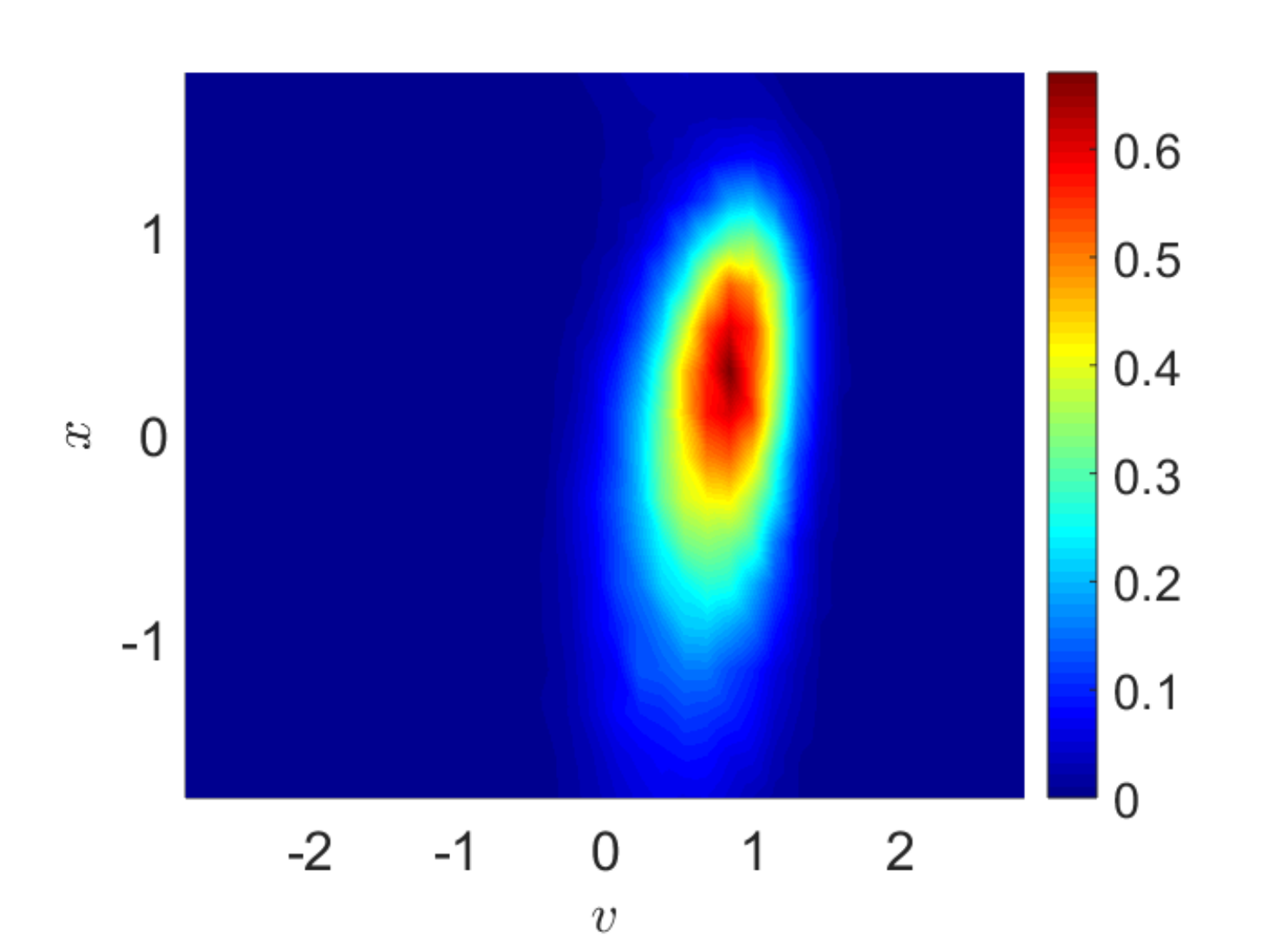}}
\caption{\textbf{Test 2B}. Evolution of the expected density for the VFP equation with CS interaction kernel \eqref{eq:CS_kernel} and with uncertain diffusion parameter in the time interval $[0,5]$. The uncertain diffusion parameter is $D(\theta) = \bar D + \lambda \bar D\theta$ with $\bar D = 0.2$ and $\lambda = 0.1$. Dynamics obtained through MCgPC scheme with $N= 10^5$ particles and fast evaluation of interaction with $S=10$. The gPC expansion is order $M=4$ andthe density is reconstructed in the phase space $(x,v)\in[-2,2]\times [-3,3]$ with $N_v=40$, $N_x = 20$. }
\label{fig:CS_02}
\end{figure}

\begin{figure}
\subfigure[$t=0$]{
\includegraphics[scale=0.3]{figure/inhomogeneous/CS/inhom_MCgPC_t0}}
\subfigure[$t=1$]{
\includegraphics[scale=0.3]{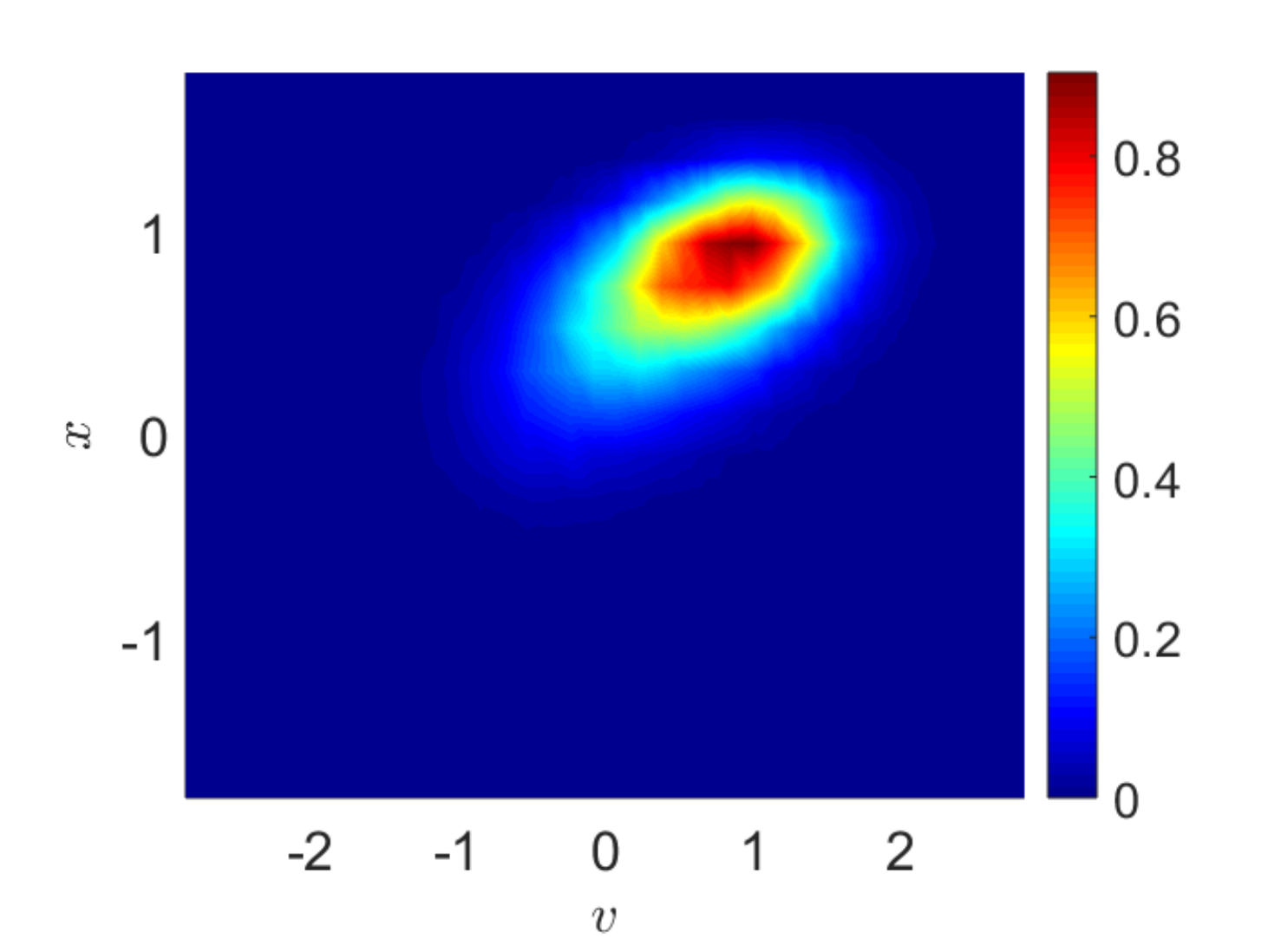}}
\subfigure[$t=2$]{
\includegraphics[scale=0.3]{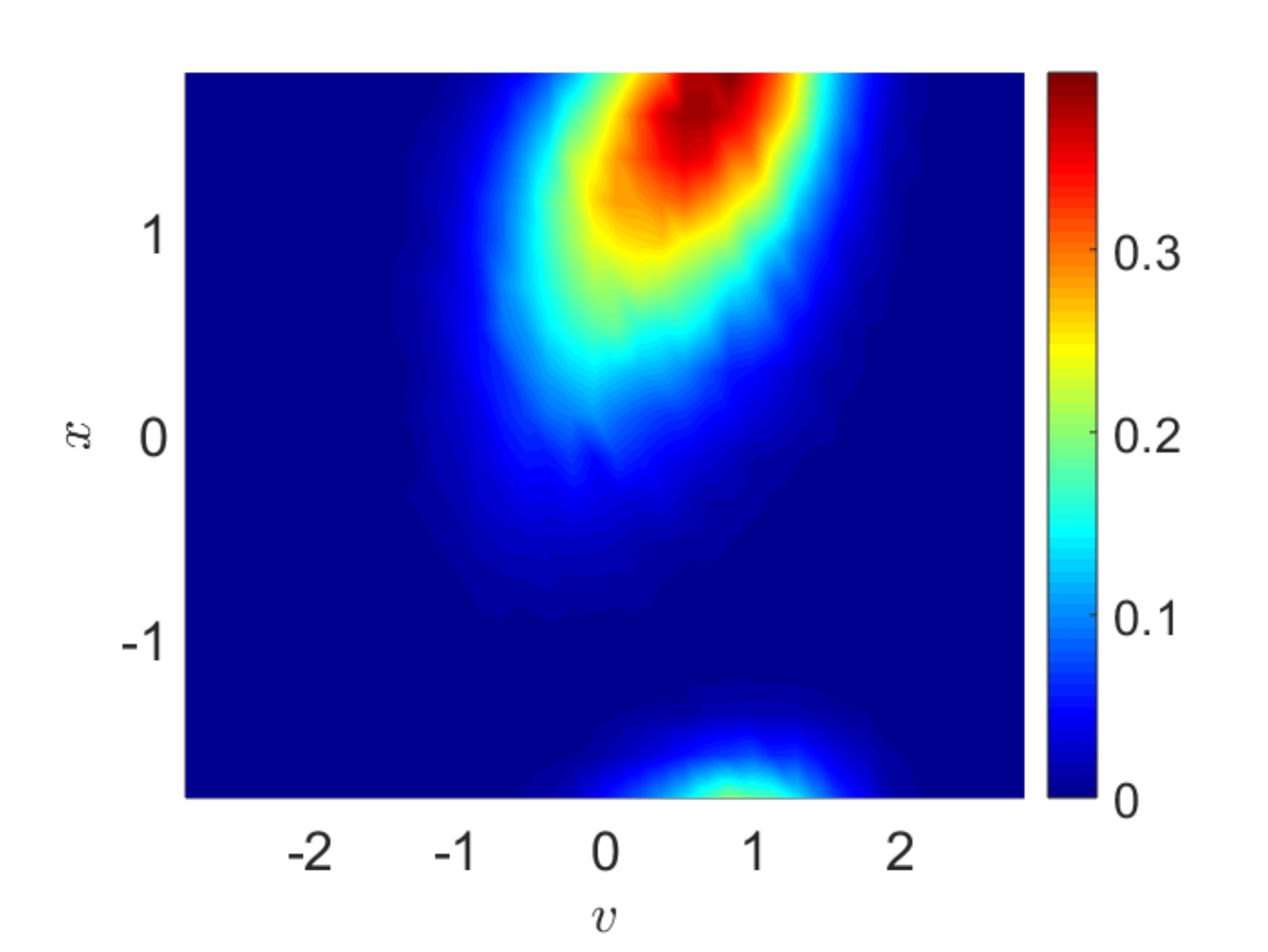}}\\
\subfigure[$t=3$]{
\includegraphics[scale=0.3]{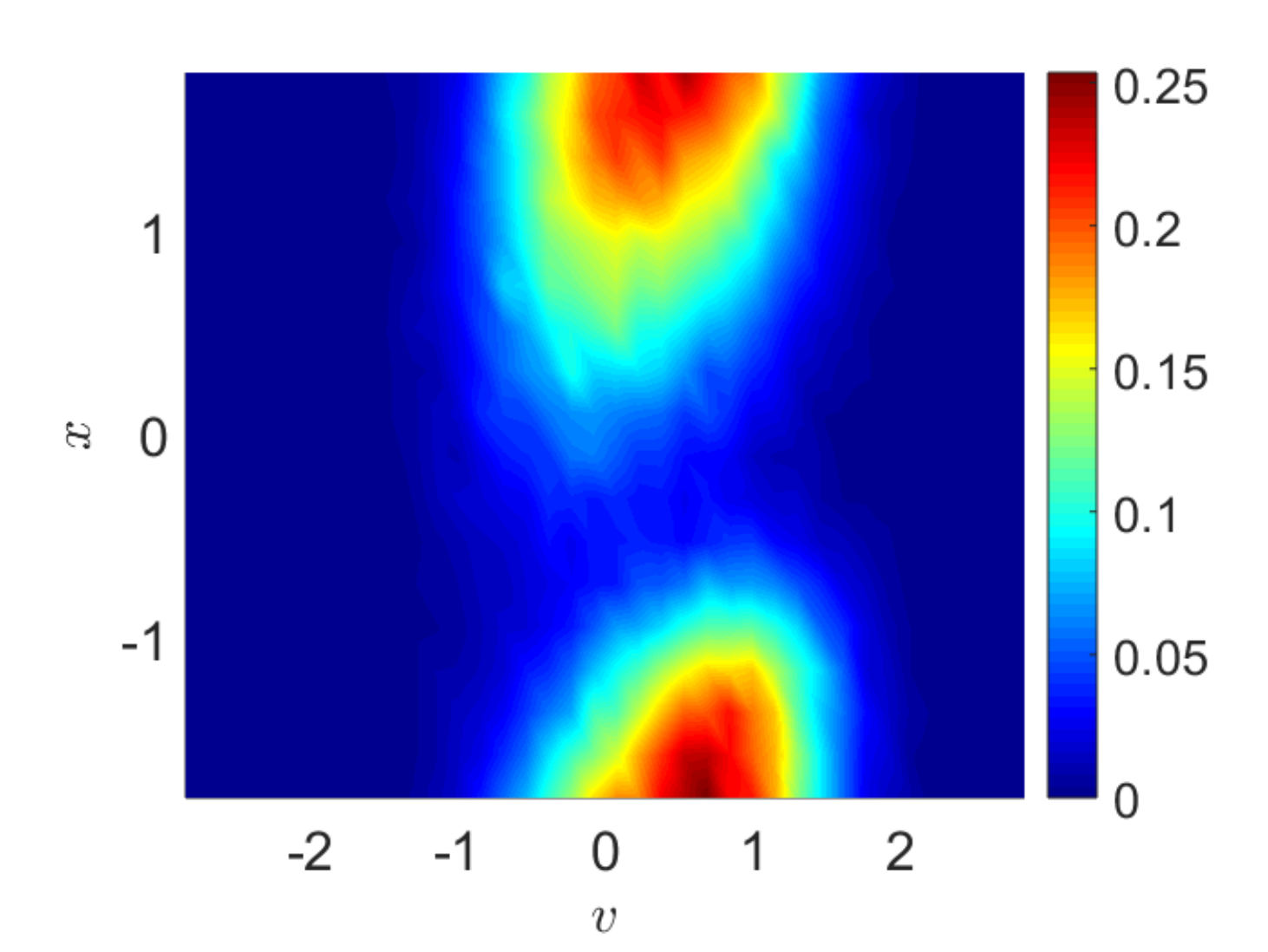}}
\subfigure[$t=4$]{
\includegraphics[scale=0.3]{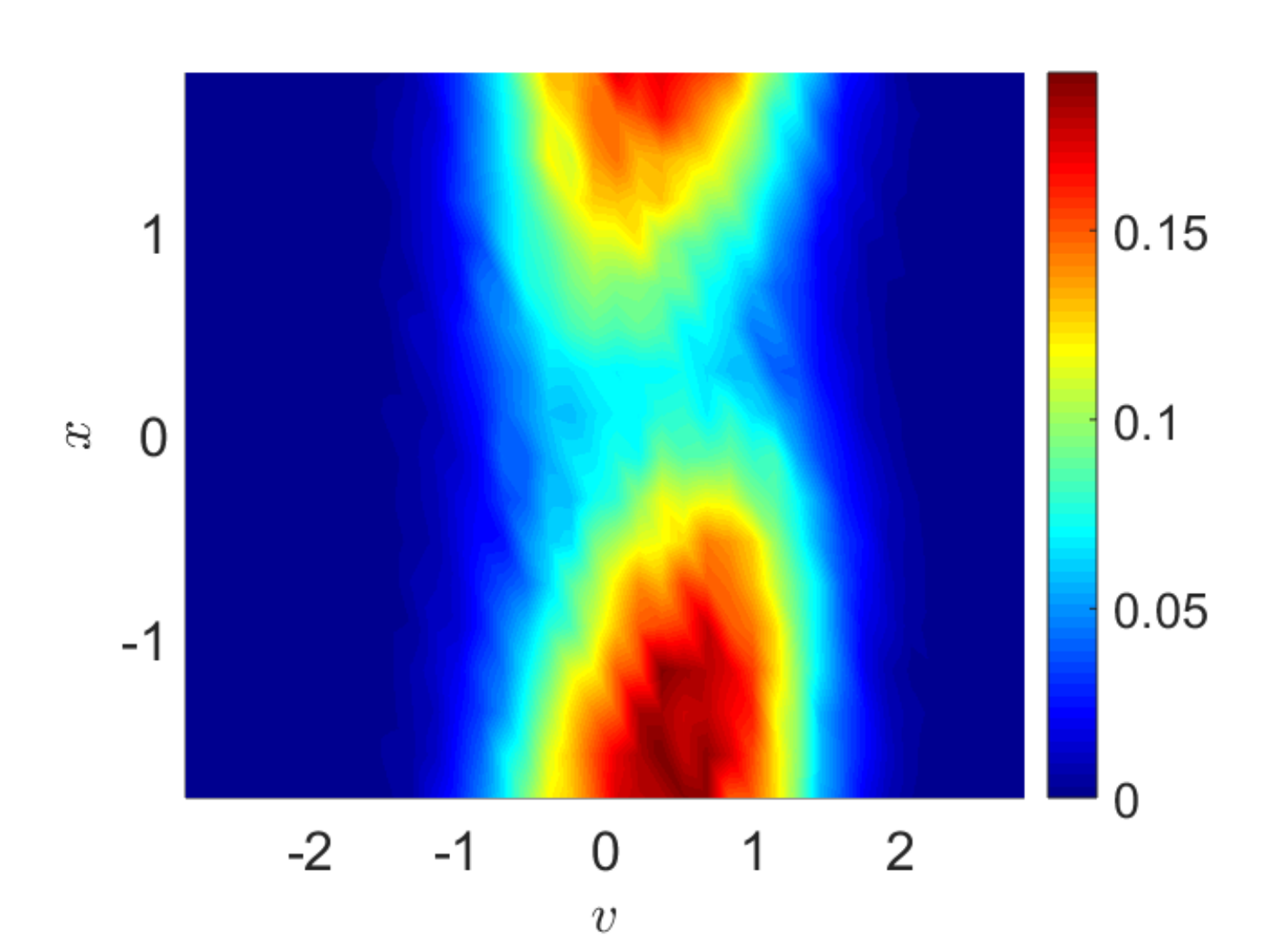}}
\subfigure[$t=5$]{
\includegraphics[scale=0.3]{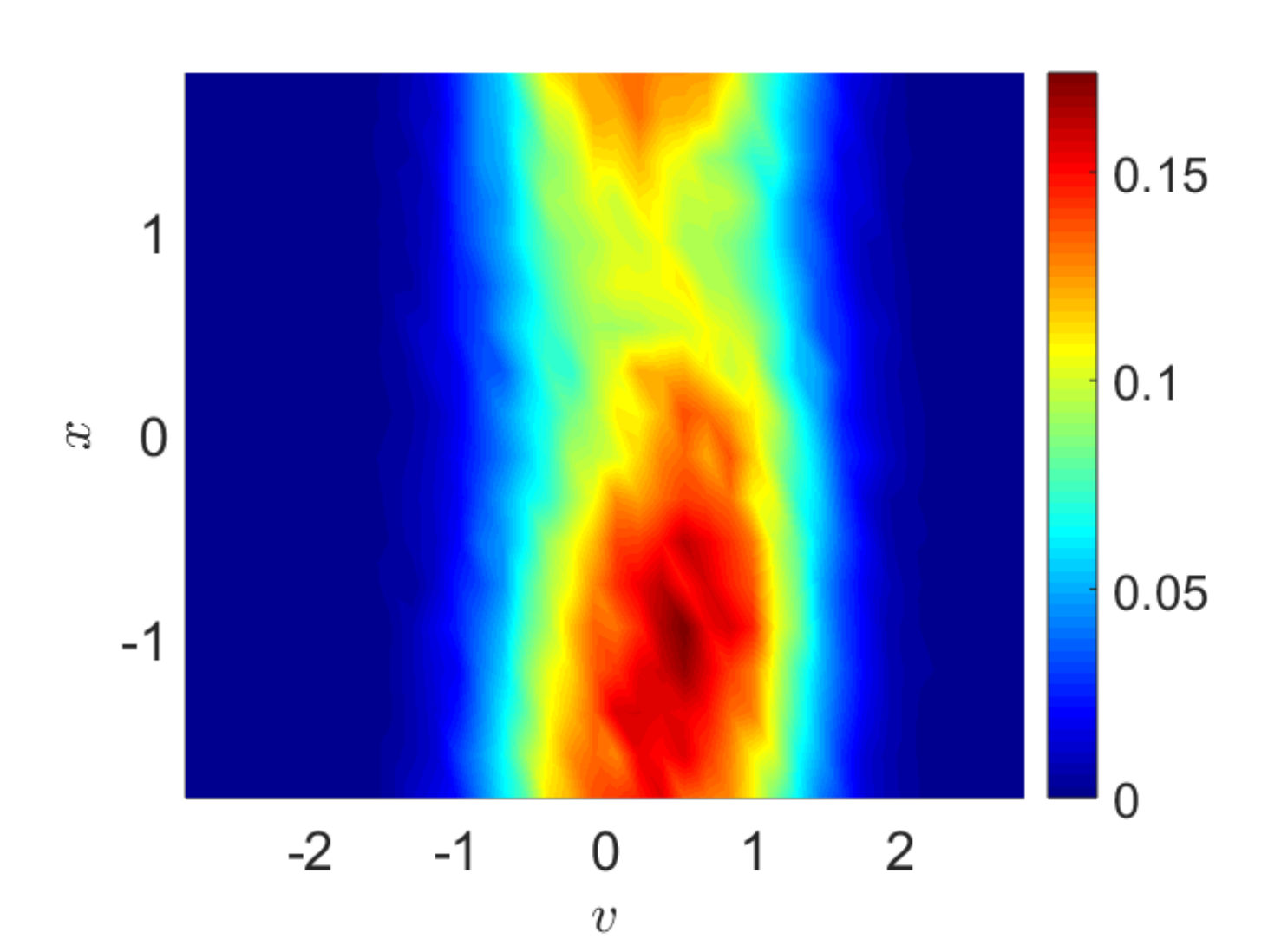}}
\caption{\textbf{Test 2B}. The same as Figure \ref{fig:CS_02} but with $\bar D = 0.8$. }
\label{fig:CS_08}
\end{figure}

 \section*{Conclusion}
The introduction of uncertainties in swarming dynamics is of paramount importance for the large-scale description of realistic phenomena. We concentrated in this manuscript on the case of Vlasov-Fokker-Planck equations for emerging collective behaviour with phase transitions depending on the strength of self-propelling and noise terms. We extended to these models the construction of nonnegative gPC schemes for kinetic-type equations. The proposed approach takes advantage of a Monte Carlo approximation of the kinetic equation in phase space which is coupled with a stochastic Galerkin gPC expansion. Several numerical tests were proposed both in the homogeneous and inhomogeneous settings, proving the effectiveness of the approach and shedding light on the action of uncertainties in terms of the stability of the phase transitions, We generically observe that they lead to a smoothing of the sharp transition values. Several extensions of the present work are possible from both the analytical and numerical viewpoints and in connection with more general kernels in the collision/interaction dynamics.   
 
%%%% Acknowledgments %%%%%%%%
\section*{Acknowledgments}
JAC was partially supported by the EPSRC grant number EP/P031587/1. 
MZ is member of the National Group on Mathematical Physics (GNFM) of INdAM (Istituto Italiano di Alta Matematica), Italy. This work has been written within the activities of the Excellence Project \textrm{CUP: E11G18000350001} of the Department of Mathematical Sciences "G. L. Lagrange" of Politecnico di Torino funded by MIUR (Italian Ministry for Education, University and Research).
% partial support from the Excellence Project \textrm{CUP: E11G18000350001} of the Department of Mathematical Sciences ``G. L. Lagrange'', Politecnico di Torino, Italy. 

%%% Bibliography  %%%%%%%%%%


\begin{thebibliography}{99}

\bibitem{ABCD}
P. Aceves-S\'anchez, M. Bostan, J. A. Carrillo, P. Degond. Hydrodynamic limits for kinetic flocking models of Cucker-Smale type. \emph{Math. Bio. Eng.}, 16(6):7882--7910, 2019. 

\bibitem{AlbiPareschi13}
G. Albi, and L. Pareschi. Binary interaction algorithms for the simulation of flocking and swarming dynamics, \emph{Multiscale Model. Simul.}, 11: 1--29, 2013.

\bibitem{APZa}
G. Albi, L. Pareschi, and M. Zanella. Uncertainty quantification in control problems for flocking models, \emph{Math. Probl. Eng.}, vol. 2015, 14 pp., 2015.

\bibitem{BCH}
R.~Bailo, J.~A. Carrillo, and J.~Hu. Fully discrete positivity-preserving and energy-dissipative schemes for nonlinear nonlocal equations with a gradient flow structure. Preprint \texttt{arXiv:1811.11502}, 2018.

\bibitem{Ball_etal}
M. Ballerini, N. Cabibbo, R. Candelier, A. Cavagna, E. Cisbani, I. Giardina, A. Orlandi, G. Parisi, A. Procaccini, M. Viale, and V. Zdravkovic. Empirical investigation of starling flocks: a benchmark study in collective animal behavior. \emph{Animal Behavior}, 76(1): 201--215, 2008.

\bibitem{BCCD}
A. B. T. Barbaro, J. A. Ca\~{n}izo, J. A. Carrillo, and P. Degond. Phase transitions in a kinetic flocking models of Cucker-Smale type, \emph{Multiscale Model. Simul.}, 14(3): 1063--1088, 2016.

\bibitem{BD}
A. B. T. Barbaro, and P. Degond. Phase transition and diffusion among socially interacting self--propelled agents, \emph{Discrete Cont. Dyn. Syst. - B}, 19: 1249--1278, 2014.

\bibitem{BGP10}
L. Bertini, G. Giacomin, and K. Pakdaman. Dynamical aspects of mean field plane rotators and the {K}uramoto model, {\em J. Stat. Phys.}, 138(1): 270--290, 2010.

\bibitem{BCC}
F. Bolley, J. A. Ca\~{n}izo, and J. A. Carrillo. Stochastic mean-field limit: non-Lipschitz forces and swarming. \emph{Math. Models Meth. Appl. Sci.}, 21(11): 2179--2210, 2011.

\bibitem{BC} 
M. Bostan, and J.-A. Carrillo. Reduced fluid models for self-propelled populations, interacting through alignment, \emph{Math. Models Methods Appl. Sci.}, 27: 1255--1299, 2017.

\bibitem{BCM}
M. Burger, V. Capasso, D. Morale. On an aggregation model with long and short range interactions. \emph{Nonlinear Anal. Real World Appl.}, 8(3): 939--958, 2007.

\bibitem{Caflisch}
R. E. Caflisch. Monte Carlo and Quasi Monte Carlo methods, \emph{Acta Numerica}, 7: 1--49, 1998.

\bibitem{CCH}
J. A. Carrillo, A. Chertock, Y. Huang. A finite-volume method for nonlinear nonlocal equations with a gradient flow structure. \emph{Commun. Comput. Phys.}, 17(1): 233--258, 2015.

\bibitem{CCP}
J. A. Carrillo, Y.-P. Choi, and L. Pareschi. Structure preserving schemes for the continuum Kuramoto model: phase transitions. \emph{J. Comp. Phys.}, 376: 365--389, 2019.

\bibitem{CFRT}
J. A. Carrillo, M. Fornasier, J. Rosado, and G. Toscani. Asymptotic flocking dynamics for the kinetic Cucker--Smale model, \emph{SIAM J. Math. Anal.}, 42(1): 218--236, 2010.

\bibitem{CFTV}
J. A. Carrillo, M. Fornasier, G. Toscani, and F. Vecil. Particle, kinetic and hydrodynamic models of swarming, In \emph{Mathematical Modeling of Collective Behavior in Socio--Economic and Life Sciences}, Eds. G. Naldi, L. Pareschi and G. Toscani, Modeling and Simulation in Science, Engineering and Technology, Birkh{\"a}user, Boston, 2010. 

\bibitem{CGPS}
J. A. Carrillo, R. S. Gvalani, G. A. Pavliotis, A. Schlichting. Long-time behavior and phase transitions for the McKean-Vlasov equation on the torus. Preprint \texttt{arXiv:1806.01719}, 2018. 

\bibitem{CPZ}
J. A. Carrillo, L. Pareschi, and M. Zanella. Particle based gPC methods for mean--field models of swarming with uncertainty, \emph{Comm. Comput. Phys.}, 25(2): 508--531, 2019.

\bibitem{CP}
L. Chayes, and V. Panferov. The {M}c{K}ean-{V}lasov equation in finite volume, \emph{J. Stat. Phys.}, 138: 351--380, 2010.

\bibitem{CS}
F. Cucker, and S. Smale. Emergent behavior in flocks, \emph{IEEE Trans. Automat. Contr.}, 52(5): 852--862, 2007.

\bibitem{DP15}
G. Dimarco, and L. Pareschi. Numerical methods for kinetic equations, \emph{Acta Numerica}, 23: 369--520, 2014.

\bibitem{DPZ}
G. Dimarco, L. Pareschi, and M. Zanella. Uncertainty quantification for kinetic models in socio--economic and life sciences. In \emph{Uncertainty Quantification for Hyperbolic and Kinetic Equations}, eds.: S. Jin, and L. Pareschi, SEMA--SIMAI Springer Series, vol. 14, pp. 151--191, 2017.

\bibitem{DFL}
P. Degond, A. Frouvelle, and J.-G. Liu. Macroscopic limits and phase transitions in a system of self-propelled particles, \emph{J. Nonlin. Sci.}, 23(3): 427--456, 2013.

\bibitem{DM}
P. Degond, and S. Motsch. Continuum limit of self-driven particles with orientation interaction. \emph{Math. Mod. Meth. Appl. Sci.}, 18.supp01: 1193--1215, 2008.

\bibitem{DCBC}
M. R. D'Orsogna, Y. L. Chuang, A. L. Bertozzi, and L. S. Chayes. Self-propelled particles with soft-core interactions: patterns, stability, and collapse. \emph{Phys. Rev. Lett.}, 96: 104302, 2006.

\bibitem{DFT}
R. Duan, M. Fornasier, and G. Toscani. A kinetic flocking model with diffusion. \emph{Commun. Math. Phys.}, 300: 95--145, 2010.

\bibitem{Fun}
D. Funaro. \emph{Polynomial Approximation of Differential Equations}, Springer-Verlag, Berlin, 1992.

\bibitem{GPY}
J. Garnier, G. Papanicolaou, and T.-W. Yang. Consensus convergence with stochastic effects. \emph{Vietnam J. Math.}, 45: 51--75, 2017.

\bibitem{GP}
S. N. Gomes, and G. A. Pavliotis. Mean field limits for interacting diffusions in a two-scale potential, \emph{J. Nonlinear Sci.}, 28(3): 905--941, 2018.

\bibitem{HaJin}
S.-Y. Ha, and S. Jin. Local sensitivity analysis for the Cucker-Smale model with random inputs, \emph{Kinet. Relat. Mod.}, 11(4): 859--889, 2018.

\bibitem{HJJ}
S.-Y. Ha, S. Jin, and J. Jung. Local sensitivity analysis for the Kuramoto model with random inputs in a large coupling regime, \emph{Netw. Heterog. Media}, to appear. 

\bibitem{HT}
S.-Y. Ha, and E. Tadmor. From particle to kinetic and hydrodynamic descriptions of flocking, \emph{Kinet. Relat. Mod.}, 1(3): 415--435, 2008.

\bibitem{HE}
R. W. Hockney, J. K. Eastwood. \emph{Computer Simulation using Particles}, McGraw Hill International Book Co., 1981. 

\bibitem{HJ}
J. Hu, and S. Jin. Uncertainty Quantification for Kinetic Equations. In \emph{Uncertainty Quantification for Hyperbolic and Kinetic Equations}, eds.: S. Jin, and L. Pareschi, SEMA-SIMAI Springer Series, vol. 14, pp. 193--229, 2017.

\bibitem{HJX}
 J. Hu, S. Jin, and D. Xiu. A stochastic Galerkin method for Hamilton-Jacobi equations with uncertainty, \emph{SIAM J. Sci. Comp.}, 37: A2246-A2269, 2015.
 
 \bibitem{JLL}
 S. Jin, L. Li, and J.-G. Liu. Random batch methods (RMB) for interacting particle systems.  \emph{J. Comput. Phys.}, 400: Article number 108877, 2020. 
  
\bibitem{Kur81}
Y. Kuramoto. Rhythms and turbulence in populations of chemical oscillators, {\em Phys. A}, 106(1-2):128--143, 1981. Statphys 14 (Proc. Fourteenth Internat. Conf. Thermodynamics and Statist. Mech., Univ. Alberta, Edmonton, Alta., 1980).

\bibitem{MT}
S. Motsch, and E. Tadmor. Heterophilious dynamics enhances consensus, \emph{SIAM Rev.} 56(4) : 577--621, 2014. 

\bibitem{PT}
L. Pareschi, and G. Toscani. \emph{Interacting Multiagent Systems: Kinetic Equations and Monte Carlo Methods}, Oxford University Press, 2013. 

\bibitem{PZ1}
L. Pareschi, and M. Zanella. Structure--preserving schemes for nonlinear Fokker--Planck equations and applications. \emph{J. Sci. Comput.},  74(3): 1575--1600, 2018.

\bibitem{PZ2}
L. Pareschi, and M. Zanella. Structure preserving schemes for mean-field equations of collective behavior. In Westdickenberg M. Klingenberg C., editor, \emph{Theory, Numerics and Applications of Hyperbolic Problems II. HYP 2016}, volume 237 of \emph{Springer Proceedings in Mathematics \& Statistics}, pages 405--421. Springer, Cham, 2018. 

\bibitem{SSK88}
H. Sakaguchi, S. Shinomoto, and Y. Kuramoto. Phase transitions and their bifurcation analysis in a large population of active rotators with mean-field coupling, {\em Progr. Theoret. Phys.}, 79: 600--607, 1988.

\bibitem{CWS}
C.-W. Shu. High order weighted essentially nonoscillatory schemes for convection dominated problems, \emph{SIAM Rev.}, 51(1): 82--126, 2009.

\bibitem{TZ}
A. Tosin, and M. Zanella. Boltzmann--type models with uncertain binary interactions, \emph{Commun. Math. Sci.}, 16(4): 963--985, 2018.

\bibitem{TZ2}
A. Tosin, and M. Zanella. Uncertainty damping in kinetic traffic models by driver-assist controls. Preprint \texttt{arXiv:1904.00257}, 2019. 

\bibitem{X}
D. Xiu. \emph{Numerical Methods for Stochastic Computations}, Princeton University Press, 2010.

\bibitem{XK}
D. Xiu, and G. E. Karniadakis. The Wiener-Askey polynomial chaos for stochastic differential equations, \emph{SIAM J. Sci. Comput.}, 24(2): 614--644, 2002.

\bibitem{ZL}
X. Zhong, and Q. Li. Galerkin methods for stationary radiative transfer equations with uncertain coefficients, \emph{J. Sci. Comput.}, 76(2): 1105--1126, 2018. 

\bibitem{ZJ} 
Y. Zhu, and S. Jin. The Vlasov-Poisson-Fokker-Planck system with uncertainty and a one-dimensional asymptotic-preserving method , \emph{Multiscale Model. Simul.},  15(4): 1502–1529, 2017.

\bibitem{Z}
M. Zanella. Structure preserving stochastic Galerkin methods for Fokker-Planck equations with background interactions. \emph{Math.Comp. Simul.}, 168:28-47, 2020. 

\end{thebibliography}
\end{document}